\documentclass[reqno]{amsart}      
\usepackage{amssymb,amsmath,graphicx}
\newtoks\prt 
  \numberwithin{equation}{section}
\newtheorem{thm}{Theorem}[section]

\newtheorem{ques}[thm]{Question} 
\newtheorem{lemma}[thm]{Lemma} 
\newtheorem{prop}[thm]{Proposition} 
\newtheorem{cor}[thm]{Corollary} 
\newtheorem{example}[thm]{Example}

\theoremstyle{definition} 
\newtheorem{example2}[thm]{Example} 
\newtheorem{remark}[thm]{Remark}

\def\eqn#1$$#2$${\begin{equation}\label#1#2\end{equation}}

\def\A{\mathcal A} 
\def\B{\mathcal B} 
\def\C{\mathcal C} 
\def\D{\mathcal D}

\def\F{\mathcal F} 
\def\N{\mathcal N}

\def\U{\mathcal U} 
\def\V{\mathcal V}

\def\M{\mathcal M}

\def\P{\mathcal{P}} 
\def\Pb{\overline{\P}}

\def\ce{\mathbb C}

\def\en{\mathbb N} 
\def\er{\mathbb R} 
\def\qe{\mathbb Q}

\def \htt{\operatorname{ht}}

\def \ims {\operatorname{ims}} 
\def \dens {\operatorname{dens}} 
\def \card {\operatorname{card}}

\def \Lev {\operatorname{Lev}} 
\def\cf{\operatorname{cf}} 
 
\def\spt{\operatorname{spt}} 
\def\lin{\operatorname{span}} 
\def\clin{\operatorname{\overline{span}}}

\def \reg {\partial _{\kern1pt\text{reg}}}

\def\di{\,\mbox{\rm d}}

\newcommand{\norm}[1]{\left\|#1\right\|}

\newcommand{\wscl}[1]{\overline{#1}^{\,w^*}}

\newcommand{\abs}[1]{\left|#1\right|}
\newcommand{\setsep}{;\,}

\begin{document}

\title{Projectional skeletons and Markushevich bases}
\author{Ond\v{r}ej F.K. Kalenda}
\address{Department of Mathematical Analysis \\
Faculty of Mathematics and Physic\\ Charles University\\
Sokolovsk\'{a} 83, 186 \ 75\\Praha 8, Czech Republic}
\email{kalenda@karlin.mff.cuni.cz}
\thanks{Supported by the Research grant GA \v{C}R 17-00941S} 
\begin{abstract} 
We prove that Banach spaces with a $1$-projectional skeleton form a $\mathcal{P}$-class and deduce that any such space
admits a strong Markushevich basis. We provide several equivalent characterizations of spaces with a projectional skeleton
and of spaces having a commutative one. We further analyze known examples of spaces with a non-commutative projectional skeleton and 
compare their behavior with the commutative case. Finally, we collect several open problems. 

\end{abstract}

\maketitle

\section{Introduction}

\emph{Projectional resolutions of the identity} (shortly \emph{PRI}), introduced and used for the first time by J.~Lindenstrauss \cite{lindenstrauss-65}, are an important tool for investigation of nonseparable Banach spaces. The main application consists in extending some results from separable spaces to certain classes of non-separable ones using transfinite induction, see, e.g., \cite[Section 6.2]{fabian-kniha}.
Let us recall the definition of a PRI. Let $X=(X,\norm{\cdot})$ be a non-separable Banach space and let $\kappa=\dens X$.
(Recall that $\dens X$ denotes the density character of $X$, i.e., the smallest cardinality of a dense subset of $X$. Further, any cardinal number is, as usually, identified with the first ordinal of the given cardinality.)
 A PRI is a transfinite sequence of projections $(P_\alpha)_{\alpha\le\kappa}$ satisfying the following properties.
\begin{itemize}
\item[(i)] $P_0=0$, $P_\kappa=I$;
\item[(ii)] $\norm{P_\alpha}=1$ for $0<\alpha\le\kappa$;
\item[(iii)] $\dens P_\alpha X\le \max\{\aleph_0,\card\alpha\}$ for $\alpha\le\kappa$;
\item[(iv)] $P_\alpha P_\beta=P_\beta P_\alpha=P_\alpha$ for $\alpha\le\beta\le\kappa$;
\item[(v)] $P_\lambda X=\overline{\bigcup_{\alpha<\lambda}P_\alpha X}$ for $\lambda\le\kappa$ limit.
\end{itemize}
So, a PRI provides a decomposition of the space $X$ to certain subspaces of a smaller density. In order to prove a property of $X$ using a transfinite induction argument, we need to know that the property is satisfied by the smaller subspaces. It inspires the following definitions of a $\P$-class and of a $\Pb$-class of Banach spaces.

Let $\C$ be a class of Banach spaces.
\begin{itemize}
\item \cite[Definition 3.45 on p. 107]{HMVZ-biortogonal} We say that $\C$ is a \emph{$\P$-class} if for any nonseparable space $X\in\C$ there is a PRI $(P_\alpha)_{\alpha\le\kappa}$ on $X$ such that $(P_{\alpha+1}-P_\alpha)X\in\C$ for each $\alpha<\kappa$.
\item \cite[p. 417]{hajek-johanis}  We say that $\C$ is a \emph{$\Pb$-class} if for any nonseparable space $X\in\C$ there is a PRI $(P_\alpha)_{\alpha\le\kappa}$ on $X$ such that $P_\alpha X\in\C$ for each $\alpha<\kappa$.
\end{itemize}
Certain classes of Banach spaces are easily seen to be both $\P$-classes and $\Pb$-classes as soon as we know they admit a PRI. For example, any weakly compactly generated Banach spaces admits a PRI by \cite{Amir-Lindenstrauss}. Since this class is stable to taking complemented subspaces, it is clearly both a $\P$-class and a $\Pb$-class. Similarly we can proceed for the classes of reflexive spaces, subspaces of weakly compactly generated spaces, weakly K-analytic spaces (see e.g. \cite[Section 4.1]{fabian-kniha}), weakly countably determined (Va\v{s}\'ak) spaces (see \cite{vasak,gulko79} or \cite[Chapter 7]{fabian-kniha}) and weakly Lindel\"of determined (WLD) spaces (see \cite{AM}). Indeed, all these classes are stable to subspaces and any nonseparable space belonging there admits a PRI.
 
The situation becomes more complicated if we look at the larger classes of $1$-Plichko Banach spaces or on spaces admitting a $1$-projectional skeleton. The class of $1$-Plichko spaces was investigated already in \cite{valdivia-sim}, later in \cite{orihuela} under the name class $\V$, 
the current name was given in \cite{ja-survey}. This class contains many Banach spaces naturally appearing in mathematics, see \cite{ja-val-exa,BHK-vN,BHK-JBW,triples}.
Let us recall the respective definitions:

Let $X$ be a Banach space.
\begin{itemize}
\item A subspace $D\subset X^*$ is said to be a \emph{$\Sigma$-subspace} of $X^*$ if there is a linearly dense set $M\subset X$ such that
$$D=\{x^*\in X^*\setsep \{x\in M\setsep x^*(x)\ne0\}\mbox{ is countable}\}.$$
\item $X$ is said to be \emph{$1$-Plichko} if $X^*$ admits a $1$-norming $\Sigma$-subspace.
\item $X$ is said to be \emph{Plichko} if if $X^*$ admits a norming $\Sigma$-subspace.
\item $X$ is said to be \emph{weakly Lindel\"of determined} (\emph{WLD}) if $X^*$ is a $\Sigma$-subspace of itself.
\end{itemize}

Note that, as indicated by the presence of the constant $1$ in the name, $1$-Plichko spaces are not stable to isomorphisms.
(The stability fails even in a very strong way, see \cite{ja-valeqnorms}.)

The definitions used in \cite{valdivia-sim,orihuela} were different, their equivalence with the current one follows from \cite[Theorem 2.7]{ja-survey}. Any $1$-Plichko space admits a PRI -- this follows from \cite[Theorem 1 and Note 1]{valdivia-sim}. Moreover, $1$-Plichko spaces form both a $\P$-class (by \cite[Theorem 4.14]{ja-survey}) and a $\Pb$-class (this can be proved by a minor adjustment of the proof of \cite[Theorem 4.14]{ja-survey}; it also follows from \cite[Theorem 17.6]{KaKuLo} -- more precisely from its proof using \cite[Theorem 27]{kubis-skeleton}). These results are not just a mere consequence of the existence of a PRI, as a (complemented) subspace of a $1$-Plichko space need not be $1$-Plichko, see \cite{ja-serdica,ja-subspacesC(K)} or \cite[Sections 4.5 and 5.2]{ja-survey}. So, one should take care during the construction of a PRI.

$1$-Plichko spaces can be characterized and generalized using the notion of a projectional skeleton introduced in \cite{kubis-skeleton}. Let us recall the definition and basic properties. 

Let $X$ be a Banach space. A \emph{projectional skeleton} on $X$ is an indexed family $(P_s)_{s\in\Gamma}$ of bounded linear projections on $X$, where $\Gamma$ is an  up-directed partial ordered set, satisfying the following properties:
\begin{itemize}
\item[(i)] $P_sX$ is separable for $s\in\Gamma$;
\item[(ii)] $P_sP_t=P_tP_s=P_s$ whenever $s,t\in\Gamma$ and $s\le t$;
\item[(iii)] If $(s_n)$ is an increasing sequence in $\Gamma$, then $s=\sup_{n\in\en}s_n$ exists in $\Gamma$ and $P_sX=\overline{\bigcup_{n\in\en}P_{s_n}X}$;
\item[(iv)] $X=\bigcup_{s\in\Gamma} P_s X$.
\end{itemize}
Note that the condition (iii) in particular implies, that any increasing sequence in $\Gamma$ has a supremum, i.e., $\Gamma$ is \emph{$\sigma$-complete}. 

If  $(P_s)_{s\in\Gamma}$ is a projectional skeleton $X$, the subspace of $X^*$ defined by
$$D=\bigcup_{s\in\Gamma}P_s^*X^*$$
is said to be \emph{induced by} the skeleton. If $\Gamma'\subset\Gamma$ is \emph{cofinal}, i.e.,
$$\forall s\in\Gamma\,\exists t\in\Gamma': s\le t,$$
and \emph{$\sigma$-closed}, i.e.,
$$\mbox{whenever $(s_n)$ is an increasing sequence in $\Gamma'$, its supremum in $\Gamma$ belongs to $\Gamma'$,}$$
then clearly $(P_s)_{s\in\Gamma'}$ is also a projectional skeleton on $X$ and the respective induced subspace is again $D$. Therefore, by \cite[Proposition 9 and Lemma 10]{kubis-skeleton} we can assume without loss of generality that the projections are uniformly bounded and, moreover, the following stronger version of (iii) holds:
\begin{itemize}
\item[(iii')] If $(s_n)$ is an increasing sequence in $\Gamma$ and $s=\sup_{n\in\en}s_n$, then $P_sx=\lim\limits_{n\to\infty}P_{s_n}x$ for $x\in X$.
\end{itemize}
A \emph{$1$-projectional skeleton} is a projectional skeleton made from norm one projections. 

By \cite[Theorem 27]{kubis-skeleton} a Banach space is $1$-Plichko if and only if it admits a commutative $1$-projectional skeleton. Here, the word \emph{commutative} means that $P_sP_t=P_tP_s$ for any $s,t\in\Gamma$, not only for comparable pairs. A more precise version of the equivalence says that $D\subset X^*$ is a $1$-norming $\Sigma$-subspace if and only if it is induced by a commutative $1$-projectional skeleton. Indeed, the `only if part' follows by \cite[Proposition 21]{kubis-skeleton}. As for the `if part', assuming $D$ is induced by a commutative $1$-projectional skeleton, there is a $1$-norming subspace $D'\subset D$ contained in a $\Sigma$-subspace. Therefore the $\Sigma$-subspace equals $D$ by Lemma~\ref{L:tightness}(b) below.

On the other hand, there are spaces admitting a (non-commutative) $1$-projectional skeleton which fail to be $1$-Plichko -- spaces of continuous functions on ordinals or, more generally, on certain trees, or duals to Asplund spaces. For a more detailed discussion of these examples we refer to Section~\ref{S:examples}.

Any Banach space with a $1$-projectional skeleton admits a PRI
\cite[Theorem 12]{kubis-skeleton}. In fact, spaces with a $1$-projectional skeleton form a $\Pb$-class \cite[Theorem 17.6]{KaKuLo}. Up to the knowledge of the author, the statement  that they form also a $\P$-class is nowhere proved in the literature, contrary to the claim contained in the introduction of \cite{cuth-simult} (referring in error to \cite[Theorem 17.6]{KaKuLo}). The aim of the present paper is, among others, to fill in this gap by the following theorem.

\begin{thm}\label{T:P-class} Let $X$ be a non-separable Banach space admitting a $1$-projectional skeleton. Let $\kappa=\dens X$. Then there is a PRI $(P_\alpha)_{\alpha\le\kappa}$ on $X$ such that the space $(P_{\alpha+1}-P_\alpha)X$ admits a $1$-projectional skeleton for each $\alpha<\kappa$. In other words, the spaces admitting a $1$-projectional skeleton form a $\P$-class.
\end{thm}

This theorem will be proved in the next section -- it follows immediately from Proposition~\ref{P:PRI}(a,c) below.
The main application of Theorem~\ref{T:P-class} is the below result on existence of a Markushevich basis.

Recall that a \emph{Markushevich basis} of a Banach space $X$ is an indexed family $(x_\alpha,x_\alpha^*)_{\alpha\in\Lambda}$ in $X\times X^*$ satisfying the following three conditions.
\begin{itemize}
\item $x_\alpha^*(x_\alpha)=1$ and $x_\alpha^*(x_\beta)=0$ for $\alpha\ne\beta$ in $\Lambda$ (i.e., it is a biorthogonal system);
\item $\clin\{x_\alpha\setsep \alpha\in\Lambda\}=X$;
\item the set $\{x_\alpha^*\setsep \alpha\in\Lambda\}$ separates points of $X$.
\end{itemize}
Moreover, a Markushevich basis  $(x_\alpha,x_\alpha^*)_{\alpha\in\Lambda}$ is said to be \emph{strong} if
$$\forall x\in X: x\in\clin\{x_\alpha\setsep \alpha\in \Lambda\ \&\ x_\alpha^*(x)\ne0\}.$$

Any separable Banach space admits a strong Markushevich basis by \cite{Terenzi}, hence it may be viewed as a weak substitute of a Schauder basis which need not exist. 

Using Theorem~\ref{T:P-class} and \cite[Theorem 5.1]{HMVZ-biortogonal} we immediately obtain the following consequence.

\begin{thm}\label{T:M-basis}
Any Banach space with a projectional skeleton admits a strong Markushevich basis.
\end{thm}

To give a proof of Theorem~\ref{T:M-basis} was the first motivation of the present paper, as it is used in the proof of \cite[Theorem 4.1]{cuka-cejm} with the reference to the introduction of \cite{cuth-simult}. 

Apart of the proof of Theorem~\ref{T:M-basis} we discuss in more detail the difference between commutative and non-commutative cases. In Section~\ref{S:commutative} we provide some characterizations and special properties of ($1$-)Plichko spaces. In Section~\ref{S:noncommutative}
we give analogous characterizations of spaces with a (possibly non-commutative) projectional skeleton. The analogy is not complete, some questions remain open. This is illustrated in Section~\ref{S:examples} where the currently known examples of spaces admitting a non-commutative projectional skeletons are described and analyzed. In the last section we collect open problems.

We finish the introductory section by giving two lemmata we will use throughout the paper.

\begin{lemma}\label{L:norma projekci} Let $X$ be a Banach space, $(P_s)_{s\in\Gamma}$ a projectional skeleton on $X$ with induced subspace $D$.
Let $C\ge1$.
\begin{itemize}
\item Assume that $\norm{P_s}\le C$ for each $s\in\Gamma$. Then $D$ is $C$-norming.
\item Assume that $D$ is $C$-norming. Then there is a cofinal $\sigma$-closed $\Gamma'\subset\Gamma$ such that $\norm{P_s}\le C$ for $s\in\Gamma'$.
\end{itemize}
\end{lemma}

\begin{proof}
The proof is easy, it is explicitly given in \cite[Lemma 1]{ja-vasak}.
\end{proof}

\begin{lemma}\label{L:tightness} Let $X$ be a Banach space.
\begin{itemize}
\item[(a)]  Let $(P_s)_{s\in\Gamma}$ be a projectional skeleton on $X$ with induced subspace $D$. Then $D$ is weak$^*$-countably closed in $X^*$ and norm-bounded subsets of $D$ have countable tightness in the weak$^*$ topology.
\item[(b)] Let $D_1$ and $D_2$ be two subspaces of $X^*$ such that bounded subsets of $D_1$ have countable tightness in the weak$^*$ topology, $D_2$ is weak$^*$-countably closed  and $D_1\cap D_2$ is norming. Then $D_1\subset D_2$.

Hence, if $D_1$ and $D_2$ are moreover subspaces induced by some projectional skeletons, then $D_1=D_2$.
\end{itemize}
\end{lemma}

\begin{proof} The first part of (b) is trivial, the second part follows from the first one using (a). So, let us prove the assertion (a).
This assertion follows from a more general result \cite[Theorem 18]{kubis-skeleton}. However, the proof uses the method of elementary submodels and contains a gap (see the comments at the end of the current proof). We give an easy direct proof.

First observe, that for any countable set $C\subset D$ there is some $s\in \Gamma$ with $C\subset P_s^*X^*$ (this follows easily from definitions). Since $P_s^*X^*$ is weak$^*$ closed, we deduce that $\wscl C\subset D$. This shows that $D$ is weak$^*$-countably closed.

Further, for any $s\in \Gamma$ the space $P_s^*X^*$ is hereditarily separable in the weak$^*$ topology. Indeed, the mapping $y^*\mapsto y^*\circ P_s$ is an isomorphism of $(P_sX)^*$ onto $P_s^*X^*$ which is also a weak$^*$-to-weak$^*$ homeomorphism. $(P_sX)^*$, as the dual of a separable space, is hereditarily separable in the weak$^*$ topology (it has even a countable network), so the same is true for $P_s^*X^*$.

Now assume that $A\subset D$ is bounded, $x^*\in D$ and $x^*\in\wscl A$. Without loss of generality assume that the projections $P_s$ are   uniformly bounded. Fix $s_0\in\Gamma$ such that $P_{s_0}^*x^*=x^*$. We can construct by induction countable sets $C_n\subset A$ and elements $s_n\in\Gamma$ such that
\begin{itemize}
\item $P_{s_{n-1}}^* C_n$ is weak$^*$ dense in $P_{s_{n-1}}^*A$;
\item $s_n\ge s_{n-1}$ and $C_n\subset P_{s_n}^*X^*$.
\end{itemize}
Let $s=\sup_n s_n$. We claim that $\bigcup_n C_n$ is weak$^*$-dense in $P_s^*A$. So, fix any $y^*\in A$ and let $U$ be a weak$^*$-neighborhood of $P_s^*y^*$. We are going to prove that $U\cap\bigcup C_n\ne\emptyset$. We may (and do) assume that $U$ is of the form
$$U=\{z^*\in X^*\setsep \abs{P_s^*y^*(x_j)-z^*(x_j)}<\varepsilon\mbox{ for }j=1,\dots,k\}$$
for some $x_1,\dots,x_n\in X$ and $\varepsilon>0$. Since $\{P_s^*y^*\}\cup\bigcup_n C_n\subset P_s^*X^*$, we can without loss of generality assume that $x_j\in P_sX$ for each $j\in\{1,\dots,k\}$. Moreover, the mentioned set is bounded, so without loss of generality we may assume that $x_j\in\bigcup_n P_{s_n}X$ (as this is a dense subset of $P_sX$). Then there is some $n$ such that $x_j\in P_{s_n}X$ for $j=1,\dots,k$. Since $P_{s_m}^*y^*\overset{w^*}{\longrightarrow} P_s^* y^*$, there is some $m>n$ such that $P_{s_m}^*y^*\in U$. Then there is $z^*\in C_{m+1}$ such that $P_{s_m}^*z^*\in U$. Since for any $j=1,\dots,k$ we have
$$z^*(x_j)=z^*(P_{s_n} x_j)=z^*(P_{s_{m}} P_{s_n}x_j)=z^*(P_{s_{m}}x_j)=P_{s_{m}}^*z^*(x_j),$$ 
we deduce $z^*\in U$. This completes the proof that $\bigcup_n C_n$ is weak$^*$-dense in $P_s^*A$, in particular
$$x^*=P_s^*x^*\in\wscl{P_s^*A}=\wscl{\bigcup_n C_n}.$$

This completes the proof. Let us now point out what is the gap in the proof in \cite{kubis-skeleton}. The quoted result claims that $D$ has countable tightness, not only bounded subsets of $D$ do. The proof uses elementary submodels, but the procedure is in fact similar to our proof. The difficulty appears when one assumes that $x_j$ are from a dense subset. 
It is possible if $A$ is bounded, but not always for an unbounded set. Fortunately, the statement of \cite[Theorem 18]{kubis-skeleton} is true
(see Remark~\ref{remark induced}(g)), but we do not know of any easy and elementary proof.
\end{proof}

\section{Projectional resolutions constructed from projectional skeletons}

The aim of this section is to prove Theorem~\ref{T:P-class}. This will be done by proving Proposition~\ref{P:PRI} below.
We will proceed by refining the proof of \cite[Theorem 17.6]{KaKuLo} using some results of \cite{cuth-simult}.

Throughout this section $X=(X,\norm{\cdot})$ will be a fixed Banach space, $(P_s)_{s\in\Gamma}$ a fixed $1$-projectional skeleton on $X$ and 
$$D=\bigcup_{s\in\Gamma}P_s^*X^*$$ 
the respective induced subspace of $X^*$. Further, $\sigma(X,D)$ will denote the weak topology on $X$ generated by $D$ (i.e., the weakest topology making all functionals from $D$ continuous).

One of the key tools to prove Theorem~\ref{T:P-class} is the following lemma, especially its part (b).
It follows easily from \cite[Proposition 3.1]{cuth-simult} (cf. the proof of \cite[Theorem 4.6]{cuth-simult}). In view of the fact that this lemma is not explicitly formulated and proved in \cite{cuth-simult} and it is important for the present paper, we provide a complete proof.

\begin{lemma}\label{L:invariance}
Let $Y\subset X$ be a closed subspace.
\begin{itemize}
\item[(a)] Suppose that $P_s(Y)\subset Y$ for each $s\in\Gamma$. Then $(P_s|_Y)_{s\in\Gamma}$ is a $1$-projectional skeleton on $Y$ and the respective induced subspace is $\{x^*|_Y\setsep x^*\in D\}$.
\item[(b)] Suppose that $Y$ is $\sigma(X,D)$-closed. Then there is a cofinal $\sigma$-closed subset $\Gamma'\subset\Gamma$ such that $P_s(Y)\subset Y$ for each $s\in\Gamma'$. In particular, $(P_s|_Y)_{s\in\Gamma'}$ is a $1$-projectional skeleton on $Y$ and the respective induced subspace is $\{x^*|_Y\setsep x^*\in D\}$.
\end{itemize}
\end{lemma}

\begin{proof} (a) It is obvious that $(P_s|_Y)_{s\in\Gamma}$ is a $1$-projectional skeleton on $Y$ and the respective induced subspace is $\bigcup_{s\in\Gamma}(P_s|_Y)^*Y^*$.
Fix any $s\in\Gamma$, $y^*\in Y^*$ and  $x^*\in X^*$ such that $y^*=x^*|_Y$ (such $x^*$ exists by the Hahn-Banach theorem). For any $y\in Y$ we have
$$ (P_s|_Y)^*y^*(y)=y^*(P_s y)=x^*(P_s y)=P_s^* x^*(y), 
$$
hence $(P_s|_Y)^*y^*=P_s^*x^*|_Y$.
Therefore 
$$\bigcup_{s\in\Gamma}(P_s|_Y)^*Y^*=\bigcup_{s\in \Gamma}\{P_s^*x^*|_Y\setsep x^*\in X^*\}=\{x^*|_Y\setsep x^*\in D\}.$$

(b) Let $K=(B_{X^*},w^*)$. Then $K$ is a compact space and $(P_s^*|_K)_{s\in\Gamma}$ is a retractional skeleton on $K$, the respective induced subset is $D\cap K$. (For the definition of a retractional skeleton see Section~\ref{S:examples}.1 or, for example, \cite{cuth-simult}; the statement follows easily from definitions, cf. \cite[Proposition 3.14]{cuth-cmuc}.)
The canonical mapping $J:X\to \C(K)$ defined by 
$$J(x)(x^*)=x^*(x), \quad x^*\in K, x\in X,$$
is a $\sigma(X,D)$-to-$\tau_p(D\cap K)$ homeomorphism of $X$ into $\C(K)$ (where $\tau_p(D\cap K)$ denotes the topology of pointwise convergence on $D\cap K$). Moreover, it is easy to observe that $J(X)$ is a $\tau_p(D\cap K)$-closed subset of $\C(K)$ (cf. \cite[Lemma 2.14]{ja-valchar} for the real case and the proof of \cite[Theorem 3.2]{ja-complex-v} for the complex case). Hence $J(Y)$ is $\tau_p(D\cap K)$ closed in $\C(K)$. Now it follows directly from  \cite[Proposition 3.1]{cuth-simult} that there is a cofinal $\sigma$-closed subset $\Gamma'\subset\Gamma$ such that for each $s\in\Gamma'$
$$f\circ (P_s^*|_K)\in J(Y) \mbox{ for each }f\in J(Y).$$
Since for any $y\in Y$ and $x^*\in K$ we have
$$(Jy \circ P_s^*)(x^*)=Jy(P_s^*x^*)=P_s^*x^*(y)=
x^*(P_sy)=J(P_sy)(x^*),$$
we conclude that $P_s(Y)\subset Y$ for $s\in\Gamma'$.

Now it is clear that $(P_s)_{s\in\Gamma'}$ is a $1$-projectional skeleton on $X$
with induced subspace equal to $D$. Hence the rest of (b) follows from (a) applied to the skeleton $(P_s)_{s\in\Gamma'}$.
\end{proof}

A key tool to constructing a PRI is the following construction of a single projection coming from \cite[Lemma 11]{kubis-skeleton}.
For any nonempty directed subset $A\subset\Gamma$ we define a mapping
\begin{equation}\label{eq:P_A def}
P_A x=\lim_{s\in A}P_s x,\qquad x\in X.
\end{equation}
By  \cite[Lemma 11]{kubis-skeleton} (or \cite[Proposition 17.8]{KaKuLo}) the mapping $P_A$ is a well-defined projection of $X$ onto $\overline{\bigcup_{s\in A}P_sX}$.
It is clear that $\norm{P_A}=1$.  For completeness we set $P_\emptyset=0$.

The next easy lemma deals with compatibility of the projections $P_A$ and $P_s$.

\begin{lemma}\label{L:P_sP_A}  Let $A\subset\Gamma$ be a nonempty directed subset. Then the following assertions hold.
\begin{itemize}
\item[(a)] $P_sP_A=P_AP_s=P_s$ for each $s\in A$.
\item[(b)] Let $t\in\Gamma$. If $P_t$ commutes with $P_s$ for each $s\in A$, then $P_t$ commutes with $P_A$.
\item[(c)] If $B\subset\Gamma$ is a directed set containing $A$, then $P_AP_B=P_BP_A=P_A$.
\end{itemize}
\end{lemma}

\begin{proof}
(a) Fix $s\in A$ and $x\in X$. Then for each $t\in A$, $t\ge s$ we have
$$P_s P_t x=P_t P_s x= P_sx.$$
Hence, by taking the limit over $t\in A$ we get
$$P_s P_A x=P_A P_s x=P_sx.$$

(b) Suppose $t\in\Gamma$ satisfies the assumptions. Then for each $s\in A$ we have
$$P_s P_t x=P_t P_s x.$$
Thus, by taking the limit over $s\in A$ we get
$$P_A P_t x=P_t P_A x.$$

(c) Fix $x\in X$. By (a) we get
$$P_s P_B x=P_B P_s x=P_sx \mbox{ for }s\in A,$$
thus by taking the limit over $s\in A$ we deduce
$$P_A P_B x=P_B P_A x=P_A x.$$
\end{proof}

Next we will study in more detail the projection $P_A$. The first statement of the assertion (iii) is used as obvious in the last two sentences of the proof of \cite[Theorem 17.6]{KaKuLo}. The added value of our version is a more precise statement of (iii) and, mainly, the assertion (iv) which plays a key role below.

If $A\subset \Gamma$, we denote by $A_\sigma$ the smallest $\sigma$-closed subset of $\Gamma$ containing $A$.

\begin{prop}\label{P:P_AX a ker P_A} Let $A\subset\Gamma$ be a nonempty directed subset. Denote $Y=P_A X$.
\begin{itemize}
\item[(i)] $A_\sigma$ is a directed subset of $\Gamma$ 
\item[(ii)] $P_A=P_{A_\sigma}$ and, moreover, $P_{A}X=\bigcup_{s\in A_\sigma}P_sX$.
\item[(iii)] The family $(P_s|_Y)_{s\in A_\sigma}$ is a $1$-projectional skeleton in $Y$. The respective induced subspace in $Y^*$ is
\begin{equation}\label{eq:D_A}
D_A=\bigcup_{s\in A_\sigma}(P_s|_Y)^*(Y^*)=\bigcup_{s\in A_\sigma}\{P_s^*x^*|_Y\setsep x^*\in X^*\}\subset \{x^*|_Y\setsep x^*\in D\}.\end{equation}
If the skeleton on $X$ is commutative, then the last inclusion can be replaced by equality.
\item[(iv)] $\ker P_A$ is $\sigma(X,D)$-closed. Therefore there is a cofinal $\sigma$-closed subset $\Gamma'\subset\Gamma$ such that the family $(P_s|_{\ker P_A})_{s\in\Gamma'}$ is a $1$-projectional skeleton on $\ker P_A$. The respective induced subspace is
\begin{equation}\label{eq:D_ker P_A}
D_A^0=\{x^*|_{\ker P_A}\setsep x^*\in D\}.
\end{equation}
\end{itemize}
\end{prop}

\begin{proof}
(i) Let us start by describing $A_\sigma$. Define sets $B_\alpha$ for $\alpha<\omega_1$ as follows.
\begin{itemize}
\item $B_0=A$;
\item $B_{\alpha+1}=B_\alpha\cup\{\sup_n t_n\setsep (t_n) \mbox{ is an increasing sequence in }B_\alpha\}$ for $\alpha<\omega_1$;
\item $B_\lambda=\bigcup_{\alpha<\lambda}B_\alpha$ if $\alpha<\omega_1$ is limit.
\end{itemize}
Then clearly $A_\sigma=\bigcup_{\alpha<\omega_1}B_\alpha$.
Indeed, since $A_\sigma$ is $\sigma$-closed and contains $B_0=A$, by transfinite induction we get $B_\alpha\subset A_\sigma$ for $\alpha<\omega_1$, which proves the inclusion `$\supset$'. To prove the converse inclusion it is enough to observe that the set on the right-hand side is $\sigma$-closed.

To show that $A_\sigma$ is directed, it is enough to prove that $B_\alpha$ is directed for each $\alpha<\omega_1$. 

It is true for $\alpha=0$ as $B_0=A$ and $A$ is assumed to be directed. Suppose that $B_\alpha$ is directed for some $\alpha<\omega_1$. Fix any two indices $s,t\in B_{\alpha+1}$. Then there are increasing sequences $(s_n)$ and $(t_n)$ in $B_{\alpha}$ such that $s=\sup_n s_n$ and $t=\sup_n t_n$. (If $s\in B_\alpha$, we can take $s_n=s$ for each $n\in\en$, and similarly for $t$.) Since $B_{\alpha}$ is directed, we can find a sequence $(u_n)$ in $B_\alpha$ such that
\begin{itemize}
\item $u_1\ge s_1$ and $u_1\ge t_1$;
\item $u_{2n}\ge u_{2n-1}$ and $u_{2n}\ge s_{n+1}$ for $n\in\en$;
\item $u_{2n+1}\ge u_{2n}$ and $u_{2n+1}\ge t_{n+1}$ for $n\in\en$.
\end{itemize}
Since $(u_n)$ is increasing, $u=\sup_n u_n\in B_{\alpha+1}$. Moreover, $u\ge s_n$ for all $n\in\en$, hence $u\ge s$. Similarly, $u\ge t$. This completes the proof that $B_{\alpha+1}$ is directed.

Since the limit induction step is obvious, the proof of (i) is completed. 

(ii) By (i) the mapping $P_{A_\sigma}$ is a well-defined projection with range 
$\overline{\bigcup_{s\in A_\sigma} P_s X}$. Since $A\subset A_\sigma$, 
$P_{A}X\subset P_{A_\sigma}X$. Conversely, using the sets $B_\alpha$ defined within the proof of (i) and the property (iii) of projectional skeletons, by transfinite induction we deduce that $P_s X\subset P_A X$ for each $s\in A_\sigma$. Hence $P_{A}X=P_{A_\sigma}X$.

Let us continue by proving $P_A=P_{A_\sigma}$. Fix $x\in X$. Then
$$P_A x = P_A P_{A_\sigma}x = P_{A_\sigma} x.$$
Indeed, the first equality follows from Lemma~\ref{L:P_sP_A}(c).
To prove the second one observe that $P_{A_\sigma}x\in P_A X$ due to the previous paragraph.

Finally, take any $x\in P_AX$. 
By the definition of $P_A$ there is a sequence $(s_n)$ in $A$ such that
$\norm{P_s x-x}<\frac1n$ whenever $s\in A$ satisfies $s\ge s_n$.
Using the fact that $A$ is directed, we can assume without loss of generality
that the sequence $(s_n)$ is increasing. Set $s=\sup_n s_n$. Then $s\in A_\sigma$ and 
$$P_s x=\lim_n P_{s_n}x= x.$$

(iii) The properties (i)--(iii) of projectional skeleton are obvious, the last property follows from (ii). 

Let us continue by proving \eqref{eq:D_A}. The first equality is just the definition of the induced subspace. 
To show the second equality fix $s\in A_\sigma$. Take any $y^*\in Y^*$ and any $x^*\in X^*$ such that $y^*=x^*|_Y$ (such $x^*$ exists by the Hahn-Banach theorem). Then $(P_s|_Y)^*y^*=P_s^*x^*|_Y$ (see the proof of Lemma~\ref{L:invariance}(a)), so the second equality follows.
The last inclusion is obvious.

Next suppose that the skeleton on $X$ is commutative. By Lemma~\ref{L:P_sP_A}(b) we see that, for any $s\in\Gamma$ we have $P_sP_A=P_AP_s$ and hence the subspace $Y=P_AX$ is invariant for $P_s$. It follows from Lemma~\ref{L:invariance}(a) that $(P_s|_Y)_{s\in\Gamma}$ is a projectional skeleton on $Y$ and the respective induced subspace
is
$$D'_A=\{x^*|_Y\setsep x^*\in D\}
$$
So, due to \eqref{eq:D_A} we see that $D'_A\supset D_A$, hence $D'_A=D_A$ by Lemma~\ref{L:tightness}(b).

(iv) Observe that
$$\begin{aligned}
\ker P_A&=\{x\in X\setsep P_A x=0\}=\{x\in X\setsep (\forall s\in A)(P_s x=0)\}
\\&=\{x\in X\setsep (\forall s\in A)(\forall x^*\in X^*)(x^*(P_s x)=0) \}
\\&=\{x\in X\setsep (\forall s\in A)(\forall x^*\in X^*)(P_s^*x^*(x)=0) \}
\\&=\bigcap\{\ker P_s^*x^*\setsep s\in A,x^*\in X\}.
\end{aligned}$$
Indeed, the first equality is just the definition of the kernel.  The inclusion `$\supset$' from the second one follows from the definition of $P_A$. To prove the converse observe that $P_s=P_sP_A$ by Lemma~\ref{L:P_sP_A}(a).
 The third equality is a consequence of the Hahn-Banach theorem and the last two equalities follow easily from definitions.

Since $P_s^*X^*\subset D$ for each $s\in A$, we conclude that $\ker P_A$ is $\sigma(X,D)$-closed. 

The rest of (iv) now follows immediately from Lemma~\ref{L:invariance}(b).
\end{proof}

The next proposition deals in more detail with the description of $D_A$ from \eqref{eq:D_A} and characterizes the situation when it is maximal possible.

\begin{prop}\label{P:sigma(X,D)-spojitost} Let $A\subset\Gamma$ be a directed subset. The following assertions are equivalent.
\begin{enumerate}
\item $P_A^*(D)\subset D$. 
\item The projection $P_A$ is $\sigma(X,D)$-to-$\sigma(X,D)$ continuous.
\item $P_AX$ is a $\sigma(X,D)$-closed subspace of $X$.
\item There is a cofinal $\sigma$-closed subset $\Gamma'\subset\Gamma$ such that $P_s(P_AX)\subset P_AX$ for each $s\in\Gamma'$.
\item There is a cofinal $\sigma$-closed subset $\Gamma'\subset\Gamma$ such that $P_s P_A = P_A P_s$ for each $s\in\Gamma'$.
\item $D_A=\{x^*|_{P_AX}\setsep x^*\in\D\}$.
\end{enumerate}
\end{prop}

\begin{proof}
(1)$\Rightarrow$(2) To prove the $P_A$ is $\sigma(X,D)$-to-$\sigma(X,D)$ continuous it is enough to show that $x^*\circ P_A$ is $\sigma(X,D)$-continuous for each $x^*\in D$. So, fix $x^*\in D$. Then $x^*\circ P_A=P_A^*x^*\in D$ by the assumption, hence it is $\sigma(X,D)$-continuous.

(2)$\Rightarrow$(3) This implication is obvious.

(3)$\Rightarrow$(4) This follows from Lemma~\ref{L:invariance}(b).

(4)$\Rightarrow$(5) Let $\Gamma'\subset\Gamma$ be the set provided by (4).
By Proposition~\ref{P:P_AX a ker P_A}(iv) we know that there is a cofinal $\sigma$-closed subset $\Gamma''\subset\Gamma'$ such that $\ker P_A$ is invariant for $P_s$ for each $s\in\Gamma''$. 

We claim that for any $s\in\Gamma''$ we have $P_sP_A=P_AP_s$. 
Indeed, since $P_s(P_A X)\subset P_A X$, we deduce $P_A P_s P_A=P_s P_A$.
Moreover, $\ker P_A=(I-P_A)X$ is also invariant for $P_s$, thus
$P_s(I-P_A)X\subset \ker P_A$, i.e., $P_AP_s(I-P_A)=0$. In other words, $P_AP_s=P_AP_sP_A$. Therefore $P_sP_A=P_AP_s$ and the proof is complete.

(5)$\Rightarrow$(1) Fix $x^*\in D$. By the definition of $D$ there is $s\in\Gamma$ with $P_s^*x^*=x^*$. Since $\Gamma'$ is cofinal in $\Gamma$, 
without loss of generality we may assume that $s\in\Gamma'$. By the choice of $\Gamma'$ we have $P_sP_A=P_AP_s$, hence $P_A^*P_s^*=P_s^*P_A*$ as well.
Therefore
$$P_A^*x^*=P_A^*P_s^*x^*=P_s^*P_A^*x^*\in D.$$

(4)$\Rightarrow$(6) Note that under the assumptions of (4)
the system $(P_s)_{s\in\Gamma'}$ is a $1$-projectional skeleton on $X$ with induced subspace $D$. Lemma~\ref{L:invariance}(a) applied to the skeleton
$(P_s)_{s\in\Gamma'}$ yields that
 $(P_s|_{P_AX})_{s\in\Gamma'}$ is a $1$-projectional skeleton on $P_AX$ with induced subspace
$$D'_A= \{x^*|_{P_A}\setsep x^*\in D\}.
$$  
Therefore $D'_A\supset D_A$ (by \eqref{eq:D_A}), hence $D'_A=D_A$ by Lemma~\ref{L:tightness}(b).

(6)$\Rightarrow$(3) This implication follows from \cite[Theorem 4.6(i)$\Rightarrow$(iii)]{cuth-simult}.
\end{proof}

\begin{cor} \
\begin{itemize}
\item[(i)] $P_A$ is $\sigma(X,D)$-to-$\sigma(X,D)$ continuous whenever $A\subset\Gamma$ is a countable directed subset.
\item[(ii)] If the skeleton is commutative, then $P_A$ is $\sigma(X,D)$-to-$\sigma(X,D)$ continuous for any directed subset $A\subset\Gamma$.
\end{itemize}
\end{cor}

\begin{proof} The assertion (ii) follows immediately from Proposition~\ref{P:sigma(X,D)-spojitost}. Let us show the assertion (i).

Enumerate $A=\{s_n\setsep n\in\en\}$. Since $\Gamma$ is directed, we can find an increasing sequence $(t_n)$ in $\Gamma$ such that $t_n\ge s_n$ for $n\in\en$. Let $t=\sup_n t_n$. Then the set
$$\Gamma'=\{s\in\Gamma\setsep s\ge t\}$$
is a cofinal $\sigma$-closed subset. Further for each $s\in \Gamma'$ and $n\in\en$ we have $s\ge t\ge s_n$, thus $P_sP_{s_n}=P_{s_n}P_s (=P_s)$. By Lemma~\ref{L:P_sP_A}(b) we deduce that $P_sP_A=P_AP_s$. Hence, we can conclude by using Proposition~\ref{P:sigma(X,D)-spojitost}.
\end{proof}

The following lemma is the key step to constructing a PRI starting from a $1$-projectional skeleton. Its proof is completely standard. In the proof of \cite[Theorem 17.6]{KaKuLo} it is used without explicit formulation and proof.
We provide a proof for the sake of completeness.

\begin{lemma}\label{L:A_alpha}
Let $\kappa=\dens X$. Then there is a transfinite sequence $(A_\alpha)_{\alpha\le\kappa}$ of subsets of $\Gamma$ satisfying the following properties.
\begin{itemize}
\item[(i)] $A_0=\emptyset$.
\item[(ii)] $A_\alpha$ is a directed subset of $\Gamma$ for $\alpha\le\kappa$.
\item[(iii)] $\card A_\alpha=\dens P_{A_\alpha} X=\max\{\card\alpha,\aleph_0\}$ if $0<\alpha\le\kappa$.
\item[(iv)] $A_\alpha\subset A_\beta$ for $\alpha<\beta\le\kappa$.
\item[(v)] $A_\lambda=\bigcup_{\alpha<\lambda} A_\alpha$ whenever $\lambda\le\kappa$ is limit.
\item[(vi)] $P_{A_\kappa}X=X$, i.e., $P_{A_\kappa}=I$.
\end{itemize}
Moreover, the system $(P_s)_{s\in (A_\kappa)_\sigma}$ is a $1$-projectional skeleton on $X$ with induced subset $D$.
\end{lemma}

\begin{proof} Since $\Gamma$ is up-directed, we can fix a mapping $\varphi:\Gamma\times \Gamma\to\Gamma$ such that
$$\varphi(s,t)\ge s\ \&\ \varphi(s,t)\ge t\qquad\mbox{for }s,t\in\Gamma.$$
If $B\subset\Gamma$ is any nonempty subset, we define the sequence $(B_k)$ by $B_0=B$ and $B_{k}=B_{k-1}\cup \varphi(B_{k-1}\times B_{k-1})$. If we set
$$\eta(B)=\bigcup_{k=0}^\infty B_k,$$
then $\eta(B)\supset B$, $\eta(B)$ is directed and $\card\eta(B)\le \max\{\card B,\aleph_0\}$.

Now we are going to perform the main construction.
Let $\{x_\alpha\setsep\alpha<\kappa\}$ be a dense subset of $X$ not containing $0$. We proceed by transfinite induction.

Set $A_0=\emptyset$. Then (i) is fulfilled.
Fix some $s\in \Gamma$ with $P_s x_0=x_0$, let $B_1\subset\Gamma$ be an infinite countable set containing $s$ and set $A_1=\eta(B_1)$. Then $A_1$ is directed, hence (ii) is satisfied. Moreover, $A_1$ is infinite countable and $P_{A_1}X$ is separable (as it is the closure of $\bigcup_{s\in A_1}P_sX$), hence (iii) is satisfied as well.
The remaining conditions are void, so the first step of the induction is completed.

Suppose that $1\le\alpha<\kappa$ and we have constructed $A_\beta$ for $\beta\le\alpha$ satisfying the conditions (i)-(v). Since $\dens P_{A_\alpha} X<\kappa$, we have 
$P_{A_{\alpha}}X\subsetneqq X$. Let $\gamma<\kappa$ be the smallest ordinal such that $x_{\gamma}\notin P_{A_{\alpha}}X$.  
Take some $s\in\Gamma$ such that $P_s x_\gamma=x_\gamma$. Finally, set $A_{\alpha+1}=\eta(A_\alpha\cup\{s\})$. Then (ii) and (iv) are satisfied for 
$A_\beta$, $\beta\le\alpha+1$. Moreover, since $A_\alpha$ is infinite, we have $\card A_{\alpha+1}=\card A_\alpha$, hence $\card A_{\alpha+1}=\max\{\card(\alpha+1),\aleph_0\}$. Further,
$$\dens P_{A_{\alpha+1}}X\ge\dens P_{A_\alpha}X=\card A_\alpha=\card A_{\alpha+1}$$
and, clearly, $\dens P_{A_{\alpha+1}}X\le\card A_{\alpha+1}$
(as $\bigcup_{s\in A_{\alpha+1}}P_sX$ is dense in $P_{A_{\alpha+1}}X$. Thus (iii) is valid as well. Since (v) and (vi) are void in this case, the `isolated' induction step is completed.

Next suppose that $\lambda\le\kappa$ is limit and we have constructed $A_\alpha$ for $\alpha<\lambda$ such that the conditions (i)-(v) are satisfied. We simply let $A_\lambda=\bigcup_{\alpha<\lambda}A_\lambda$. Then clearly the conditions (ii), (iv) and (v) are again fulfilled. 
Moreover,
$$\card A_\lambda=\sup_{\alpha<\lambda}\card A_\alpha=
\sup_{\alpha<\lambda}\max\{\card\alpha,\aleph_0\}=\card\lambda=\max\{\card\lambda,\aleph_0\}.$$
Further,
$$\dens P_{A_\lambda} X\ge\sup_{\alpha<\lambda}\dens P_{A_\alpha}X=\card\lambda$$
and $\dens P_{A_\lambda}X\le\card\lambda$ as $\bigcup_{s\in A_\lambda}P_sX$ is dense in $P_{A_\lambda}X$.
So, the condition (iii) is fulfilled as well.

It remains to prove (vi). We will show that $P_{A_\kappa}X$ contains $x_\alpha$ for each $\alpha<\kappa$. To this end it is enough to observe that $x_\alpha\in P_{A_{\alpha+1}}X$ for each $\alpha<\kappa$. This can be proved by transfinite induction: $x_0\in P_{A_1}X$ by the construction. Suppose that $\alpha<\kappa$ is such that $x_\gamma\in P_{\gamma+1}$ for each $\gamma<\alpha$. Then
$\{x_\gamma\setsep \gamma<\alpha\}\subset P_{A_\alpha}X$ Therefore, by the construction of $A_{\alpha+1}$ we have $x_\alpha\in P_{A_{\alpha+1}}X$.
\end{proof}

The next proposition is the main achievement of this section. The assertions (a) and (b) are just a bit more precise version of \cite[Theorem 17.6]{KaKuLo}, the assertion (c) is new and provides a proof of Theorem~\ref{T:P-class}.

\begin{prop}\label{P:PRI} Let $\kappa=\dens X$ and let $(A_\alpha)_{\alpha\le\kappa}$ be the family provided by Lemma~\ref{L:A_alpha}.
Then the following assertions hold.
\begin{itemize}
\item[(a)] The family $(P_{A_\alpha})_{\alpha\le\kappa}$ is a PRI on $X$.
\item[(b)] For each $\alpha<\kappa$ the family
$$P_s|_{P_{A_\alpha}X}, s\in (A_{\alpha})_\sigma$$
is a $1$-projectional skeleton on $P_{A_\alpha}X$. The induced subspace is
$$D_\alpha=\{P_s^* x^*|_{P_{A_\alpha}X}\setsep x^*\in X^*\}\subset\{ x^*|_{P_{A_\alpha}X}\setsep x^*\in D\}.$$
If the skeleton on $X$ is commutative, we have 
$$D_\alpha=\{ x^*|_{P_{A_\alpha}X}\setsep x^*\in D\}.$$
\item[(c)] For each $\alpha<\kappa$ the space $(P_{\alpha+1}-P_\alpha)X$ admits a $1$-projectional skeleton with induced subspace
$$D_{\alpha}^{\alpha+1}=\{x^*|_{(P_{\alpha+1}-P_\alpha)X}\setsep x^*\in D_{\alpha+1}\}.$$
\end{itemize}
\end{prop}

\begin{proof} 
(a) The property (i) of a PRI follows from the properties (i) and (vi) of the family $(A_\alpha)$. By the property (ii) it is clear that  $P_{A_\alpha}$ is a norm one projection for $\alpha>0$, thus the property (ii) of a PRI is fulfilled.
The properties (iii)--(v) of a PRI follow from the respective properties of the family $(A_\alpha)$, in case of (iv) together with Lemma~\ref{L:P_sP_A}(c).

The assertion (b) follows from Proposition~\ref{P:P_AX a ker P_A}(iii). The assertion (c) follows from (b) and Proposition~\ref{P:P_AX a ker P_A}(iv).
\end{proof}

\section{A characterization of commutativity of a projectional skeleton}\label{S:commutative}

The aim of this section is to prove two theorems -- Theorem~\ref{T:char sigma subspace} characterizing $\Sigma$-subspaces and Theorem~\ref{T:komutativni} characterizing commutativity of a projectional skeleton. A large part of the characterization given in Theorem~\ref{T:char sigma subspace} is not new, but we provide a unified approach and some new points of view.
This is explained in more detail in Remarks~\ref{Remark Sigma}  below.

We continue by recalling definitions of some notions used in the following theorem.

A \emph{projectional generator} on a Banach space $X$ is a pair $(D,\Phi)$, where $D$ is a norming subspace of $X^*$ and $\Phi$ is a mapping defined on $D$ whose values are countable subsets of $X$ satisfying moreover the condition
$$\forall A\subset D: \overline{A}\mbox{ is a linear subspace}\Rightarrow \Phi(A)^\perp\cap\wscl{A}=\{0\}.$$
The notion of a projectional generator was introduced in \cite{OV-projgen} as a technical tool for constructing a PRI. It is used for example in \cite{fabian-kniha}. There are some minor differences between the definitions used by different authors, it is not clear whether the definitions are equivalent but the differences are not important for applications.

Further, if $M$ is any set, by $[M]^{\le\omega}$ we denote the family of all the countable subsets of $M$ (including the finite sets).  A mapping $\varphi:[M_1]^{\le\omega}\to[M_2]^{\le\omega}$ is called \emph{$\omega$-monotone} if it satisfies the following two properties.
\begin{itemize}
\item $\forall  A,B\in[M_1]^{\le\omega}:A\subset B\Rightarrow \varphi(A)\subset\varphi(B)$;
\item $\varphi(\bigcup_n A_n)=\bigcup_n\varphi(A_n)$ whenever $(A_n)$ is an increasing sequence in $[M_1]^{\le\omega}$.
\end{itemize}
This terminology is a bit misleading (a more natural name would be \emph{$\sigma$-continuous monotone mapping}), but we prefer to use the usual terminology which is nowadays becoming standard, cf. \cite{RoHe-sokolov,RoHe-monotprop,cuth-sepdeterm}.

Finally, a topological space is said to be \emph{primarily Lindel\"of}
if it is a continuous image of a closed subset of the space $(L_\Gamma)^\en$ for a set $\Gamma$, where $L_\Gamma$ is the one-point lindel\"ofication of the discrete set $\Gamma$ (i.e., $L_\Gamma=\Gamma\cup\{\infty\}$, the points of $\Gamma$ are isolated in $L_\Gamma$ and neighborhoods of $\infty$ are complements of countable subsets of $\Gamma$). This class of topological spaces was used in \cite{polprep} to characterize Corson compact spaces (see also \cite{archang-kniha}),
the characterization was generalized to Valdivia compacta and $1$-Plichko Banach spaces in \cite{ja-valchar} (see also \cite[Chapter 2]{ja-survey}).

\begin{thm}\label{T:char sigma subspace} Let $X$ be a Banach space and $D\subset X^*$ a norming subspace.
The following assertions are equivalent.
\begin{enumerate}
\item $D$ is a $\Sigma$-subspace.
\item $D$ is weak$^*$-countably closed and there is a linearly dense subset $M\subset X$ such that the pair $(D,\Phi)$, where
$$\Phi(x^*)=\{x\in M\setsep x^*(x)\ne0\},\quad x^*\in D,$$
is a projectional generator on $X$.
\item There is a linearly dense subset $M\subset X$ and an $\omega$-monotone mapping $\psi:[M\cup D]^{\le\omega}\to[M]^{\le\omega}$ such that for any $A\in[M\cup D]^{\le\omega}$ we have
\begin{itemize}
\item $A\cap M\subset \psi(A)$;
\item $A\cap D\subset (M\setminus\psi(A))^\perp\subset D$;
\item the mapping $x^*\mapsto x^*|_{\clin{\psi(A)}}$ is a bijection of $(M\setminus\psi(A))^\perp$ onto $(\clin \psi(A))^*$.
\end{itemize}
\item There is a linearly dense subset $M\subset X$ and an $\omega$-monotone mapping $\vartheta:[M]^{\le\omega}\to[M]^{\le\omega}$ such that
$$D=\bigcup\{(M\setminus\vartheta(A))^\perp\setsep A\in[M]^{\le\omega}\}$$ and, moreover,  for any $A\in[M]^{\le\omega}$ we have
\begin{itemize}
\item $A\subset \psi(A)$;
\item the mapping $x^*\mapsto x^*|_{\clin{\vartheta(A)}}$ is a bijection of $(M\setminus\psi(A))^\perp$ onto $(\clin\vartheta(A))^*$.
\end{itemize}
\item $D$ is induced by a commutative projectional skeleton on $X$.
\item There is a Markushevich basis $(x_\alpha,x^*_\alpha)_{\alpha\in\Lambda}$
in $X$ such that
$$D=\{x^*\in X^*\setsep   \{\alpha\in\Lambda\setsep x^*(x_\alpha)\ne0\}\mbox{ is countable}\}.$$
\item $D$ is weak$^*$-countably closed and there is a Markushevich basis $(x_\alpha,x^*_\alpha)_{\alpha\in\Lambda}$ in $X$ such that
\begin{itemize}
\item $x^*_\alpha\in D$ for $\alpha\in\Lambda$, and
\item $\{x_\alpha\setsep\alpha\in\Lambda\}\cup\{0\}$ is $\sigma(X,D)$-Lindel\"of.
\end{itemize}
\item $D$ is weak$^*$-countably closed and there is a Markushevich basis $(x_\alpha,x^*_\alpha)_{\alpha\in\Lambda}$ in $X$ such that $\{x_\alpha\setsep\alpha\in\Lambda\}\cup\{0\}$ is primarily Lindel\"of in the topology $\sigma(X,D)$.
\item  $D$ is weak$^*$-countably closed  and $(X,\sigma(X,D))$ is primarily Lindel\"of.
\end{enumerate}
\end{thm}

\begin{proof} It is clear that the assertions (1)--(9) are not changed by taking an equivalent norm. Therefore without loss of generality we may and shall assume that $D$ is $1$-norming. So, we can fix a mapping $\eta$ assigning to each $x\in X$ a countable subset $\eta(x)\subset D\cap B_{X^*}$ such that
$$\norm{x}=\sup\{\abs{x^*(x)}\setsep x^*\in\eta(x)\}.$$

(1)$\Rightarrow$(2) Suppose that $D$ is a $\Sigma$-subspace and let $M\subset X$ be a linearly dense set witnessing it. Define $\Phi$ as in the statement of (2). We will show that the pair $(D,\Phi)$ is a projectional generator. By the very definition of a $\Sigma$-subspace it is clear that $\Phi$ is countably-valued. Further, take any $A\subset D$. Then
$$\wscl{\lin A}\cap \Phi(A)^\perp=\{0\}.$$
Indeed, suppose $x^*\in\wscl{\lin A}\setminus\{0\}$. Since $M$ is linearly dense, there is $x\in M$ with $x^*(x)\ne 0$. Hence, there is $y^*\in A$ such that $y^*(x)\ne0$.
It follows that $x\in\Phi(y^*)\subset\Phi(A)$. So, $x^*\notin\Phi(A)^\perp$.

Finally, to show that $D$ is weak$^*$-countably closed, fix a countable set $C\subset D$. Then 
$$C\subset (M\setminus\Phi(C))^\perp\subset D.$$
Indeed, the first inclusion follows from the definition of $\Phi$ and the second one follows from the definition of a $\Sigma$-subspace as $\Phi(C)$ is countable. Since
$(M\setminus\Phi(C))^\perp$ is weak$^*$-closed, the prove is completed.

(2)$\Rightarrow$(1) Let $M$ and $\Phi$ be as in the statement of (2). Since $M$ is linearly dense, we can define $D'$ to be the $\Sigma$-subspace generated by $M$. $\Phi$ is countably-valued, thus $D\subset D'$. Further, $D$ is $1$-norming and weak$^*$-countably closed, thus $D\cap B_{X^*}$ is weak$^*$-dense and weak$^*$-countably closed in $D'\cap B_{X^*}$. Since $D'\cap B_{X^*}$ equipped with the weak$^*$-topology has countable tightness (in fact, it is Fr\'echet-Urysohn, cf. \cite[Lemma 1.6]{ja-survey}), we conclude that $D\cap B_{X^*}=D'\cap B_{X^*}$,
hence $D=D'$.

(1)$\Rightarrow$(3) We are going to prove the implication for real spaces. The proof for complex ones is exactly the same, one just needs to replace everywhere $\qe$ by its complex version $\qe+i\qe$.

Suppose that $D$ is a $\Sigma$-subspace and let $M$ be a linearly dense set witnessing it. Define $\Phi$ as in the statement of (2).
Define a mapping $\theta_1:[D\cup M]^{\le\omega}\to[D\cup M]^{<\omega}$
by the formula
$$\theta_1(A)=A\cup\Phi(A\cap D)\cup\eta(\lin_{\qe} (A\cap M)),\quad A\in[D\cup M]^{\le\omega}.$$
It is obvious that $\theta_1$ is an $\omega$-monotone mapping. Further, for each $n\ge 2$ define a mapping $\theta_n$ by 
$$\theta_n=\theta_1(\theta_{n-1}(A)),\quad A\in[D\cup M]^{\le\omega}.$$
It is clear that $\theta_n$ is an $\omega$-monotone mapping for each $n\in\en$.
Finally, for each $A\in[D\cup M]^{\le\omega}$ define
$$\theta_\infty(A)=\bigcup_{n\in\en} \theta_n(A)\mbox{ and }\psi(A)=\theta_\infty(A)\cap M.$$
Then both $\theta_\infty$ and $\psi$ are $\omega$-monotone mappings.

Let us prove that $\psi$ has the required properties. To this end fix $A\in[D\cup M]^{\le\omega}$.

By construction it is obvious that $A\subset \theta_\infty(A)$, hence $A\cap M\subset \psi(A)$.  Further, $\Phi(A\cap D)\subset\theta_1(A)\cap M\subset\psi(A)$, hence $A\cap D\subset (M\setminus\psi(A))^\perp$.

Since $\psi(A)$ is a countable subset of $M$, we deduce that $(M\setminus\psi(A))^\perp\subset D$ by the definition of a $\Sigma$-subspace.

It remains to prove the last property. We start by showing that

$$\forall x\in\lin_{\qe}\psi(A)\, \forall y\in\lin_{\qe}(M\setminus\psi(A)): \norm{x}\le\norm{x+y}.\eqno{(*)}$$  

So, fix any $x\in\lin_{\qe}\psi(A)$ and $y\in\lin_{\qe}(M\setminus \psi(A))$.
Observe that there is $n\in\en$ such that $x\in\lin_{\qe}(\theta_n(A)\cap M)$. It follows that $\eta(x)\subset\theta_{n+1}(A)$, hence $\Phi(\eta(x))\subset\theta_{n+2}(A)\cap M\subset\psi(A)$. Thus $\eta(x)\subset (M\setminus\psi(A))^\perp$.
Therefore
$$\begin{aligned}
\norm{x}&=\sup \{\abs{x^*(x)}\setsep x^*\in\eta(x)\}
\le \sup \{\abs{x^*(x)}\setsep x^*\in(M\setminus\psi(A))^\perp, \norm{x^*}\le 1\}
\\&=  \sup \{\abs{x^*(x+y)}\setsep x^*\in(M\setminus\psi(A))^\perp, \norm{x^*}\le 1\}
\le\norm{x+y},
\end{aligned}$$
which completes the proof of $(*)$.

Next we are going to show that
$$\forall x^*\in (M\setminus\psi(A))^\perp: \norm{x^*|_{\clin\psi(A)}}=\norm{x^*}.\eqno{(**)}$$

Since the inequality `$\le$' is obvious, it is enough to prove the converse one. Fix any $c<\norm{x^*}$. Then there is $z_0\in B_X$ with $\abs{x^*(z_0)}>c$. Since $M$ is linearly dense, $\lin_{\qe}M$ is norm-dense in $X$, thus there is $z_1\in B_X\cap\lin_{\qe}M$ with $\abs{x^*(z_1)}>c$. Then $z_1$ can be uniquely expressed as $z_1=x+y$ with $x\in\lin_{\qe}\psi(A)$ and $y\in\lin_{\qe}(M\setminus\psi(A))$. By $(*)$ we get $\norm{x}\le\norm{z_1}\le1$. Moreover, $x^*(y)=0$, hence
$$c<\abs{x^*(z_1)}=\abs{x^*(x)}\le\norm{x^*|_{\clin\psi(A)}}.$$
Thus $(**)$ is proved.

The last ingredient is
$$\forall y^*\in (\clin\psi(A))^*\,\exists x^*\in(M\setminus\psi(A))^\perp: x^*|_{\clin \psi(A)}=y^*.\eqno{(***)}$$  

Define first $x^*$ on $\lin_{\qe} M=\lin_{\qe}\psi(A)+\lin_{\qe}(M\setminus\psi(A))$ by
$$x^*(x+y)=y^*(x),\quad x\in\lin_{\qe}\psi(A), y\in\lin_{\qe}(M\setminus\psi(A)).$$
It follows from $(*)$ that  $\lin_{\qe}\psi(A)\cap\lin_{\qe}(M\setminus\psi(A))=\{0\}$,
hence $x^*$ is a well-defined $\qe$-linear functional. Moreover, it also follows from $(*)$ that $\abs{x^*(z)}\le\norm{y^*}$ for each $z\in B_X\cap\lin_{\qe}M$. It follows that $x^*$ can be uniquely extended to an element of $X^*$. It is clear that this extension belongs to $(M\setminus\psi(A))^\perp$ and extends $y^*$.

Finally, putting together $(**)$ and $(***)$ we get that
$$x^*\mapsto x^*|_{\clin\psi(A)}\mbox{ is an isometry of $(M\setminus\psi(A))^\perp$ onto }(\clin\psi(A))^*,$$
which completes the proof.

(3)$\Rightarrow$(4) Let $M$ and $\psi$ be as in the statement of (3). Let us define the mapping $\vartheta$ by the formula
$$\vartheta(A)=\psi(A),\quad A\in[M]^{\le\omega}.$$
It is clear that the mapping $\vartheta$ is $\omega$-monotone. For any $A\in[M]^{\le\omega}$ obviously $A\subset\vartheta(A)$ and the mapping $x^*\mapsto x^*|_{\clin{\vartheta(A)}}$ is a bijection of $(M\setminus\vartheta(A))^\perp$ onto $(\clin\vartheta(A))^*$, due to the properties of $\psi$.

It remains to show the formula for $D$. By the properties of $\psi$ and definition of $\vartheta$ we deduce that
$$D=\bigcup\{(M\setminus\psi(C))^\perp\setsep C\in[M\cup D]^{\le\omega}\}\supset\bigcup\{(M\setminus\vartheta(A))^\perp\setsep A\in[M]^{\le\omega}\}.$$
To prove the converse inclusion fix any $x^*\in D$. Then there is $C\in[M\cup D]^{\le\omega}$ such that $x^*\in(M\setminus\psi(C))^\perp$. Set $A=\psi(C)$. Since $\vartheta(A)\supset A$, we deduce that
$$x^*\in(M\setminus\psi(C))^\perp=(M\setminus A)^\perp\subset(M\setminus\vartheta(A))^\perp.$$
This completes the proof.

(4)$\Rightarrow$(5) Let $M$ and $\vartheta$ be as in the assertion (4). We are going to construct a projectional skeleton. We start by defining the respective index set. Set
$$\Gamma=\{A\in[M]^{\le\omega}\setsep; \vartheta(A)=A\}\cup\{\emptyset\}$$
and consider the partial order on $\Gamma$ given by inclusion. This index set has the following properties:
\begin{itemize}
\item[(i)] $\forall C\in[M]^{\le\omega}\,\exists A\in\Gamma: C\subset A$;
\item[(ii)] $\Gamma$ is up-directed;
\item[(iii)] if $(A_n)_n$ is an increasing sequence in $\Gamma$, then $\bigcup_n A_n\in\Gamma$;
\item[(iv)] $\Gamma$ is closed to taking arbitrary intersections;
\item[(v)] $\Gamma$ is a lattice, i.e., any two-point subset of $\Gamma$ admits a supremum and an infimum in $\Gamma$.
\end{itemize}

Let us now prove these properties:

(i) Fix $C\in[M]^{\le\omega}$. Set $C_1=C$ and define, by induction, $C_{n+1}=\vartheta(C_n)$ for $n\in\en$. Finally, set $A=\bigcup_n C_n$. 
By the properties of $\vartheta$ we deduce that the sequence $(C_n)$ is increasing. Hence $A\subset C$ and, moreover,
$$\vartheta(A)=\bigcup_n\vartheta(C_n)=\bigcup_n C_{n+1}=A,$$
so $A\in\Gamma$.

(ii) Let $A_1,A_2\in \Gamma$. By (i) there is $A\in\Gamma$ such that $A\supset A_1\cup A_2$.

(iii) Since $\vartheta$ is $\omega$-monotone, we have
$$\vartheta\left(\bigcup_n A_n\right)=\bigcup_n\vartheta(A_n)=\bigcup_n A_n.$$

(iv) Let $\Gamma'\subset\Gamma$ be any subset. If $\bigcap\Gamma'=\emptyset$, it is an element of $\Gamma$. Suppose that $A=\bigcap\Gamma'\ne\emptyset$. Then
$$A\subset\vartheta(A)\subset\bigcap_{C\in\Gamma'}\vartheta(C)=\bigcap_{C\in\Gamma'}C=A,$$
thus $A\in\Gamma$.

(v) If $A_1,A_2\in\Gamma$, their infimum is $A_1\wedge A_2=A_1\cap A_2$ (by (iv) the intersection belongs to $\Gamma$) and their supremum is
$$A_1\vee A_2=\bigcap\{C\in\Gamma\setsep C\supset A_1\cup A_2\}.$$
Indeed, the subset of $\Gamma$ on the right-hand side is nonempty by (i) and the intersection belongs to $\Gamma$ by (iv).

\smallskip

This completes the proof of the properties of $\Gamma$. It remains to construct the projections.
To this end we will use the following easy lemma. The lemma was essentially used in \cite{CuFa-Asplund}, but we formulate it explicitly as we will use it also in the following section.

\begin{lemma}\label{L:jednaprojekce}
Let $X$ be a Banach space, $Y\subset X$ a closed subspace, $V\subset X^*$ a weak$^*$-closed subspace such that
the mapping
$$x^*\mapsto x^*|_Y$$
is a bijection of $V$ onto $Y^*$. Then there is a bounded linear projection $P$ on $X$ such that $PX=Y$, $P^*X^*=V$ and $\ker P=V_\perp$.
\end{lemma}

\begin{proof} By the open mapping theorem the above restriction map is an isomorphism of $V$ onto $Y^*$. So, there is some $c>0$ such that $\norm{x^*|_Y}\ge c\norm{x^*}$ for $x^*\in V$.  
It follows that
$$\norm{y+v}\ge c\norm{y},\quad y\in Y, v\in V_\perp.$$
Indeed, take $y\in Y$ and $v\in V_\perp$. Fix $y^*\in Y^*$ such that $\norm{y^*}=1$ and $\abs{y^*(y)}=\norm{y}$. By the assumption there is $x^*\in V$ with $x^*|_Y=y^*$. By the above we get $\norm{x^*}\le\frac1c$, thus $\norm{cx^*}\le1$. Therefore
$$\norm{y+v}\ge\abs{cx^*(y+v)}=c\abs{x^*(y)}=c\abs{y^*(y)}
=c\norm{y}.$$
It follows that $Y\cap V_\perp=\{0\}$ and the projection of $P:Y+V_\perp\to Y$ with kernel $V_\perp$ is bounded (with norm at most $\frac1c$). In particular, $Y+V_\perp$ is closed.
Finally,
$$(Y+V_\perp)^\perp=Y^\perp\cap (V_\perp)^\perp=Y^\perp\cap V=\{0\}.$$
Indeed, we used the assumptions that $V$ is weak$^*$-closed and that the only $x^*\in V$ with $x^*|_Y=0$ is the zero functional. Hence the bipolar theorem shows that $Y+V_\perp=X$. It follows that the projection $P$ is defined on the whole $X$. Moreover, $P^*X^*$ is weak$^*$-closed (as $P^*$ is a weak$^*$-to-weak$^*$-continuous projection), thus
$$P^*X^*=\wscl{P^*X^*}=(\ker P)^\perp=(V_\perp)^\perp=V.$$
This completes the proof.
\end{proof}

Now let us continue by constructing the projectional skeleton. For any $A\in\Gamma$ let $P_A$ be the projection provided by Lemma~\ref{L:jednaprojekce} for the pair of subspaces $Y=\clin A$ and $V=(M\setminus A)^\perp$. Then 
$$P_AX=\clin A, P_A^*X^*=(M\setminus A)^\perp \mbox{ and }\ker P_A=((M\setminus A)^\perp)_\perp=\clin(M\setminus A).$$

Let us continue by showing
\begin{itemize}
\item[$(\circ)$] $\forall A,B\in\Gamma: P_{A\cap B}=P_A P_B$. 
\end{itemize}
Fix $A,B\in\Gamma$. If $A\cap B=\emptyset$, then $B\subset M\setminus A$ and hence
$$P_BX=\clin B\subset\clin(M\setminus A)=\ker P_A,$$
thus $P_AP_B=0$.

Suppose that $A\cap B=C\ne \emptyset$. To show that $P_A P_B=P_C$ it is enough to prove the equality for any $x\in M$:
$$P_AP_B x=\begin{cases}
P_Ax=x=P_Cx&\mbox{ if }x\in A\cap B=C,\\
P_Ax=0=P_Cx&\mbox{ if }x\in B\setminus A=B\setminus C,\\
P_A0=0=P_Cx&\mbox{ if }x\in M\setminus B\subset M\setminus C.
\end{cases}
$$

Now we are ready to prove that $(P_A)_{A\in\Gamma}$ is a commutative projectional skeleton. Firstly, $\Gamma$ is an up-directed partially order set by the above property (ii). Let us check the properties of a projectional skeleton. $P_AX=\clin A$, so it is separable for each $A\in\Gamma$, hence the property (i) is fulfilled. The property (ii) follows from $(\circ)$. To prove the property (iii) fix an increasing sequence $(A_n)$ in $\Gamma$. By the property (iii) of $\Gamma$ the union $A=\bigcup_n A_n$ belongs to $\Gamma$. Then $A$ is clearly the supremum of the sequence $(A_n)$ and, moreover,
$$P_AX=\clin A\supset \bigcup_n\clin A_n=\bigcup_n P_{A_n}X\supset \bigcup_n \lin A_n=\lin A,$$
hence $P_AX=\overline{\bigcup_n P_{A_n}X}$. Further, let us prove the property (iv). To this end fix any $x\in X$. Since $M$ is linearly dense, there is a countable set $C\subset M$ with $x\in\clin C$. By the property (i) of $\Gamma$ there is $A\in\Gamma$ with $A\supset C$. Then 
$$x\in\clin C\subset\clin A=P_AX.$$
Finally, the skeleton is commutative by $(\circ)$.

It remains to show that $D$ is the subspace induced by this skeleton, i.e.,
$$D=\bigcup_{A\in\Gamma}P_A^*X^*=\bigcup_{A\in\Gamma}(M\setminus A)^\perp.$$
The inclusion `$\supset$' follows from the assumption (4), as for any $A\in\Gamma$ we have
$$(M\setminus A)^\perp=(M\setminus\vartheta(A))^\perp\subset D.$$
Conversely, let $x^*\in D$. By (4) there is $C\in[M]^{\le\omega}$ such that $x^*\in(M\setminus\vartheta(C))^\perp$. By the property (i) of $\Gamma$ there is $A\in\Gamma$ with $A\supset\vartheta(C)$. Then
$$x^*\in(M\setminus\vartheta(C))^\perp\subset(M\setminus A)^\perp.$$

(5)$\Rightarrow$(6) This implication will be proved by transfinite induction on the density of $X$. If $X$ is separable, then $D=X^*$ and $X$ admits a countable Markushevich basis, so the statement is obvious. 

Let $\kappa$ be an uncountable cardinal such that the implication holds whenever $\dens X<\kappa$.
Suppose that $\dens X=\kappa$ and (5) is satisfied. 
 Let $(P_s)_{s\in\Gamma}$ be a commutative projectional skeleton inducing $D$.
Since $D$ is $1$-norming, we can without loss of generality assume that it is a $1$-projectional skeleton (up to passing to a closed cofinal subset of $\Gamma$, see Lemma~\ref{L:norma projekci}).
Let  $(P_\alpha)_{\alpha\le \kappa}$ be a PRI on $X$
provided by Proposition~\ref{P:PRI}.
Then the space $(P_{\alpha+1}-P_\alpha)X$, for each $\alpha<\kappa$ admits a commutative $1$-projectional skeleton with the induced subspace
$$D_\alpha^{\alpha+1}=\{x^*|_{(P_{\alpha+1}-P_\alpha)X}\setsep x^*\in D\}.$$
For any $\alpha<\kappa$ there is, due to the induction hypothesis, a Markushevich basis
$(x_{\alpha,j},x^*_{\alpha,j})_{j\in J_\alpha}$ of the space $(P_{\alpha+1}-P_\alpha)X$ such that
$$D_\alpha^{\alpha+1}=\{y^*\in((P_{\alpha+1}-P_\alpha)X)^*\setsep \{j\in J_\alpha\setsep y^*(x_{\alpha,j})\ne0\}\mbox{ is countable}\}.$$
By the proof of \cite[Proposition 6.2.4]{fabian-kniha} the family
$$(x_{\alpha,j},x^*_{\alpha,j}\circ(P_{A_{\alpha+1}}-P_{A_\alpha}))_{j\in J_\alpha,\alpha<\kappa}$$
is a Markushevich basis of $X$.
It remains to show that
$$D=\{x^*\in X^*\setsep \{(\alpha,j)\setsep j\in J_\alpha, \alpha<\kappa \ \&\  x^*(x_{\alpha,j})\ne0\}\mbox{ is countable}\}.$$

$\subset$: Let $x^*\in D$. Then $x^*|_{(P_{\alpha+1}-P_\alpha)X}\in D_\alpha^{\alpha+1}$ for each $\alpha<\kappa$. So, to show that $x^*$ belongs to the set on the right-hand side it is enough to show that the set
$$C=\{\alpha<\kappa\setsep x^*|_{(P_{\alpha+1}-P_\alpha)X}\ne0\}$$
is countable. We start by observing that, by the definition of $D$, there is  $s\in \Gamma$ with $P_s^*x^*=x^*$. Suppose $C$ is uncountable. For each $\alpha\in C$ find $z_\alpha\in (P_{\alpha+1}-P_\alpha)X$ with $x^*(z_\alpha)\ne0$. Since $C$ is uncountable, there is some $\varepsilon>0$ such that
$$C_1=\{\alpha\in C\setsep\abs{x^*(z_\alpha)}\ge\varepsilon\}$$
is uncountable. Then, given $\alpha,\beta\in C_1$ with $\alpha>\beta$ we have
$$\begin{aligned}
2\norm{x^*}\cdot\norm{P_s z_\alpha-P_s z_\beta}&\ge{\abs{x^*((P_{\alpha+1}-P_\alpha)P_s(z_\alpha-z_\beta))}}={\abs{x^*(P_s(P_{\alpha+1}-P_\alpha)(z_\alpha-z_\beta))}}
\\&=\abs{x^*(P_s(z_\alpha))}
={\abs{P_s^*x^*(z_\alpha)}}=\abs{x^*(z_\alpha)}\ge\varepsilon
\end{aligned}$$
So, the set $\{P_sz\alpha\setsep\alpha\in C_1\}$ is an uncountable discrete subset of the separable space $P_sX$, which is a contradiction.

$\supset$: Let $x^*$ belong to the set on the right-hand side. Then the set
$$C=\{\alpha<\kappa\setsep\exists j\in J_\alpha: x^*(x_{\alpha,j})\ne 0\}$$
is countable. Moreover, 
$$x^*|_{(P_{\alpha+1}-P_\alpha)X}\in D_\alpha^{\alpha+1}\mbox{ for }\alpha\in C
\mbox{ and }x^*|_{(P_{\alpha+1}-P_\alpha)X}=0\mbox{ for }\alpha\in[0,\kappa)\setminus C.
$$
So, for any $\alpha\in C$ there is $y_\alpha^*\in D$ such that 
$$y_\alpha^*|_{(P_{\alpha+1}-P_\alpha)X}=x^*|_{(P_{\alpha+1}-P_\alpha)X},$$
Then there is $s_\alpha\in\Gamma$ such that $y_\alpha^*=P_{s_\alpha}^*y_\alpha^*$. Since $C$ is countable, there is $s\in\Gamma$ such that $s\ge s_\alpha$ for $\alpha\in C$. To show that $x^*\in D$ it is enough to prove that $P_s^*x^*=x^*$.

Recall that by the construction of the PRI the projection $P_s$ commutes with each $P_\alpha$ (by Lemma~\ref{L:P_sP_A}(b)). Suppose that $\alpha<\kappa$ and $x\in (P_{\alpha+1}-P_\alpha)X$. Then
$$P_s^*x^*(x)=P_s^*x^*((P_{\alpha+1}-P_\alpha)x)=x^*(P_s(P_{\alpha+1}-P_\alpha)x)
=x^*((P_{\alpha+1}-P_\alpha)P_s x).$$
Hence, if $\alpha\in C$, we get
$$\begin{aligned}
P_s^*x^*(x)&=x^*((P_{\alpha+1}-P_\alpha)P_s x)=y^*_\alpha((P_{\alpha+1}-P_\alpha)P_s x)
=y^*_\alpha(P_s(P_{\alpha+1}-P_\alpha)x)\\&=P_s^*y^*_\alpha((P_{\alpha+1}-P_\alpha)x)=
y^*_\alpha((P_{\alpha+1}-P_\alpha)x)=x^*((P_{\alpha+1}-P_\alpha)x)=x^*(x).
\end{aligned}$$
If $\alpha\notin C$, then
$$
P_s^*x^*(x)=x^*((P_{\alpha+1}-P_\alpha)P_s x)=0=x^*((P_{\alpha+1}-P_\alpha)x)=x^*(x).
$$
Since $\bigcup_{\alpha<\kappa}(P_{\alpha+1}-P_\alpha)X$ is linearly dense, we deduce that $P_s^*x^*=x^*$.

(6)$\Rightarrow$(7) Let $(x_\alpha^*,x_\alpha)_{\alpha\in\Lambda}$ be the Markushevich basis provided by (6). It follows immediately that $D$ is a $\Sigma$-subspace, thus it is weak$^*$-countably closed by the already proved implication (1)$\Rightarrow$(2). Further, since the Markushevich basis is a biorthogonal system, obviously $x_\alpha^*\in D$ for any $\alpha\in\Lambda$. 

It remains to show that the set $H=\{x_\alpha\setsep \alpha\in\Lambda\}\cup\{0\}$ is $\sigma(X,D)$-Lindel\"of. So, let $\U$ be a cover of $H$ consisting of $\sigma(X,D)$-open sets. Then there is $U\in\U$ such that $0\in U$. By the definition of the topology $\sigma(X,D)$ there are $x_1^*,\dots,x_n^*\in D$ and $\varepsilon>0$ such that
$$\{x\in X\setsep \abs{x_j^*(x)}<\varepsilon\mbox{ for }j=1,\dots,n\}\subset U.$$
For each $j\in\{1,\dots,n\}$ the set
$$M_j=\{\alpha\in\Lambda\setsep x_j^*(x_\alpha)\ne 0\}$$
is countable and, moreover, $H\setminus\bigcup_{j=1}^n M_j\subset U$. So, $H\setminus U$ is countable and hence one can find a countable subfamily of $\U$ covering $H$.

(7)$\Rightarrow$(8) Let $(x_\alpha^*,x_\alpha)_{\alpha\in\Lambda}$ be the Markushevich basis provided by (7). Set $H=\{x_\alpha\setsep \alpha\in\Lambda\}\cup\{0\}$ and observe that all the nonzero points of $H$ are isolated. Indeed, let $\alpha\in\Lambda$. Since $x_\alpha^*\in D$, the set 
$$U_\alpha=\{x\in X\setsep x_\alpha^*(x)\ne0\}$$ 
is $\sigma(X,D)$-open and $U_\alpha\cap H=\{x_\alpha\}$.
Since $H$ is $\sigma(X,D)$-Lindel\"of, it follows that for each $\sigma(X,D)$-open neighborhood $U$ of $0$ the set $H\setminus D$ is countable. Therefore $H$ is a canonical continuous image of the space $L_\Lambda$, thus it is primarily Lindel\"of.

(8)$\Rightarrow$(9) Let $(x_\alpha^*,x_\alpha)_{\alpha\in\Lambda}$ be the Markushevich basis provided by (8). Set $H=\{x_\alpha\setsep \alpha\in\Lambda\}\cup\{0\}$. By assumption the set $H$ is primarily Lindel\"of in the topology $\sigma(X,D)$. Recall that primarily Lindel\"of spaces are preserved by taking closed subsets, countable products, continuous images and countable unions (see \cite[Proposition IV.3.4]{archang-kniha}). Further, compact metric spaces are primarily Lindel\"of
(as any nonempty compact metric space is a continuous image of $\{0,1\}^\en$ and the two-point discrete space is clearly primarily Lindel\"of). Therefore $\lin H$ is $\sigma(X,D)$-primarily Lindel\"of as $\lin H=\bigcup_n H_n$, where
$$H_n=\{t_1x_1+\dots+t_nx_n\setsep x_1,\dots, x_n\in H, \abs{t_j}\le n \mbox{ for }j=1,\dots,n\}$$
are $\sigma(X,D)$-primarily Lindel\"of, being a continuous image of
$$H^n\times[-n,n]^n \qquad (H^n\times\{t\in\ce\setsep \abs{t}\le n\}^n\mbox{ in the complex case}).$$
Further, the closed unit ball $B_X$ is $\sigma(X,D)$-closed as $D$ is $1$-norming, hence $B_X\cap\lin H$ is primarily Lindel\"of and thus  the product space
$$Z=(\lin H)\times(B_X\cap \lin H)^\en$$
is primarily Lindel\"of as well. Finally, the mapping $F:Z\to X$ defined by
$$F(x,(x_n)_{n=1}^\infty)=x+\sum_{n=1}^\infty 2^{-n} x_n$$
is well defined (the series converges absolutely in the norm) and maps $Z$ onto $X$ (as $\lin H$ is dense in $X$). So, to complete the proof it is enough to show that $F$ is continuous to the topology $\sigma(X,D)$. To this end it suffices to prove that
$x^*\circ F$ is continuous on $Z$ for each $x^*\in D$. But
$$(x^*\circ F)(x,(x_n)_{n=1}^\infty)=x^*(x)+\sum_{n=1}^\infty 2^{-n} x^*(x_n)=
\lim_{N\to\infty}(x^*(x)+\sum_{n=1}^N 2^{-n} x^*(x_n)),
$$
the partial sums are continuous on $Z$ and the limit is uniform on $Z$.

(9)$\Rightarrow$(1) This implication follows from \cite[Theorem 2.7]{ja-survey}.
\end{proof}

\begin{remark}\label{Remark Sigma}
(a) As remarked above, the content of Theorem~\ref{T:char sigma subspace} is not completely new. More precisely, the equivalence (1)$\Leftrightarrow$(2) is almost trivial, the equivalence (1)$\Leftrightarrow$(5) follows from the proof of \cite[Theorem 27]{kubis-skeleton}, the equivalences (1)$\Leftrightarrow$(6)$\Leftrightarrow$(9) are proved in \cite[Lemma 4.18]{ja-survey}.

The added value of Theorem~\ref{T:char sigma subspace} consists firstly in the assertions (3) and (4), secondly in a detailed analysis of the properties of Markushevich bases in the assertions (7) and (8) and, finally, in providing a proof of (1)$\Leftrightarrow$(5) avoiding the set-theoretical method of elementary submodels.

The assertions (3) and (4) provide another view on projectional skeletons which combine some approaches from \cite{CuFa-Asplund,CuFa-AsplundWCG} with an idea of \cite{cuth-sepdeterm}. For example, a  similar statement to the implication (1)$\Rightarrow$(3) is \cite[Lemma 11]{CuFa-AsplundWCG} where rich families are used instead of $\omega$-monotone mappings. In \cite{cuth-sepdeterm} the author shows the equivalence of separable reduction methods using rich families and $\omega$-monotone mappings. We show that $\omega$-monotone mappings can be used to characterize projectional skeletons as well. Another use of $\omega$-monotone mappings is demonstrated in \cite{CuFa-Asplund} by the use the notion of \emph{Asplund generator} to characterize Asplund spaces.

(b) We point out that the projectional skeleton constructed in the proof of (4)$\Rightarrow$(5) is automatically simple in the sense of \cite[Section 4]{cuka-cejm} (i.e., `indexed by the ranges of projections') and, moreover, its index set is a lattice.

(c) The assertion (6) can be strengthened by requiring that the Markushevich bases in question is moreover strong. 
The proof can be done by transfinite induction exactly in the same way as the proof of (5)$\Rightarrow$(6). Indeed, separable spaces admit a strong Markushevich basis by \cite{Terenzi} (see also \cite[Theorem 1.36]{HMVZ-biortogonal}) and this property is preserved in the induction step as remarked in the proof of \cite[Theorem 5.1]{HMVZ-biortogonal}.

(d) Observe that in the proofs of (6)$\Rightarrow$(7)$\Rightarrow$(8) the Markushevich basis has not been changed. Thus any Markushevich basis with the property from (6) has also the properties from (7) and (8). Moreover, if the basis satisfies the properties from (7), the set $H=\{x_\alpha\setsep\alpha\in\Lambda\}\cup\{0\}$ is $\sigma(X,D)$-closed and its nonzero points are isolated. The latter statement was proved above.  To see the first one, fix any $x\in X\setminus H$. We distinguish the following three cases:
\begin{itemize}
\item[$\bullet$] There are two distinct points $\alpha,\beta\in\Lambda$ such that $x_\alpha^*(x)\ne0\ne x_\beta^*(x)$. Then
$$\{y\in X\setsep x_\alpha^*(y)\ne0\ne x_\beta^*(y)\}$$
is a $\sigma(X,D)$-open set containing $x$ and disjoint with $H$.
\item[$\bullet$] There is exactly one $\alpha\in \Lambda$ such that $x_\alpha^*(x)\ne0$. Then
$x_\alpha^*(x)\ne1$ (otherwise $x=x_\alpha$ as the functionals $x_\beta^*$, $\beta\in\Lambda$ separate points of $X$) and hence
$$\{y\in X\setsep x_\alpha^*(y)\notin\{0,1\}\}$$
is a $\sigma(X,D)$-open set containing $x$ and disjoint with $H$.
\item[$\bullet$] $x_\alpha^*(x)=0$ for each $\alpha\in\Lambda$. Then $x=0\in H$, a contradiction.
\end{itemize}

(e) Let $(x_\alpha,x_\alpha^*)_{\alpha\in\Lambda}$ be a Markushevich basis with the properties from the assertion (6). Then $D$ is a $\Sigma$-subspace, as the set $M=\{x_\alpha\setsep\alpha\in\Lambda\}$ witnesses it. Therefore (1) is satisfied and, going through the proofs of (1)$\Rightarrow$(3)$\Rightarrow$(4)$\Rightarrow$(5) we can construct a commutative projectional skeleton $(P_s)_{s\in\Gamma}$ with induced subspace $D$. Moreover, by the construction, this skeleton has a special behavior on the basis. More precisely,
$$P_s x_\alpha = \begin{cases}
x_\alpha&\mbox{ if }x_\alpha\in P_sX,\\ 0&\mbox{ otherwise}.
\end{cases}\eqno{(\square)}$$
This behavior is specific for the commutative case. Indeed, suppose we have a projectional skeleton with induced subspace $D$ and a Markushevich basis such that $(\square)$ is satisfied. Then $D$ is the $\Sigma$-subspace generated by the set  $M=\{x_\alpha\setsep\alpha\in\Lambda\}$. To see this take any $x^*\in D$. Then there is $s\in\Gamma$ with $P_s^*x^*=x^*$. Then
$$P_s^*x^*(x_\alpha)=x^*(P_s x_\alpha)=0\mbox{ if }x_\alpha\notin P_sX.$$
Therefore
$$\{\alpha\in\Lambda\setsep x^*(x_\alpha)\ne0\}\subset\{\alpha\in\Lambda\setsep x_\alpha\in P_sX\}$$
and this set is countable as it is relatively discrete in the weak topology, hence, a fortiori, in the norm topology, and $P_sX$ is separable. 

Conversely, let $x^*\in X^*$ be such that the set
$$\Lambda_0=\{\alpha\in\Lambda\setsep x^*(x_\alpha)\ne0\}$$
is countable. By the properties of projectional skeletons there is $s\in\Gamma$ with $\{x_\alpha\setsep\alpha\in\Lambda_0\}\subset P_sX$. Then we have
$$P_s^*x^*(x_\alpha)=x^*(P_sx_\alpha)=\begin{cases}
x^*(x_\alpha)&\mbox{ if }x_\alpha\in P_sX,\\
x^*(0)=0=x^*(x_\alpha)&\mbox{ otherwise},
\end{cases}$$
hence $P_s^*x^*=x^*$, so $x^*\in D$.
\end{remark}

We continue by the following theorem. Given a projectional skeleton, it characterizes when the induced subspace is in fact a $\Sigma$-subspace.

\begin{thm}\label{T:komutativni} Let $X$ be a Banach space, $(P_s)_{s\in\Gamma}$ a projectional skeleton on $X$ and let $D$ be the induced subspace. The following assertions are equivalent.
\begin{enumerate}
\item $D$ is a $\Sigma$-subspace of $X^*$.
\item There is a $\sigma$-closed cofinal subset $\Gamma'\subset\Gamma$ such that $P_sP_t=P_tP_s$ for $s,t\in\Gamma'$.
\item There is a $\sigma$-closed cofinal subset $\Gamma'\subset\Gamma$ such that the projection $P_A$ (defined in \eqref{eq:P_A def}) is $\sigma(X,D)$-to-$\sigma(X,D)$ continuous for any directed subset $A\subset\Gamma'$.
\end{enumerate}
\end{thm}

Note that the assertion (1) can be replaced by any of its equivalents provided by Theorem~\ref{T:char sigma subspace} and the continuity requirement in (3) can be replaced by any of its equivalents from Proposition~\ref{P:sigma(X,D)-spojitost}.

An important tool in the proof of the theorem is the following lemma on uniqueness.

\begin{lemma}\label{L:spectral} Let $X$ be a Banach space and let $(P_s)_{s\in\Gamma}$ and $(Q_j)_{j\in J}$ be two projectional skeletons on $X$ inducing the same subspace $D\subset X^*$. Then for any choice of $s\in\Gamma$ and $j\in J$ there are $s'\in\Gamma$ and $j'\in J$ such that $s'\ge s$, $j'\ge j$ and $P_{s'}=Q_{j'}$.
\end{lemma}

\begin{proof} First observe that if $P$ is any bounded projection on $X$ with separable range, then $P^*X^*$ is weak$^*$-separable. Indeed, $Y=PX$ is separable and the projection $P$ can be expressed as $P=TQ$, where $T$ is the canonical isometric embedding of $Y$ into $X$ and $Q$ is the projection $P$ considered as an operator $X\to Y$. $Y^*$ is weak$^*$-separable (being the dual to a separable space), $Q^*$ is weak$^*$-to-weak$^*$ continuous (being an adjoint operator), so $Q^*Y^*$ is weak$^*$ separable. Further, $T^*:X^*\to Y^*$ is the restriction operator, so it is onto $Y^*$. It remains to observe that $P^*=Q^*T^*$ and hence $Q^*Y^*=P^*X^*$. 

Secondly, up to passing to cofinal $\sigma$-closed subsets of $\Gamma$ and $J$ we may assume without loss of generality that the projections from both skeletons are uniformly bounded \cite[Proposition 9]{kubis-skeleton} and hence the stronger condition (iii') holds (see the introductory section).
Let us define sequences $(s_n)$ in $\Gamma$ and $(j_n)$ in $J$ inductively as follows:
\begin{itemize}
\item $s_0=s$, $j_0=j$.
\item Given $s_{n-1}$ and $j_{n-1}$ defined, find $s_n\in\Gamma$, $s_n\ge s_{n-1}$ such that
$$P_{s_n}X\supset Q_{j_{n-1}}X\mbox{ and }P^*_{s_n}X^*\supset Q_{j_{n-1}}^*X^*.$$
This is possible by the properties of projectional skeletons, as $Q_{j_{n-1}}X$ is a separable subspace of $X$ and $ Q_{j_{n-1}}^*X^*$ is a weak$^*$-separable subspace of $D$.
\item In the same way, given $s_{n}$ and $j_{n-1}$ defined, find $j_n\in J$, $j_n\ge j_{n-1}$ such that
$$Q_{j_n}X\supset P_{s_{n}}X\mbox{ and }Q^*_{j_n}X^*\supset P_{s_{n}}^*X^*.$$
\end{itemize}

Finally, set $s'=\sup_n s_n$ and $j'=\sup_n j_n$. Then
$$P_{s'}X=\overline{\bigcup_n P_{s_n}X}=\overline{\bigcup_n Q_{j_n}X}=Q_{j'}X$$
due to the property (iii) of projectional skeletons and
$$P_{s'}^*X^*=\wscl{\bigcup_n P_{s_n}^*X^*}=\wscl{\bigcup_n Q_{j_n}^*X^*}=Q_{j'}^*X^*$$
by the property (iii') of projectional skeletons. Indeed, by the property (iii') we have,
given any $x^*\in X^*$ and $x\in X$,
$$P_{s_n}^*x^*(x)=x^*(P_{s_n}x)\to x^*(P_s x)=P_s^*x^*(x),$$
thus $P_{s_n}^*x^*\overset{w^*}{\to}P_s^*x^*$, and similarly for the other skeleton.

So, we have $P_{s'}X=Q_{j'}X$ and $P_{s'}^*X^*=Q_{j'}^*X^*$, so also
$$\ker P_{s'}=(P_{s'}^*X^*)_\perp=(Q_{j'}^*X^*)_\perp=\ker Q_{j'}.$$
Therefore, the projections $P_{s'}$ and $Q_{j'}$ have the same ranges and the same kernels,
thus they are equal.
\end{proof}

\begin{proof}[Proof of Theorem~\ref{T:komutativni}]
Note that $D$ is norming (by \cite[Proposition 9 and Section 4.3]{kubis-skeleton}).
Since the assertions (1)--(3) are not affected by renormings, we may assume without loss of generality that $D$ is $1$-norming. So, up to passing to a cofinal $\sigma$-closed subset of $\Gamma$ (this does not affect the assertions (1)--(3)) we may assume that $(P_s)_{s\in\Gamma}$ is a $1$-projectional skeleton.

(1)$\Rightarrow$(2) Suppose that $D$ is a $\Sigma$-subspace. By Theorem~\ref{T:char sigma subspace} It follows that there is a commutative $1$-projectional skeleton $(Q_j)_{j\in J}$ inducing $D$. Let 
$$\Gamma_0=\{s\in\Gamma\setsep \exists j\in J: P_s=Q_j\}.$$
By Lemma~\ref{L:spectral} we see that $\Gamma_0$ is cofinal in $\Gamma$. Observe that any cofinal set is automatically up-directed. So, it makes sense to define $\Gamma'=(\Gamma_0)_\sigma$ (using the notation from Section 2). Then clearly $\Gamma'$ is a cofinal $\sigma$-closed subset of $\Gamma$. We claim that $P_sP_t=P_tP_s$ whenever $s,t\in\Gamma'$.
To prove that we will use the transfinite construction of $(\Gamma_0)_\sigma$ described in the proof of Proposition~\ref{P:P_AX a ker P_A}(i). Let $\Gamma_\alpha$, $\alpha<\omega_1$, be the respective approximations or  $(\Gamma_0)_\sigma$. We will prove by transfinite induction that
$$\forall\alpha<\omega_1\forall s,t\in\Gamma_\alpha: P_sP_t=P_tP_s.$$
The validity for $\alpha=0$ follows from the definition of $\Gamma_0$ and commutativity of the skeleton $(Q_j)_{j\in J}$. Suppose it holds for some $\alpha<\omega_1$ and suppose $s,t\in \Gamma_{\alpha+1}$. Then there are increasing sequences (possibly constant) $(s_n)$ and $(t_n)$ in $\Gamma_\alpha$ with $s=\sup_n s_n$ and $t=\sup_n t_n$. Then, using the property (iii') of projectional skeletons, we deduce that for any $x\in X$ we have
$$P_sP_t x=\lim_n P_s P_{t_n}x=\lim_n\lim_m P_{s_m}P_{t_n}x=\lim_n\lim_m P_{t_n}P_{s_m}x
=\lim_n P_{t_n}P_s x=P_tP_sx,$$
thus $P_sP_t=P_tP_s$. Since the limit induction step is obvious, the proof is complete.

(2)$\Rightarrow$(3) Assume (2) holds. Let $A\subset\Gamma'$ be any directed subset. By Lemma~\ref{L:P_sP_A}(b) we see that $P_sP_A=P_AP_s$ for $s\in\Gamma'$. Thus, by Proposition~\ref{P:sigma(X,D)-spojitost}(5)$\Rightarrow$(2) we deduce that $P_A$ is $\sigma(X,D)$-to-$\sigma(X,D)$ continuous.

(3)$\Rightarrow$(1) Assume (3) holds. Without loss of generality assume that $\Gamma'=\Gamma$.
We are going to prove that the assertion (6) of Theorem~\ref{T:char sigma subspace} holds. This can be shown by repeating the proof of the implication (5)$\Rightarrow$(6) of Theorem~\ref{T:char sigma subspace} with few differences.

We use again transfinite induction. The first step, the separable case, is exactly the same.
In the induction step we build a PRI $(P_\alpha)_{\alpha\le\kappa}$ using Proposition~\ref{P:PRI}. Observe that any $P_\alpha$ is of the form $P_{A_\alpha}$ for some $A_\alpha\subset\Gamma$ directed. By the assumption the projection $P_A$ is $\sigma(X,D)$-to-$\sigma(X,D)$ continuous. So, by Proposition~\ref{P:sigma(X,D)-spojitost}(2)$\Rightarrow$(6) and Proposition~\ref{P:PRI}(c) the space $(P_{\alpha+1}-P_\alpha)X$ admits a $1$-projectional skeleton with induced subspace $D_\alpha^{\alpha+1}$ of the same form as in the proof of
Theorem~\ref{T:char sigma subspace}.

Finally, to be able to use the transfinite induction, it remains to show that the skeleton on
$(P_{\alpha+1}-P_\alpha)X$ satisfies the assumption of (3) as well. So, recall that the skeleton is of the form
$$P_s|_{(P_{\alpha+1}-P_\alpha)X},\quad s\in\Delta,$$
where $\Delta$ is a suitable cofinal $\sigma$-closed subset of $(A_{\alpha+1})_\sigma$. 

So, fix any directed set $B\subset\Delta$. Recall that $\Delta$ is chosen in such a way that $(P_{\alpha+1}-P_\alpha)X$ is invariant for $P_s$ for any $s\in\Delta$ (see Proposition~\ref{P:P_AX a ker P_A}(iv)).  It follows that $(P_{\alpha+1}-P_\alpha)X$
is invariant for $P_B$ as well. Thus,
$$P_B|_{(P_{\alpha+1}-P_\alpha)X}$$ 
is a projection on $(P_{\alpha+1}-P_\alpha)X$ and it is enough to show that
this projection is $\sigma({(P_{\alpha+1}-P_\alpha)X},D_\alpha^{\alpha+1})$-to-$\sigma({(P_{\alpha+1}-P_\alpha)X},D_\alpha^{\alpha+1})$ continuous. To this end we will check the validity of the property (1) of Proposition~\ref{P:sigma(X,D)-spojitost}. 
So, fix any $x^*\in D_\alpha^{\alpha+1}$. Then there is $y^*\in D$ such that
$y^*|_{{(P_{\alpha+1}-P_\alpha)X}}=x^*$. Fix any $x\in (P_{\alpha+1}-P_\alpha)X$.
Then
$$ \begin{aligned}
(P_B|_{(P_{\alpha+1}-P_\alpha)X})^*x^*(x)&=x^*(P_B|_{(P_{\alpha+1}-P_\alpha)X}x)=
x^*(P_Bx)=y^*(P_Bx)\\&=P_B^*y^*(x)=P_B^*y^*|_{(P_{\alpha+1}-P_\alpha)X}x.
\end{aligned}$$
Indeed, the first equality is just the definition of an adjoint mapping; the second one follows from the fact that $x\in (P_{\alpha+1}-P_\alpha)X$; the third one uses the invariance of $(P_{\alpha+1}-P_\alpha)X$ for $P_B$ and the choice of $y^*$; the fourth one is again the use of the definition of an adjoint mapping; and the last one follows from the fact that $x\in (P_{\alpha+1}-P_\alpha)X$. 

Finally, since  $P_B$ is $\sigma(X,D)$-to-$\sigma(X,D)$ continuous by the very assumption of (3) and $y^*\in D$, Proposition~\ref{P:sigma(X,D)-spojitost} yields that $P_B^*y^*\in D$. It follows that 
$$(P_B|_{(P_{\alpha+1}-P_\alpha)X})^*x^*=P_B^*y^*|_{(P_{\alpha+1}-P_\alpha)X}\in D_{\alpha}^{\alpha+1},$$
which completes the proof of the validity of the condition (1) from Proposition~\ref{P:sigma(X,D)-spojitost}, hence $P_B|_{(P_{\alpha+1}-P_\alpha)X}$ is $\sigma({(P_{\alpha+1}-P_\alpha)X},D_\alpha^{\alpha+1})$-to-$\sigma({(P_{\alpha+1}-P_\alpha)X},D_\alpha^{\alpha+1})$ continuous. This completes the proof.
\end{proof}

\begin{remark} An alternative proof of the implication (3)$\Rightarrow$(1) in Theorem~\ref{T:komutativni} may be done using \cite[Theorem 23]{kubis-skeleton}. Indeed, let us use transfinite induction on $\dens X$. The separable case is obvious, so assume that $\kappa$ is an uncountable cardinal and the statement holds whenever $\dens X<\kappa$. As above, without loss of generality assume $\Gamma'=\Gamma$. We build a PRI $(P_\alpha)_{\alpha\le\kappa}$ using Proposition~\ref{P:PRI}. Observe that any $P_\alpha$ is of the form $P_{A_\alpha}$ for some $A_\alpha\subset\Gamma$ directed. Given $\alpha<\kappa$, the family
$(P_s|_{P_\alpha X})_{s\in (A_\alpha)_\sigma}$ is a $1$-projectional skeleton on $P_\alpha X$ (by Proposition~\ref{P:P_AX a ker P_A}(iii))
and, due to Proposition~\ref{P:sigma(X,D)-spojitost}(2)$\Rightarrow$(6), the respective induced subspace is
$$D_\alpha=\{x^*|_{P_\alpha X}\setsep x^*\in D\}.$$
Moreover, given any directed $B\subset (A_\alpha)_\sigma$, the projection $P_B$ is $\sigma(X,D)$-to-$\sigma(X,D)$ continuous (by the assumptions of (3)), hence $P_B^*(D)\subset D$ by Proposition~\ref{P:sigma(X,D)-spojitost}. Similarly as in the above proof we show that $(P_B|_{P_\alpha}X)^*D_\alpha\subset D_\alpha$ and using Proposition~\ref{P:sigma(X,D)-spojitost} we deduce that $P_B|_{P_\alpha X}$ is $\sigma(P_\alpha X,D_\alpha)$-to-$\sigma(P_\alpha X,D_\alpha)$ continuous. Thus, using the assumption hypothesis, $D_\alpha$ is a $\Sigma$-subspace of $(P_\alpha X)^*$. By \cite[Theorem 23]{kubis-skeleton} we deduce that $D$ is contained in a $\Sigma$-subspace of $X^*$, thus $D$ itself is a $\Sigma$-subspace 
(as $D$ is weak$^*$-countably closed and any $\Sigma$-subspace is weak$^*$-countably tight).
\end{remark}

\begin{cor}\label{C:wld}
Let $X$ be a Banach space with a full projectional skeleton, i.e. having a projectional skeleton whose induced subspace is $X^*$. Then $X^*$ is a $\Sigma$-subspace of itself, i.e., $X$ is weakly Lindel\"of determined.
\end{cor}

\begin{proof}
This follows immediately from Theorem~\ref{T:komutativni}(3)$\Rightarrow$(1), as the topology $\sigma(X,D)$ now coincides with the weak topology on $X$ and any bounded linear operator is automatically weak-to-weak continuous.
\end{proof}

\begin{remark} There are some natural and widely studied subclasses of weakly Lindel\"of determined spaces, in particular weakly compactly generated spaces and their subspaces. An interesting line of research would be to try to characterize such classed by the existence of a special kind of a projectional skeleton. This problem is addressed in a forthcoming paper
\cite{FM-preprint}.
\end{remark}

\section{Equivalents of a projectional skeleton}\label{S:noncommutative}

In this section we study  characterizations of subspaces induced by a possibly non-commutative projectional skeleton. They are collected in Theorem~\ref{T:char induced} which can be viewed as a non-commutative version of Theorem~\ref{T:char sigma subspace}. However, as we will see, the analogy is not complete, some problems remain open. Before formulating the theorem we give the definitions of two more notions used in the statement or in the proof.

A topological space $T$ is called \emph{monotonically retractable} if there is an assignment
$$A\mapsto (r_A,\N(A)),\qquad A\in[T]^{\le\omega},$$
such that for any $A\in[T]^{\omega}$
\begin{itemize}
\item $r_A$ is a continuous retraction on $T$ such that $A\subset r_A(T)$;
\item $\N(A)$ is a countable network of the retraction $r_A$ (i.e., $\N(A)$ is a countable family of subsets of $T$ and for any open subset $U\subset T$ its preimage $r_A^{-1}(U)$ is the union of a subfamily of $\N(A)$);
\end{itemize}
and, moreover, the mapping $\N$ is $\omega$-monotone.

Further, a topological space $T$ is called \emph{monotonically Sokolov} if there is an assignment
$$\A\mapsto (r_{\A},\N(\A)),\qquad\A\in[\F(T)]^{\le\omega},$$
where $\F(T)$ denotes the family of all the nonempty closed subsets of $T$, such that for any $\A\in[\F(T)]^{\le\omega}$
\begin{itemize}
\item $r_{\A}$ is a continuous retraction on $T$ such that $r_{\A}(F)\subset F$ for $F\in\A$;
\item $\N(\A)$ is a countable outer network of $r_{\A}(T)$  (i.e., $\N(\A)$ is a countable family of subsets of $T$ and for any open subset $U\subset T$ and any $x\in U\cap r_{\A}(T)$ there is $N\in\N(\A)$ such that $x\in N\subset U$);
\end{itemize}
and, moreover, the mapping $\N$ is $\omega$-monotone.

Monotonically retractable spaces were introduced in \cite{RoHe-topproc}, monotonically Sokolov spaces in \cite{RoHe-sokolov}. Monotonically retractable spaces are closely related to retractional skeletons \cite{cuka-mr,RoHe-monotprop}, monotonically Sokolov spaces can be viewed, in a sense, as a non-commutative version of primarily Lindel\"of spaces (cf. the next theorem and the questions in the last section).

\begin{thm}\label{T:char induced}
Let $X$ be a Banach space and $D\subset X^*$ a norming subspace. The following assertions are equivalent.
\begin{enumerate}
\item $D$ is induced by a projectional skeleton in $X$.
\item There is an $\omega$-monotone mapping $\psi:[D\cup X]^{\le\omega}\to[D\cup X]^{\le\omega}$ such that for any $A\in[D\cup X]^{\le\omega}$ the following properties hold:
\begin{itemize}
\item[(i)] $A\subset\psi (A)$,
\item[(ii)] $\wscl{\psi(A)\cap D}\subset D$,
\item[(iii)] $\wscl{\psi(A)\cap D}$ is a linear space,
\item[(iv)] The mapping $x^*\mapsto x^*|_{\clin{\psi(A)\cap X}}$ is a bijection of $\wscl{\psi(A)\cap  D}$ onto  $(\clin{\psi(A)\cap X})^*$.
\end{itemize}
\item There is an $\omega$-monotone mapping $\theta:[X]^{\le\omega}\to[D\cup X]^{\le\omega}$ such that 
 $$D=\bigcup\{\wscl{\theta(A)\cap D}\setsep A\in[X]^{\le \omega}\}$$
and for any $A\in[X]^{\le\omega}$ the following properties hold:
\begin{itemize}
\item[(i)] $A\subset\theta(A)$,
\item[(ii)] $\wscl{\theta(A)\cap D}$ is a linear space,
\item[(iii)] The mapping $x^*\mapsto x^*|_{\clin{\theta(A)\cap X}}$ is a bijection of $\wscl{\theta(A)\cap  D}$ onto  $(\clin{\theta(A)\cap X})^*$.
\end{itemize}
\item $D$ is weak$^*$-countably closed and $(X,\sigma(X,D))$ is monotonically Sokolov.
\item $D$ is weak$^*$-countably closed and there is $M\subset X$ such that $\clin M=X$ and $(M,\sigma(X,D))$ is monotonically Sokolov.
\item  $D$ is weak$^*$-countably closed and $(X,\sigma(X,D))$ is a continuous image of a monotonically Sokolov space.
\end{enumerate}
\end{thm}

\begin{proof} Since the assertions are not affected by renorming, we may and shall assume that $D$ is $1$-norming.

(1)$\Rightarrow$(2) Let $(P_s)_{s\in\Gamma}$ be a projectional skeleton on $X$ such that the respective induced subspace in $D$. By Lemma~\ref{L:norma projekci} we can assume without loss of generality that it is a $1$-projectional skeleton. 

Let us fix a mapping $\sigma:D\cup X\to \Gamma$ such that
$P_{\sigma(x)}x=x$ for $x\in X$ and $P_{\sigma(x^*)}^* x^*=x^*$  for $x^*\in D$. Then the set-valued version $\sigma:[D\cup X]^{\le\omega}\to[\Gamma]^{\le\omega}$ (defined by $\sigma(A)=\{\sigma(a)\setsep a\in A\}$) is clearly $\omega$-monotone.

Further, fix a mapping $\phi:\Gamma\times\Gamma\to\Gamma$ such that $\phi(s,t)\ge s$ and  $\phi(s,t)\ge t$ for $s,t\in\Gamma$. For $A\in[\Gamma]^{\le\omega}$ we set
$A_0=A$, $A_n=A_{n-1}\cup\phi(A_{n-1}\times A_{n-1})$ and $\upsilon(A)=\bigcup_n A_n$.
It is clear that $\upsilon:[\Gamma]^{\le\omega}\to[\Gamma]^{\le\omega}$ is an $\omega$-monotone mapping and that $\upsilon(A)$ is an up-directed set containing $A$ for each $A\in[\Gamma]^{\le\omega}$.

We continue by choosing for each $s\in\Gamma$ a countable set $\eta(s)\subset D\cup X$ such that
$$\overline{\eta(s)\cap X}= P_s X\mbox{ and }\wscl{\eta(s)\cap D}=P_s^* X^*.
$$
This choice is possible as $P_s X$ is separable and $P_s^*X^*$ is weak$^*$-separable for each $s\in\Gamma$. Further, the set-valued version $\eta:[\Gamma]^{\le\omega}\to[D\cup X]^{\le\omega}$ is clearly $\omega$-monotone.

Finally, for $A\in[X\cup D]^{\le\omega}$  set 
$$\psi_0(A)= A\cup\eta(\upsilon(\sigma(A))).$$
Clearly $\psi_0$ is an $\omega$-monotone mapping. For $n\in\en$ and $A\in[X\cup D]^{\le\omega}$ set $\psi_n(A)=\psi_0(\psi_{n-1}(A))$ and $\psi(A)=\bigcup_n\psi_n(A)$. 

It is clear that $\psi$ is $\omega$-monotone. The property (i) is obvious, The property (ii) follows from the fact that $D$ is induced by a skeleton and hence weak$^*$-countably closed.

It remains to prove the properties (iii) and (iv). To this end fix $A\in[X\cup D]^{\le\omega}$ and set $C=\upsilon(\sigma(\psi(A)))$. Then $C$ is a countable up-directed subset of $\Gamma$, hence $C$ has a supremum $s\in \Gamma$. We claim that
$$\overline{\psi(A)\cap X}=P_sX\mbox{ and }\wscl{\psi(A)\cap D}=P_s^*X^*.$$
The inclusions `$\subset$' are in both cases obvious. Further, given any $t\in C$ there is some $n\in\en$ with $t\in\upsilon(\sigma(\psi_n(A)))$, thus $\eta(t)\subset \psi_{n+1}(A)\subset\psi(A)$, hence also $\lin_\qe(\eta(t)\cap D)\cup\lin_\qe(\eta(t)\cap X)\subset \psi(A)$. It follows that 
$$\begin{aligned}
P_t X&=\overline{\eta(t)\cap X}\subset\overline{\psi(A)\cap X},
\\
P_t^* X^*&=\wscl{\eta(t)\cap D}\subset\wscl{\psi(A)\cap D}.
\end{aligned}$$
Next observe that
$$P_s X=\overline{\bigcup_{t\in C} P_t X}\mbox{ and }P_s^* X^*=\wscl{\bigcup_{t\in C} P_t^* X^*}.$$
The property (iii) now easily follows.
To prove the property (iv) fix $x^*\in P_s^*X^*$. Clearly $\norm{x^*|_{P_sX}}\le\norm{x^*}$. Conversely, for each $x\in B_X$ we have
$$\abs{x^*(x)}=\abs{P_s^*x^*(x)}=\abs{x^*(P_sx)}\le \norm{x^*|_{P_sX}},$$
so the respective assignment is an isometry, thus it is one-to-one. It is also onto, as for any $y^*\in(P_sX)^*$ we have $x^*=y^*\circ P_s\in X^*$, $P_s^*x^*=x^*$ and $x^*|_{P_sX}=y^*$.

(2)$\Rightarrow$(3) Let $\psi$ be the mapping provided by (2). Further, for each $x\in X$ let $\nu(x)\subset D\cap B_{X^*}$ be a countable set with $\norm{x}=\sup\{\abs{x^*(x)}\setsep x^*\in \nu(x)\}$. It is clear that $\nu$, considered as a set-valued map $[X]^{\le\omega}\to[D]^{\le\omega}$, is $\omega$-monotone. For each $A\in  [X]^{\le\omega}$ set $\theta(A)=\psi(A\cup\nu(A))$. Then $\theta$ is an $\omega$-monotone map. The properties (i)--(iii) follow immediately from the  properties of $\psi$.

It remains to prove the formula for $D$. To this end set
$$D'=\bigcup\{\wscl{\theta(A)\cap D}\setsep A\in[X]^{le\omega}\}.$$
By the property (ii) of $\psi$ we get $D'\subset D$. Further, $D'$ is a linear subspace (by (iii) as $[X]^{\le\omega}$ is directed and $\theta$ is $\omega$-monotone), $D'$ is $1$-norming (as $\nu(x)\subset D'$ for any $x\in X$). $D'$ is also weak$^*$-countably closed. Indeed, let $C=\{x_n^*\setsep n\in\en\}\subset D'$. For each $n$ fix $A_n\in[X]^{\le\omega}$ with $x^*\in\wscl{\theta(A_n)\cap D}$. Set $A=\bigcup_n A_n$. Then clearly $\wscl{C}\subset \wscl{\theta(A)\cap D}\subset D'$.

To show that $D'=D$ it is now sufficient to prove that $D\cap B_{X^*}$ has countable tightness in the weak$^*$-topology. So, fix $x^*\in D\cap B_{X^*}$ and $M\subset D\cap B_{X^*}$ with $X^*\in\wscl M$. We perform the following inductive construction.

We start by setting $A_1=\psi(\{x^*\})$. Given $A_n$ we proceed as follows.
\begin{itemize}
\item Enumerate $A_n\cap X=\{x^n_k\setsep k\in\en\}$.
\item Find $x_n^*\in M$ such that $\abs{x^*(x_k^l)-x_n^*(x_k^l)}\le\frac1n$ for $k,l\le n$.
\item Let $A_{n+1}=\psi(A_n\cup\{x_n^*\})$.
\end{itemize}
Set $A=\bigcup_{n}A_n$. Then $\psi(A)=A$, $x^*\in A$ and $x_n^*\in A$ for $n\in\en$. Further,
the construction yields $x_n^*(x)\to x^*(x)$ for $x\in A\cap X$. Since the sequence $(x_n^*)$ is bounded, we conclude that  $x_n^*(x)\to x^*(x)$ for $x\in\clin(A\cap X)$. We claim that $x_n^*\overset{w^*}{\longrightarrow} x^*$. Indeed, since $(x_n^*)$ is bounded, it has some weak$^*$-cluster point, say $y^*$. Then $y^*\in\wscl{A\cap D}$ and $y^*|_{\clin(A\cap X)}=x^*|_{\clin(A\cap X)}$. Hence $y^*=x^*$ (by the property (iv) of $\psi$). It follows that $x^*$ is the unique cluster point of $(x_n^*)$, so $x_n^*\overset{w^*}{\longrightarrow} x^*$.

(3)$\Rightarrow$(1) Let $\theta$ be the mapping provided by (3). For $A\in [X]^{\le\omega}$ set $D_A=\wscl{D\cap\theta(A)}$ and $X_A=\clin(X\cap \theta(A))$. By the property (iv) and Lemma~\ref{L:jednaprojekce} there is a bounded linear projection $P_A:X\to X$ with $P_A X=X_A$ and $P_A^*X^*=D_A$. If $A,B\in[X]^{\le\omega}$ are such that $A\subset B$, then $X_A\subset X_B$ and $D_A\subset D_B$,
so $\ker P_A=(D_A)_\perp \supset (D_B)_\perp=\ker P_B$. It follows that $P_AP_B=P_BP_A=P_A$. Thus $(P_A)_{A\in[X]^{\le\omega}}$ is a projectional skeleton on $X$. Indeed, the properties (i) and (iv) are obvious, the property (ii) has been just proved and the property (iii) follows from  $\omega$-monotonicity of $\theta$. Moreover, the subspace induced by this skeleton is exactly $D$ by the property (ii) of $\theta$. This completes the proof.

(1)$\Rightarrow$(4)  Firstly, $D$ is  weak$^*$ countably closed being induced by a skeleton.
To prove that $(X,\sigma(X,D))$ is monotonically Sokolov we shall construct the respective mappings using similar ideas as in the proof of (1)$\Rightarrow$(2).

 Let $(P_s)_{s\in\Gamma}$ be a projectional skeleton on $X$ such that the respective induced subspace in $D$. By Lemma~\ref{L:norma projekci} we can assume without loss of generality that it is a $1$-projectional skeleton. Let $\phi:\Gamma\times\Gamma\to \Gamma$ and $\upsilon:[\Gamma]^{\le\omega}\to[\Gamma]^{\le\omega}$ be the mappings defined in the proof of (1)$\Rightarrow$(2). 
 
For any $x\in X$ let $\sigma(x)\in\Gamma$ be such that $P_{\sigma(x)}x=x$ (it is a restriction of the mapping $\sigma$ from (1)$\Rightarrow$(2)).  
Further, for any nonempty $F\subset X$ let $\alpha(F)$ be an element of $F$ and, if $F\subset X$ is a norm-separable subset, let $\beta(F)$ be a countable dense subset of $F$ (note that $\beta(\emptyset)=\emptyset$).

For any $s\in\Gamma$ set
$$\N_0(s)=\left\{U\left(x,\frac1n\right) \setsep x\in \lin\qe(\eta(s)\cap X), n\in\en \right\},
$$
where $U(x,r)$ denotes the open ball centered at $x$ with radius $r$ (in the norm of $X$).
It is clear that $\N_0(s)$ is a countable family of subsets of $X$ and that the set-valued version of $\N_0$, considered as a mapping from $[\Gamma]^{\le\omega}\to[\P(X)]^{\le\omega}$ is $\omega$-monotone.
Let $\F(X)$ denote the family of all the nonempty $\sigma(X,D)$-closed subsets of $X$. Let us define by induction $\omega$-monotone mappings $\phi_n: [\F(X)]^{\le\omega}\to[X]^{\le\omega}$.

Start by setting
$$\phi_1(\A)=\{\alpha(F)\setsep F\in\A\},\qquad \A\in[\F(X)]^{\le\omega}.$$
It is clear that $\phi_1$ is $\omega$-monotone. Further, given an $\omega$-monotone mapping $\phi_n:[\F(X)]^{\le\omega}\to[X]^{\omega}$, we set
$$\Gamma_n(\A)=\upsilon(\sigma(\phi_n(\A)))\mbox{ and }t_n(\A)=\sup\Gamma_n(\A) 
\mbox{ for }\A\in[\F(X)]^{\le\omega}.$$
The mapping $\Gamma_n$ is an $\omega$-monotone mapping $[\F(X)]^{\le\omega}\to[\Gamma]^{\le\omega}$. Moreover, $\Gamma_n(\A)$ is directed, so $t_n(\A)$ is well defined. Further, set
$$\phi_{n+1}(\A)=\bigcup\{\beta(P_s(X))\setsep s\in\Gamma_n(\A)\}
 \cup\bigcup\{\beta(P_s F\setminus F)
\setsep s\in\Gamma_n(\A), F\in \A \},\qquad \A\in[\F(X)]^{\le\omega}.
$$
Note that the range of each $P_{s}$ is norm-separable, hence the formula has a sense. Moreover, the mapping $\phi_{n+1}$ is $\omega$-monotone. It follows that also the mappings
$$\phi(\A)=\bigcup_n\phi_n(\A) \mbox{ and }\Gamma(\A)=\bigcup_n\Gamma_n(\A), \qquad \A\in[\F(X)]^{\le\omega}$$
are $\omega$-monotone. 

Since $\Gamma(\A)$ is a directed countable subset of $\Gamma$, $t(\A)=\sup \Gamma(\A)$ is well defined. Then $r_\A=P_{t(\A)}$ is a bounded linear projection on $X$. Since $r_\A^*X^*\subset D$, $r_\A$ is $\sigma(X,D)$-to-$\sigma(X,D)$ continuous (in fact, even 
$\sigma(X,D)$-to-weak continuous). Moreover, $r_\A(F)\subset F$ for $F\in\A$. Indeed, let
$F\in\A$ and $x\in F$. Assume that $r_\A(x)\notin F$. Then there is some $s\in\Gamma(A)$ such that $P_sx\notin F$. Further, there is some $n\in\en$ with $s\in\Gamma_n(\A)$. But then
$$x\in P_sF\setminus F\subset \overline{\beta(P_sF\setminus F)}\subset \overline {\phi_{n+1}(\A)}
\subset\overline{\phi(A)}=r_{\A} X,$$
so $r_\A x=x\in F$, a contradiction.

Finally, set $\N(\A)=\bigcup_{s\in\Gamma(\A)}\N_0(s)$. Then $\N(\A)$ is an outer network for $r_\A X$ and the assignment $\N$ is $\omega$-monotone. This completes the proof.

(4)$\Rightarrow$(5) This implication is trivial.

(5)$\Rightarrow$(6) The class of monotonically Sokolov spaces is stable to taking closed subsets and countable products \cite[Theorem 3.4(c,d)]{RoHe-sokolov}. Further, compact metric spaces are monotonically Sokolov for trivial reasons. It follows that the class of continuous images of monotonically Sokolov spaces is stable to the same operations and, moreover, to taking  continuous images and countable unions. Indeed, a countable union is a continuous image of a countable topological sum and monotonically Sokolov spaces are obviously stable to taking countable topological sums. 
Therefore the proof can be done by copying the proof of the implication (8)$\Rightarrow$(9) of Theorem~\ref{T:char sigma subspace}.

(6)$\Rightarrow$(1) Fix a monotonically Sokolov space $T$ and a continuous surjection $F:T\to (X,\sigma(X,D))$. By \cite[Theorem 3.5]{RoHe-sokolov} the space $\C_p(T)$ is monotonically retractable. Define the mapping $G:D\to\C_p(T)$ by $G(x^*)=x^*\circ F$. Then $G$ is clearly weak$^*$-to-$\tau_p$ continuous. Since $F$ is onto, $G$ is one-to-one. Further, $D$ is weak$^*$-countably closed, hence $D\cap B_{X^*}$ is weak$^*$-countably compact. It follows that $G(D\cap B_{X^*})$ is closed in $\C_p(T)$ (by \cite[Fact 2.1(h)]{cuka-mr}) and $G|_{D\cap B_{X^*}}$ is a homeomorphic embedding (by \cite[Fact 2.1(h)]{cuka-mr} it is a closed mapping). Hence $D\cap B_{X^*}$ is monotonically retractable and so it admits a full retractional skeleton
by \cite[Theorem 1.1]{cuka-mr} (see \cite[Theorem 4.3]{RoHe-monotprop} for an elementary proof). By \cite[Theorem 3.4(a)]{RoHe-sokolov} the space $(X,\sigma(X,D))$ is Lindel\"of, hence \cite[Lemma 5.3]{cuka-mr} shows that $B_{X^*}$ is the \v{C}ech-Stone compactification of $D\cap B_{X^*}$. Thus $D\cap B_{X^*}$ is induced by a retractional skeleton in $B_{X^*}$ by \cite[Proposition 4.5]{cuka-mr}, so $D$ is induced by a projectional skeleton on $X$ by \cite[Lemma 5.2]{cuka-mr}.
\end{proof}

\begin{remark}\label{remark induced}
(a) The previous theorem can be viewed as a noncommutative analogue of Theorem~\ref{T:char sigma subspace}. But the analogy is not complete and there are some differences. Firstly, noncommutative analogues of the assertion (1) and (2) of Theorem~\ref{T:char sigma subspace} are missing. Indeed, there is up to now no known analogue of (1). As for (2), existence of a projectional generator is a sufficient for condition for the existence of a projectional skeleton, but it is not clear whether it is necessary. Related problems are discussed in the last section.

(b) Assertions (2) and (3) of the previous theorem can be viewed as non-commutative analogues of the assertions (3) and (4) of Theorem~\ref{T:char sigma subspace}. However, there are some important differences. In Theorem~\ref{T:char sigma subspace} there is some linearly dense subset $M\subset X$ and the respective $\omega$-monotone mappings have values in $[M]^{\le\omega}$, while in Theorem~\ref{T:char induced} the values are in $[X\cup D]^{\le\omega}$. This features has implications also for constructing a projectional skeleton from an $\omega$-monotone mapping. While in the commutative case the projections are easily determined by a subset of $M$, in the non-commutative cases we need a pair of subsets -- a subset of $X$ and a subset of $D$. It is not clear, whether the
projections can be determined in a canonical way merely by a subset of $X$ in the non-commutative case as well.

(c) The previous theorem contains no analogue of the  assertions (6)--(8) of Theorem~\ref{T:char sigma subspace}. The reason is that a spaces admitting a projectional skeleton admits a Markushevich basis (by Theorem~\ref{T:M-basis}) but it is not clear whether such a basis has some canonical relationship to the induced subspace. This is illustrated by concrete examples in Section~\ref{S:examples} and the related open problems are discussed in the last section. The only known relationships of a Markushevich basis and a subspace induced by a skeleton are contained in the following two results.

(d) The assertions (4)--(6) of the previous theorem can be viewed as a noncommutative analogue of the assertion (9) of Theorem~\ref{T:char sigma subspace}. Monotonically Sokolov spaces (or, more precisely, their continuous images) serve as a noncommutative analogue of primarily Lindel\"of spaces.
Some more discussion on the relationship of these two classes is contained in the last section.

(e) The equivalence (1)$\Leftrightarrow$(4) from the previous theorem has been already known, it is proved in \cite[Theorem 1.5]{cuka-mr}. This equivalence is refined by adding the conditions (5) and (6). Moreover, the proof of (1)$\Rightarrow$(4) is done directly, unlike in \cite{cuka-mr}.

(f) As a consequence of the Theorem~\ref{T:char induced} we get that Theorem 18 (and hence Corollary 19) of \cite{kubis-skeleton} is true
in spite of the gap in the proof pointed out in the proof of Lemma~\ref{L:tightness} above. Indeed, assume that $D$ is a subspace of $X^*$ induced by a projectional skeleton on $X$. By Theorem~\ref{T:char induced} we get that $(X,\sigma(X,D))$ is monotonically Sokolov,
hence $\C_p(X,\sigma(X,D))$ is monotonically retractable by  \cite[Theorem 3.5]{RoHe-sokolov}, so it has countable tightness by \cite[Fact 2.1(g)]{cuka-mr}. Since $(D,w^*)$ is homeomorphic to a subset of $\C_p(X,\sigma(X,D))$, it has countable tightness as well.
\end{remark}

The next corollary is one of the promised results on the relationship of Markushevich bases and projectional skeletons. It is an immediate consequence of the implication (5)$\Rightarrow$(1) 
of Theorem~\ref{T:char induced}.

\begin{cor}
Let $X$ be a Banach space and $D\subset X^*$ a norming weak$^*$-countably closed subset.
If there is a Markushevich basis $(x_\alpha,x_\alpha^*)_{\alpha\in\Lambda}$ of $X$ such that
the set $\{x_\alpha\setsep \alpha\in \Lambda\}\cup\{0\}$ is monotonically Sokolov in the topology $\sigma(X,D)$, then $D$ is induced by a projectional skeleton on $X$.
\end{cor}

Let us point out that it is not clear whether the converse implication holds as well. This problem is discussed in more detail in the last section.
The second result is the following improvement of the assertion (3) of Theorem~\ref{T:char induced}.

\begin{prop}
Let $X$ be a Banach space with a Markushevich basis $(x_\alpha,x^*_\alpha)_{\alpha\in\Lambda}$. Assume that $D\subset X^*$ is a norming subspace induced by a projectional skeleton on $X$. Then there is an $\omega$-monotone mapping $\varphi:[\Lambda]^{\le\omega}\to[\Lambda\cup D]^{\le\omega}$ such that
$$D=\bigcup\{\wscl{\varphi(A)\cap D}\setsep A\in[\Lambda]^{\le \omega}\}$$
and for any $A\in[\Lambda]^{\le\omega}$ the following properties hold:
\begin{itemize}
\item[(i)] $A\subset\varphi(A)$,
\item[(ii)] $\wscl{\varphi(A)\cap D}$ is a linear space,
\item[(iii)] The mapping $x^*\mapsto x^*|_{\clin{\{x_\alpha\setsep \alpha\in\theta(A)\cap \Lambda\}}}$ is a bijection of the subspace $\wscl{\theta(A)\cap  D}$ onto the space $(\clin{\{x_\alpha\setsep \alpha\in\theta(A)\cap \Lambda\}})^*$.
\end{itemize}
\end{prop}

\begin{proof} Let $\theta:[X]^{\le\omega}\to[X\cup D]^{\le\omega}$ be the mapping from Theorem~\ref{T:char induced}(3). We will modify it using the Markushevich basis. To this end we define one more mapping. For any $x\in X$ let $C(x)\subset\Lambda$ be a countable set such that $x\in\clin\{x_\alpha\setsep \alpha\in C(x)\}$. Further, for any $A\in[\Lambda]^{\le\omega}$ we set
$$\zeta_1(A)=A\cup C(\theta(\{x_\alpha\setsep \alpha\in A\})\cap X)$$
and, by induction, define
$$\zeta_{n+1}(A)=\zeta_1(\zeta_n(A)),\qquad n\in\en.$$
Then the mapping
$$\zeta(A)=\bigcup_{n\in\en}\zeta_n(A),\qquad A\in[\Lambda]^{\le\omega}$$
is $\omega$-monotone.
Finally, set
$$\varphi(A)=\zeta(A)\cup(\theta(\{x_\alpha\setsep\alpha\in\zeta(A)\})\cap D), \qquad A\in [\Lambda]^{\le\omega}.$$
It is clear that $\varphi$ is an $\omega$-monotone mapping such that $A\subset \varphi(A)$ and $\wscl{\varphi(A)\cap D}$ is a linear space for each $A\in[\Lambda]^{\le\omega}$. 

To prove the property (iii) of $\varphi$ fix $A\in[\Lambda]^{\le\omega}$. Then 
$$\clin\{x_\alpha\setsep \alpha\in\zeta(A)\}=\clin(\theta(\{x_\alpha\setsep\alpha\in\zeta(A)\})\cap X).$$
Indeed, the inclusion $\subset$ is obvious. To see the converse observe that
$$\begin{aligned}
\theta(\{x_\alpha\setsep\alpha\in\zeta(A)\})\cap X&=\bigcup_n\theta(\{x_\alpha\setsep\alpha\in\zeta_n(A)\})\cap X 
\subset\bigcup_n \clin \{x_\alpha\setsep \alpha\in C(\theta(\{x_\alpha\setsep\alpha\in\zeta_n(A)\})\cap X)\}\\&
\subset\bigcup_n \clin \{x_\alpha\setsep \alpha\in \zeta_{n+1}(A)\} \subset\clin\{x_\alpha\setsep \alpha\in\zeta(A)\}.\end{aligned}
$$
So, the property (iii) of $\varphi$ follows immediately from the property (iii) of $\theta$.

It remains to prove the formula for $D$. The inclusion $\supset$ follows from the properties of $\theta$. One possibility to prove the converse inclusion is to observe that the set on the right-hand side is a weak$^*$-countably closed subspace which separates points of $X$, hence it is weak$^*$-dense.
Since $D$ has countable tightness in the weak$^*$ topology by Remark~\ref{remark induced}(f), the conclusion follows.
\end{proof}

\begin{remark} In the previous proposition no special assumption on the Markushevich basis is needed. Just a mere existence of some Markushevich basis is used. In fact, a similar statement can be formulated for an arbitrary linearly dense subset of $X$ in place of $\{x_\alpha\setsep \alpha\in\Lambda\}$. The proof of such a statement would be essentially the same.

However, the fact that we start with a Markushevich basis can be used to construct a simple projectional skeleton (in the sense of \cite[Section 4]{cuka-mr}) by applying the method of the proof of the implication (3)$\Rightarrow$(1) of Theorem~\ref{T:char induced} to the mapping $\varphi$ in place of $\theta$. Again, the only important thing is the existence of some Markushevich basis (this corresponds to the methods of \cite{cuka-mr}).
\end{remark}

\section{Examples of spaces with a noncommutative projectional skeleton}\label{S:examples}

While $1$-Plichko spaces, i.e., spaces with a commutative $1$-projectional skeleton appear often and naturally in mathematics (see \cite{ja-val-exa,BHK-vN,BHK-vN,triples}), the supply of spaces with a non-commutative skeleton is not so large. Up to now they include spaces of continuous functions on ordinal segments, spaces of continuous functions on certain trees equipped with the coarse-wedge topology and duals to Asplund spaces. And, of course, spaces made by certain standard constructions starting from the mentioned examples. 
In this section we provide an analysis of the three mentioned classes.
We focus on explicit description of projectional skeletons, Markushevich bases and projectional generators on these spaces. The related open problems are discussed in the last section.  We also show the applications of Theorem~\ref{T:komutativni} in these cases.

Since two of these classes are spaces of continuous functions, in the first subsection we recall some notions and facts on retractions on compact spaces. 

\subsection{Retractions on compact spaces} If $K$ is a compact Hausdorff space, $\C(K)$ denotes the space of (real- or complex-valued) continuous functions on $K$ equipped with the supremum norm. Its dual $\C(K)^*$ is, by the Riesz representation theorem, canonically isometric to $\M(K)$, the space of (real- or complex-valued) Radon measures on $K$ equipped with the total variation norm. In the sequel we will identify $\C(K)^*$ with $\M(K)$.

An analogue of projectional skeleton in the realm of compact spaces is the notion of a retractional skeleton introduced in \cite{kubis-michalewski}. We recall that a \emph{retractional skeleton} on a compact Hausdorff space $K$ is a family $(r_s)_{s\in\Gamma}$ of continuous retractions on $K$ indexed by a $\sigma$-complete up-directed partially ordered set satisfying the following conditions.
\begin{itemize}
\item[(i)] $r_s(K)$ is metrizable for each $s\in\Gamma$;
\item[(ii)] $r_s\circ r_t=r_t\circ r_s=r_s$ whenever $s,t\in\Gamma$ are such that $s\le t$;
\item[(iii)] if $(s_n)$ is an increasing sequence in $\Gamma$ and $s=\sup_n s_n$, then $r_s(x)=\lim_n r_{s_n}(x)$ for $x\in K$;
\item[(iv)] $\lim_{s\in \Gamma}r_s(x)=x$ for $x\in K$.
\end{itemize}
Further, the set
$$S=\bigcup_{s\in\Gamma}r_s(K)$$
is said to be \emph{induced by the skeleton}.

A notion related to a $\Sigma$-subspace is that of a dense $\Sigma$-subset. Recall that $A$ is a \emph{$\Sigma$-subset} of a compact space $K$ if there is a homeomorphic injection $h:K\to\er^\Gamma$ such that
$$A=\{x\in K\setsep\{\gamma\in\gamma\setsep h(x)(\gamma)\ne0\}\mbox{ is countable}\}.$$
A compact space having a dense $\Sigma$-subset is called \emph{Valdivia}. If $K$ is even a $\Sigma$-subset of itself, it is called \emph{Corson}. By \cite[Theorem 6.1]{kubis-michalewski}
(more precisely by its proof) a dense subset of $K$ is a $\Sigma$-subset if and only if it is induced by a commutative retractional skeleton.

Next we recall few fact on the relationship of retractions on $K$ with projections on $\C(K)$. We start by the following well-known result.

\begin{lemma}\label{L:compact retrakce} Let $K$ be a compact Hausdorff space and let $r:K\to K$ be a continuous retraction. Define the operator $P:\C(K)\to\C(K)$ by $Pf=f\circ r$, $f\in\C(K)$. Then the following assertions hold.
\begin{itemize}
\item[(a)] $P$ is a linear projection of norm one.
\item[(b)] The mapping
$f\mapsto f|_{r(K)}$ is an isometric isomorphism of $P(\C(K))$ onto $\C(r(K))$. I.e., it is a linear onto isometry, which moreover preserves multiplication and in the complex case also complex conjugation.
\item[(c)] The adjoint projection $P^*$ satisfies the formula
$$P^*(\mu)=r(\mu),\quad \mu\in\M(K)=\C(K)^*,$$
where $r(\mu)$ is the image of the measure $\mu$ under the mapping $r$, i.e.,
$$P^*(\mu)(B)=\mu(r^{-1}(B)),\quad B\subset K\mbox{ Borel}, \mu\in \M(K).$$
Moreover 
$$P^*(\M(K))=\{\mu\in\M(K)\setsep \abs{\mu}(K\setminus r(K))=0\}.$$
\end{itemize}
\end{lemma}

\begin{proof} The assertions (a) and (b) are well known and obvious. The assertion (c) is also easy and known, let us give a proof for completeness.

Fix $\mu\in \M(K)$. Then $r(\mu)$ is a well-defined Borel measure on $K$. Moreover, $r(\mu)$ is a Radon measure -- this is obvious in case $\mu$ is nonnegative; any real-valued measure is a difference of two non-negative ones and any complex-valued measure is a linear combination of four non-negative ones. So, $r(\mu)\in M(K)$. Moreover, by the rule of integration with respect to the image measure we have, for any $f\in\C(K)$,
$$\int f\di r(\mu)=\int f\circ r\di\mu=\int P(f)\di\mu=\int f\di P^*(\mu),$$
thus $P^*(\mu)=r(\mu)$.

Finally, let us show the last equality. Let $\mu\in M(K)$ and $B\subset K\setminus r(K)$ be a Borel set. Then
$$P^*(\mu)(B)=r(\mu)(B)=\mu(r^{-1}(B))=\mu(\emptyset)=0,$$
so $\abs{P^*(\mu)}(K\setminus r(K))=0$. 

Conversely, assume that $\mu$ belongs to the set on the right-hand side.
Then for any $B\subset K$ Borel we have
$$\begin{aligned}
r(\mu)(B)&=\mu(r^{-1}(B))=\mu(r^{-1}(B\cap r(K)))=\mu(r^{-1}(B\cap r(K))\cap r(K))\\&=\mu(B\cap r(K))=\mu(B),\end{aligned}$$
so $\mu=r(\mu)=P^*(\mu)$.
\end{proof}

\begin{lemma}\label{L:compact-projekce} Let $K$ be a compact Hausdorff space and let $(r_s)_{s\in\Gamma}$ be a net of continuous retractions on $K$ which pointwise converges to a continuous retraction $r$ and, moreover,
$$r_s\circ r_t=r_t\circ r_s=r_s\mbox{ whenever }s\le t.$$
Define the projections 
$$P(f)=f\circ r \mbox{ and }P_s(f)=f\circ r_s, \quad f\in\C(K), s\in\Gamma.$$
Then
$$P_sP_t=P_tP_s=P_s\mbox{ whenever }s\le t$$
and, moreover, the net $(P_s)$ converges to $P$ in the strong operator topology.
\end{lemma}

\begin{proof} The equalities $P_sP_t=P_tP_s=P_s$ for $s\le t$ are obvious. Further, it is clear that for any $f\in \C(K)$ and any $x\in X$
$$P_s(f)(x)=f(r_s(x))\to f(r(x))=P(f)(x),$$
So $P_s(f)\to P(f)$ pointwise. It remains to show that this can be strengthened to the norm convergence. 

To this end set 
$$A=\bigcup_{s\in\Gamma}P_s(\C(K)).
$$
By Lemma~\ref{L:compact retrakce}(b) we know that each $P_s(\C(K))$ is an algebra containing constant functions and stable to complex conjugation in the complex case. Since $A$ is a directed union of such algebras, it is an algebra with the same properties. Further, it is a subalgebra of $P(\C(K))$ and it separates points of $r(K)$. Indeed, if $x,y\in r(K)$ are different, then there is $s\in\Gamma$ with $r_s(x)\ne r_s(y)$. By the Urysohn lemma there is $g\in\C(r_s(K))$ with $g(r_s(x))\ne g(r_s(y))$. Then $g\circ r_s\in A$ and separates $x$ and $y$. So, the Stone-Weierstrass theorem (together with Lemma~\ref{L:compact retrakce}(b) applied to $r$) we see that $A$ is norm dense in $P(\C(K))$. 

To complete the proof fix $f\in\C(K)$ and $\varepsilon>0$. By the previous paragraph there is $g\in A$ with $\norm{Pf-g}<\varepsilon$. Fix some $s\in\Gamma$ with $g=P_s(g)$. Then for each $t\in\Gamma$, $t\ge s$ we have
$$\norm{Pf-P_tf}\le\norm{Pf-g}+\norm{g-P_tf}
<\varepsilon+\norm{P_t(g-Pf)}\le\varepsilon+\norm{g-Pf}<2\varepsilon,$$
where we used the equalities $P_t g=P_tP_s g= P_sg=g$  and $P_tP=P$.
This completes the proof.
\end{proof}

The first part of the assertion (a) of the following proposition is stated in \cite[Proposition 28]{kubis-skeleton}. It is claimed there that the assertion is clear. We add an easy proof using Lemma~\ref{L:compact-projekce}.

\begin{prop}\label{P:rs->ps} Let $K$ be a compact Hausdorff space and let $(r_s)_{s\in\Gamma}$ be a retractional skeleton on $K$. Denote by $S$ the respective induced subset of $K$.

Define $P_s(f)=f\circ r_s$ for $f\in\C(K)$ and $s\in\Gamma$. Then the following hold.
\begin{itemize}
\item[(a)] $(P_s)_{s\in\Gamma}$ is a $1$-projectional skeleton on $\C(K)$ and the respective induced subspace is
$$D=\{\mu\in\M(K)\setsep \spt\mu\mbox{ is a separable subset of }S\}.$$
\item[(b)] If the skeleton $(r_s)_{s\in\Gamma}$ is commutative, then so is the skeleton $(P_s)_{s\in\Gamma}$.
\item[(c)] If $D$ is a $\Sigma$-subspace, then there is a cofinal $\sigma$-closed subset $\Gamma'\subset \Gamma$ such that $r_s\circ r_t=r_t\circ r_s$ for $s,t\in\Gamma'$. So, in particular, $S$ is induced by a commutative retractional skeleton on $K$, hence it is a $\Sigma$-subspace of $K$.
\end{itemize}
\end{prop}

\begin{proof}
(a) Let us check the properties (i)--(iv) of projectional skeletons. Given $s\in\Gamma$, $P_s(\C(K))$ is isometric to $\C(r_s(K))$ (by Lemma~\ref{L:compact retrakce}(b)), so it is separable (as $r_s(K)$) is metrizable. Hence the property (i) is fulfilled. The properties (ii) and (iii)  (in fact (iii')) follow from the respective properties of a retractional skeleton using Lemma~\ref{L:compact-projekce}. Further, by the property (iv) or retractional skeletons and Lemma~\ref{L:compact-projekce} it follows that 
$$f=\lim_{s\in\Gamma} P_s f,\quad f\in\C(K).$$
So, given $f\in\C(K)$ one can find an increasing sequence $(s_n)$ in $\Gamma$ such that $\norm{P_{s_n}f-f}<\frac1n$.  Let $s=\sup_n s_n$. Since $P_{s_n}f\to P_s f$, necessarily $P_sf=f$. This completes the proof of the property (iv) of projectional skeletons. By Lemma~\ref{L:compact retrakce}(a) it is even a $1$-projectional skeleton.

Let $D$ be the subspace of $\M(K)$ induced by the skeleton. If $\mu\in D$, then there is $s\in \Gamma$ with $P_s^*\mu=\mu$. By Lemma~\ref{L:compact retrakce}(c) the support of $\mu$ is contained in $r_s(K)$, so it is a separable subset of $S$. Conversely, suppose that $\spt\mu$ is a separable subset of $S$. Fix a countable dense set $C\subset \spt\mu$. Then there is $s\in\Gamma$ such that $r_s(x)=x$ for $x\in C$. It follows that $\spt\mu\subset r_s(K)$, thus by Lemma~\ref{L:compact retrakce}(c) we deduce $\mu\in P_s^*\M(K)\subset D$. This completes the proof of the formula for $D$.

The assertion (b) is obvious, the assertion (c) follows from Theorem~\ref{T:komutativni} (and \cite[Theorem 6.1]{kubis-michalewski}).
\end{proof}

\subsection{Spaces of continuous functions on ordinal segments}

Let $\eta$ be an ordinal. Then  the ordinal segment $[0,\eta]$ equipped with the order topology is a compact Hausdorff space. Thus $\C([0,\eta])$ is a Banach space and its dual $\C([0,\eta])^*$ is canonically isometric to $\M([0,\eta])$. Since $[0,\eta]$ is scattered, any Radon measure on $[0,\eta]$ is countably supported, so $\M([0,\eta])$ is isometric to the Banach space $\ell^1([0,\eta])$. Anyway, we will use the measure notation since it reflects the topological structure of $[0,\eta]$. Moreover, by $I(\eta)$ we will denote the set of all the isolated ordinals from $[0,\eta]$.
 
The following lemma describes natural retractions on $[0,\eta]$.

\begin{lemma}\label{L:retrakce-ord} Let $A\subset[0,\eta]$ be a closed subset containing $0$.
The following assertions are equivalent.
\begin{enumerate}
\item Any isolated point of $A$ is an isolated ordinal.
\item $A\cap I(\eta)$ is a dense subset of $A$.
\item The mapping $r_A$ defined by
$$r_A(x)=\max ([0,x]\cap A),\quad x\in [0,\eta],$$
is a continuous retraction of $[0,\eta]$ onto $A$.
\end{enumerate}
\end{lemma}

\begin{proof} Let us first remark that the mapping $r_A$ from (3) is well defined as $A$ is closed and contains $0$. Clearly, it is a retraction of $[0,\eta]$ onto $A$, so the key content of (3) is the continuity of $r_A$.

\smallskip

(1)$\Rightarrow$(2) Let $U\subset [0,\eta]$ be an open set such that $U\cap A\ne\emptyset$. Set $x=\min (U\cap A)$. Then $[0,x]\cap U$ is an open set and
$\{x\}=[0,x]\cap U\cap A$,
thus $x$ is an isolated point of $A$. By (1) it is an isolated ordinal. Thus $U\cap A$ intersects $I(\eta)$, which completes the proof of (2).

(2)$\Rightarrow$(3) As remarked above, $r_A$ is a well-defined retraction of $[0,\eta]$ onto $A$. Hence, we are going to prove it is continuous. It is clearly continuous at each isolated ordinal. So, assume $x\in [0,\eta]$ is a limit ordinal and let us show that $r_A$ is continuous at $x$.

Let $U$ be any neighborhood of $r_A(x)$. By the definition of the order topology there is some $y<r_A(x)$ such that $(y,r_{A}(x)]\subset U$. Since $r_A(x)\in A$ and $(y,r_A(x)]$ is an open neighborhood of $r_A(x)$,  there is some $z\in(y,r_A(x)]\cap A\cap I(\eta)$. Then $[z,x]$ is an open neighborhood of $x$ and 
$$r_A([z,x])\subset[z,r_A(x)]\subset U.$$

(3)$\Rightarrow$(1) Let us proceed by contraposition. Assume (1) fails, hence there is an isolated point $x\in A$ which is a limit ordinal. Then
$y=\sup (A\cap[0,x))<x$
and $y\in A$ as $A$ is closed. Thus $r_A(x)=x$ and $r_A(z)=y$ for $z\in[y,x)$, which shows that $r_A$ is not continuous at $x$.
\end{proof}

The family of subsets of $[0,\eta]$ satisfying the equivalent conditions of the previous lemma is very important for the study of retractions on $[0,\eta]$. Therefore we denote it by $\A(\eta)$. I.e., we set
\begin{equation}\label{eq:A(eta)}
\A(\eta)=\{A\subset[0,\eta]\setsep A\mbox{ is closed}, 0\in A\ \&\ I(\eta)\cap A\mbox{ is dense in }A \}.\end{equation}
It is clear that the family $\A(\eta)$ is closed to taking finite unions, so it is up-directed by inclusion. We continue by investigating its properties. The following lemma is trivial.

\begin{lemma}\label{L:kompatibilita-ord} Let $A,B\in\A(\eta)$ be such that $A\subset B$. Then $r_A\circ r_B=r_B\circ r_A=r_A$.
\end{lemma}

The following lemma establishes a continuity-like property of the family $\A(\eta)$.

\begin{lemma}\label{L:spojitost-ord} Let $\A'\subset\A(\eta)$ be a nonempty subset up-directed by  inclusion. Then $$B=\overline{\bigcup\A'}\in\A(\eta)$$ and, moreover,
$$r_B(x)=\lim_{A\in\A'}r_A(x),\qquad x\in [0,\eta].$$
\end{lemma}

\begin{proof}
It is clear that $B\in \A(\eta)$. It remains to prove the equality. To this end fix any $x\in [0,\eta]$ and any $U$, a neighborhood of $r_B(x)$. By the definition of the order topology there is $y<r_B(x)$ such that $(y,r_B(x)]\subset U$. Since $r_B(x)\in B$ and $(y,r_B(x)]$ is an open neighborhood of $r_B(x)$, there is some
$z\in(y,r_B(x)]\cap\bigcup\A'$. Fix $A_0\in A'$ with $z\in A_0$. Then for any $A\in\A'$ such that $A\supset A_0$ we have 
$$z\in A\cap[0,x]\subset B\cap [0,x]\subset[0,r_B(x)]$$ and hence
$$r_A(x)\in[z,r_B(x)]\subset(y,r_B(x)]\subset U,$$
which completes the proof.
\end{proof}

Now we are ready to describe a retractional skeleton on $[0,\eta]$. In case $\eta$ is a cardinal number, this was done in \cite[Example 6.4]{kubis-michalewski}. The proof for general ordinals is the same.

\begin{prop}\label{P:ordinal-rs} Let $\A_\omega(\eta)$ denote the family of all the countable sets from $\A(\eta)$. Then $(r_A)_{A\in\A_\omega(\eta)}$ is a retractional skeleton on $[0,\eta]$. 
Moreover, the subset of $[0,\eta]$ induced by the skeleton is
$$S(\eta)=\bigcup_{A\in\A_\omega(\eta)}r_A([0,\eta])=\{x\in[0,\eta]\setsep x\mbox{ is isolated or limit with countable cofinality}\}.$$
\end{prop}

\begin{proof}
It is clear that $\A_\omega(\eta)$ is closed to taking finite unions, so it is up-directed by inclusion. Each $r_A$ is a continuous retraction by Lemma~\ref{L:retrakce-ord}. Let us prove the properties (i)--(iv) of a retractional skeleton.

We have $r_A([0,\eta])=A$, which is a countable compact, hence metrizable. This proves the property (i). The property (ii) follows from Lemma~\ref{L:kompatibilita-ord}, the property (iii) from Lemma~\ref{L:spojitost-ord} (using the fact that the closure of a countable set of ordinals is countable). The property (iv) follows from Lemma~\ref{L:spojitost-ord} applied to $\A'=\A_\omega(\eta)$ as clearly $\bigcup\A_\omega(\eta)$ is dense in $[0,\eta]$ (it contains all the isolated ordinals).

Finally, the subset induced by the skeleton is
$$S(\eta)=\bigcup_{A\in\A_\omega(\eta)}r_A([0,\eta])=\bigcup\A_\omega(\eta).$$
Then $S(\eta)$ contains no ordinal of uncountable cofinality. Indeed suppose that there is some $A\in\A_\omega(\eta)$ containing some $x$ of uncountable cofinality.
Since $A$ is countable, $x$ is an isolated point of $A$, so by the definition of $\A(\eta)$ it must be an isolated ordinal, which is a contradiction.

Conversely, if $x\in[0,\eta]$ is an isolated ordinal, then
$\{0,x\}\in\A_\omega(\eta)$, hence $x\in S(\eta)$. Finally, assume that $x$ is a limit ordinal of countable cofinality. Then there is a strictly increasing sequence $(x_n)$ of ordinals with supremum $x$. Then
$$\{0,x\}\cup\{x_n+1\setsep n\in\en\}\in\A_\omega(\eta),$$
hence $x\in S(\eta)$.
\end{proof}

Let us continue by investigation of the associated projections on $\C([0,\eta])$.
For $A\in\A(\eta)$ we define the projection $P_A$ on $\C([0,\eta])$
by
$$P_A(f)=f\circ r_A,\quad f\in\C([0,\eta]).$$
By Lemma~\ref{L:compact retrakce} we know that it is a norm-one projection.

\begin{prop}\label{P:ord-ps} \ 
\begin{itemize}
\item[(a)] $(P_A)_{A\in\A_\omega(\eta)}$ is a $1$-projectional skeleton on $\C([0,\eta])$. The respective induced subspace of the dual is
$$D(\eta)=\{\mu\in\M([0,\eta])\setsep \mu(\{x\})=0\mbox{ for each }x\in[0,\eta]\mbox{ with uncountable cofinality}\}.$$
\item[(b)] Let $\A'\subset\A_\omega(\eta)$ be up-directed. Then the projection $P_{\A'}$ defined by \eqref{eq:P_A def} coincides with the projection $P_{\overline{\bigcup\A'}}$ defined above.
\end{itemize}
\end{prop}

\begin{proof}
The assertion (a) follows immediately from Proposition~\ref{P:ordinal-rs} and Proposition~\ref{P:rs->ps}; the assertion 
(b) follows from Lemma~\ref{L:spojitost-ord} and Lemma~\ref{L:compact-projekce}.
\end{proof}

Next we are going to characterize ordinals $\eta$ for which $\C([0,\eta])$ is $1$-Plichko. This is not a new result (see the comments in the proof) but we wish to provide a proof using Theorem~\ref{T:komutativni}. To this end we first need to characterize $\sigma(\C([0,\eta]),D(\eta))$-continuity of the projection $P_A^*$. This is done in the following easy lemma.

\begin{lemma}\label{L:ord-sigma(X,D) spoj}
Let $A\in\A(\eta)$. Then the projection $P_A^*$ is $\sigma(\C([0,\eta]),D(\eta))$-to-$\sigma(\C([0,\eta]),D(\eta))$ continuous if and only if 
$$\forall x\in A: (x<\eta \ \&\ \cf(x)\ge\omega_1) \Rightarrow x+1\in A.$$
\end{lemma}

\begin{proof}
Using Proposition~\ref{P:ord-ps}(b) and Proposition~\ref{P:sigma(X,D)-spojitost} we see that $P_A^*$ is  $\sigma(\C([0,\eta])$-to-$\sigma(\C([0,\eta])$ continuous if and only if $P_A^*(D(\eta))\subset D(\eta)$. This is in turn equivalent to the inclusion $r_A(S(\eta))\subset S(\eta)$. Indeed, if $P_A^*(D(\eta))\subset D(\eta)$,
then, in particular, $P_A^*(\delta_x)\in D(\eta)$ for each $x\in S(\eta)$. Since
$P_A^*(\delta_x)=\delta_{r_A(x)}$, necessarily $r_A(x)\in S(\eta)$. Conversely, assume that $r_A(S(\eta))\subset S(\eta)$. Let $\mu\in D$. Then
$\mu$ is supported by a countable set $C\subset S(\eta)$. Since clearly $P_A^*(\mu)=r_A(\mu)$ is supported by $r_A(C)$, which is a countable subset of $S(\eta)$,
we deduce that $P_A^*(\mu)\in D(\eta)$.

Finally, it follows from the definition of $r_A$ and from the description of $S(\eta)$ that the inclusion $r_A(S(\eta))\subset S(\eta)$ is equivalent to the condition given in the statement.
\end{proof}

Now we are ready to present the promised characterization of $1$-Plichko spaces of the form $\C([0,\eta])$. 

\begin{thm} The following assertions are equivalent.
\begin{enumerate}
\item $\C([0,\eta])$ is $1$-Plichko.
\item $D(\eta)$ is a $\Sigma$-subspace.
\item $\eta<\omega_2$.
\end{enumerate}
\end{thm}

\begin{proof}
The implication (2)$\Rightarrow$(1) follows immediately from the respective definitions. 
The implication (1)$\Rightarrow$(3) follows from  \cite[Theorem 5.3 and Example 1.10(ii)]{ja-survey}. The implication (3)$\Rightarrow$(2) follows from \cite[Proposition 3.7(ii) and Proposition 5.1]{ja-survey} using moreover \cite[Lemmata 1.6 and 1.7]{ja-survey}. So the proof is complete. However, we will give an alternative proof using Theorem~\ref{T:komutativni}.

\smallskip

(1)$\Rightarrow$(2) Suppose $\C([0,\eta])$ is $1$-Plichko. Let $D'$ be a $1$-norming $\Sigma$-subspace. Since $D'$ is $1$-norming and $D'\cap B_{\C([0,\eta])^*}$ is weak$^*$-countably compact, we have
$$\norm{f}=\max\{\abs{x^*(f)}\setsep x^*\in D',\norm{x^*}\le 1\}.$$
So, applying to characteristic functions of isolated points, we deduce that $D'$ contains $\delta_x$ for any isolated point $x\in [0,\eta]$. So, $D(\eta)\cap D'$ is $1$-norming, thus $D(\eta)=D'$ (by Lemma~\ref{L:tightness}(b)).

(2)$\Rightarrow$(3) We will use Theorem~\ref{T:komutativni}. Assume that $\eta\ge\omega_2$ and let $\A'\subset\A_\omega(\eta)$ be a cofinal $\sigma$-closed subset. For any $\alpha<\omega_1$ let us choose some $A_\alpha\in A'$ by the following procedure.
\begin{itemize}
\item[(i)] Let $A_0\in\A'$ be arbitrary.
\item[(ii)] Assume that $\alpha<\omega_1$ and $A_\alpha$ has been chosen. Find $A_{\alpha+1}\in \A'$ such that 
$A_{\alpha+1}\supset A_\alpha\cup\{\max(A_\alpha\cap[0,\omega_2))+1\}$. This is possible as $\A'$ is cofinal.
\item[(iii)] Assume that $\lambda<\omega_1$ is limit and $A_\alpha$, $\alpha<\lambda$, have been chosen. Set
$$A_\lambda=\overline{\bigcup_{\alpha<\lambda}A_\alpha}.$$
Then $A_\lambda\in\A'$, as $\A'$ is $\sigma$-closed. 
\end{itemize}
Then 
$\A_0=\{A_\alpha\setsep \alpha<\omega_1\}$ is a directed subset of $\A'$
(in fact, it is linearly ordered). Set $A=\overline{\bigcup\A_0}$. Then 
$\max(A\cap[0,\omega_2))$ has uncountable cofinality, thus $P_A^*$ is not 
$\sigma(\C([0,\eta]),D(\eta))$-to-$\sigma(\C([0,\eta]),D(\eta))$ continuous by Lemma~\ref{L:ord-sigma(X,D) spoj}. Hence, $D(\eta)$ is not a $\Sigma$-subspace by Theorem~\ref{T:komutativni}.

(3)$\Rightarrow$(2)  We will use again Theorem~\ref{T:komutativni}. If $\eta\le\omega_1$, then the family $\A_\omega$ itself witnesses that $D$ is a $\Sigma$-subspace (using Theorem~\ref{T:komutativni}). If $\eta>\omega_1$, then the whole family does not work, we need to restrict to a cofinal $\sigma$-closed subfamily. 

To this end fix a bijection $\xi: I(\omega_1)\to I(\eta)$ such that $\xi(0)=0$ and set
$$A_\alpha=\overline{\xi(I(\omega_1)\cap [0,\alpha])},\qquad \alpha<\omega_1.$$
Then $(A_\alpha)_{\alpha<\omega_1}$ is a strictly increasing transfinite sequence in $\A_\omega$. Moreover, the family $\{A_\alpha\setsep\alpha<\omega_1\}$ is clearly a cofinal $\sigma$-closed subset of $\A_\omega$. Since it is linearly ordered, the respective projections commute, hence the assertion (2) of Theorem~\ref{T:komutativni} is fulfilled.

(Note that the validity of the assertion (3) of Theorem~\ref{T:komutativni}
is in this case also obvious due to the characterization from Lemma~\ref{L:ord-sigma(X,D) spoj}.)
\end{proof}

We continue by describing a canonical Markushevich basis and a projectional generator on $\C([0,\eta])$.

\begin{prop}\label{P:ord Mbaze ocasni} \ 
\begin{itemize}
\item[(a)] The family $(g_\alpha,\nu_\alpha)_{\alpha\in I(\eta)}$, where
$$\nu_\alpha=\chi_{[\alpha,\eta]}, \qquad \nu_\alpha=\begin{cases}
\delta_0&\mbox{ if }\alpha=0,\\
\delta_{\alpha}-\delta_{\alpha-1} &\mbox{ if }\alpha\ge 1,
\end{cases}
$$
is a strong Markushevich basis of $\C([0,\eta])$.
\item[(b)] Let $H=\{0\}\cup\{g_\alpha\setsep \alpha\in I(\eta)\}$. Then $H$ is $\sigma(\C([0,\eta]),D(\eta))$-closed. The accumulation points of $H$ are elements $g_\alpha$, where $\alpha\in I(\eta)$ is such that $\alpha-1$ has uncountable cofinality; and, in case $\eta$ is a limit ordinal of uncountable cofinality, the zero function.
\item[(c)] $P_A(H)\subset H$ for each $A\in\A_\omega$. More precisely, for any $A\in\A_\omega$ we have
$$ P_A(0)=0, \quad P_A(g_\alpha)=\begin{cases}
g_\alpha & \mbox{ if }\alpha\in A,\\
g_\beta & \mbox{ if }[\alpha,\eta]\cap A\ne\emptyset\mbox{ and }\beta=\min[\alpha,\eta]\cap A,\\
0 & \mbox{ if }[\alpha,\eta]\cap A=\emptyset.
\end{cases}$$
\item[(d)] For any $x\in[0,\eta]$ limit of countable cofinality choose a countable set $C(x)\subset I(\eta)$  with supremum $x$. For any $\mu\in D(\eta)$ define
$$\begin{aligned}
\Phi(\mu)=\{g_x\setsep & x\in I(\eta), \mu(\{x\})\ne0\}\\&\cup\bigcup
\big\{ \{g_\alpha\setsep \alpha\in C(x)\}\setsep
x\in[0,\eta]\mbox{ limit with countable cofinality},\mu(\{x\})\ne0\big\}.\end{aligned}$$
Then $(D(\eta),\Phi)$ is a projectional generator.
\end{itemize}
\end{prop}

\begin{proof}
(a) Let us check the properties defining a Markushevich basis. The first property -- biorthogonality -- is obvious.  Let us continue by the third property, i.e., let us show that the family $(\nu_\alpha)_{\alpha\in I(\eta)}$ separates points of $\C([0,\eta])$. To this end fix any $f\in\C([0,\eta])$ and assume $\nu_\alpha(f)=0$  for each $\alpha\in I(\eta)$. Then $f(0)=0$ and $f(\alpha)=f(\alpha+1)$ for any $\alpha<\eta$. So, using moreover continuity of $f$, we see that $f$ is a constant zero function.

The second property follows from the stronger property defining strong Markushevich bases. Fix $f\in\C([0,\eta])$. Set
$$A=\{\alpha\in I(\eta)\setsep \nu_\alpha(f)\ne0 \}, \ 
M=\{g_\alpha \setsep \alpha\in A\}.$$
The proof will be complete if we show that $f\in\clin A$. This will be done by the Hahn-Banach theorem. Let $\mu\in\M([0,\eta])$ be such that $\mu|_M=0$. We will show that $\mu(f)=0$ as well.
If $f=0$, the conclusion is obvious. So, suppose that $f$ is not the constant  zero function and set 
$$J=\left\{\alpha\in[0,\eta]\setsep \int_{[0,\alpha]}f\di\mu=0\ \&\ \mu((\alpha,\eta])=0\right\}.$$
Our aim is to show that $\eta\in J$. The first step is to show that $J\ne\emptyset$ as
$$\beta=\max\{\alpha\in[0,\eta]\setsep f|_{[0,\alpha]}=f(0)\}\in J.$$
$\beta$ is well defined as $f$ is continuous. If $f(0)=0$, then $\beta<\eta$ and $\beta+1\in A$, so $g_{\beta+1}\in M$. Thus $\beta\in J$.

If $f(0)\ne0$, then $0\in A$, so $\mu([0,\eta])=0$. If $\beta=\eta$, this implies that $\beta\in J$.
If $\beta<\eta$, then $\beta+1\in A$, so $\mu([\beta+1,\eta])=0$. It follows that also $\mu([0,\beta])=0$. Since $f$ is constant on $[0,\beta]$, clearly $\beta\in J$.

Next we set $\gamma=\sup J$. We distinguish two cases:

\begin{itemize}
\item[(i)] $\gamma\in J$. If $\gamma=\eta$, the proof is finished. So, assume $\gamma<\eta$.
Set 
$$\beta=\max\{\alpha\in(\gamma,\eta]\setsep f|_{[\gamma+1,\beta]}=f(\gamma+1)\}.$$
$\beta$ is well defined as $f$ is continuous. Then $\mu((\beta,\eta])=0$ (this is trivial if $\beta=\eta$, in case $\beta<\eta$ we use that $\beta+1\in A$). Hence $\mu((\gamma,\beta])=0$, so clearly $\beta\in J$, a contradiction with the choice of $\gamma$. 

\item[(ii)] $\gamma\notin J$. Then $\gamma$ is limit. Moreover,
$$\int_{[0,\gamma)}f=0\mbox{ and }\mu([\gamma,\eta])=0.$$
Indeed, if $\gamma$ has countable cofinality, it follows from the $\sigma$-additivity of $\mu$. Assume $\gamma$ has uncountable cofinality. Since $\mu$ is countably supported, there is $\gamma'<\gamma$ such that $\mu|_{(\gamma',\gamma)}=0$. Since there is $\alpha\in J\cap (\gamma',\gamma)$, the above equalities follow.

Now we can proceed similarly as in the case (i). Set 
$$\beta=\max\{\alpha\in[\gamma,\eta]\setsep f|_{[\gamma,\beta]}=f(\gamma)\}.$$
Again, $\mu((\beta,\eta])=0$, hence $\mu([\gamma,\beta])=0$. It follows that $\beta\in J$, a contradiction.
\end{itemize}

(b) Before proceeding to the proof of (b) we will prove a lemma comparing two topologies on $\C([0,\eta])$. It will be used also later. Recall that $\tau_p(S(\eta))$ is the topology of pointwise convergence on $S(\eta)$.

\begin{lemma}\label{L:ord dve topologie}
The topology $\tau_p(S(\eta))$ is weaker than the topology $\sigma(\C([0,\eta]),D(\eta))$. On bounded subsets of $\C([0,\eta])$ the two topologies coincide. Moreover, the norm-closed unit ball of $\C([0,\eta])$ is $\tau_p(S(\eta))$-closed.
\end{lemma}

\begin{proof}
The first statement is obvious and the third one follows from the density of $S(\eta)$ in $[0,\eta]$. To show the second statement fix
a bounded set $M\subset\C([0,\eta])$ and consider the identity mapping
$$\iota:(M,\tau_p(S(\eta)))\to (M,\sigma(\C([0,\eta]),D(\eta)).$$
To show tat $\iota$ is continuous, it is enough to show that $\mu\circ\iota$ is $\tau_p(S)$-continuous for each $\mu\in D(\eta)$. So, fix $\mu\in D(\eta)$. Then $\mu=\sum_{j=1}^\infty c_j\delta_{x_j}$ for some points $x_j\in S(\eta)$ and a summable sequence of scalars $(c_j)$. Then for each $f\in M$ we have
$$\mu\circ\iota (f)=\sum_{j=1}^\infty c_j f(x_j).$$
Now, the partial sums of this series are $\tau_p(S(\eta))$-continuous and the series converges uniformly on $M$ (as $M$ is bounded and the sequence $(c_j)$ is summable). This completes the proof.
\end{proof}

Let us continue the proof of the assertion (b). By Lemma~\ref{L:ord dve topologie} we can work with the topology $\tau_p(S(\eta))$. First observe that $H$ is $\tau_p(S)$-closed. Indeed, a continuous function belongs to $H$ if and only if it is non-decreasing and attains only values $0$ and $1$. Since $S(\eta)$ is dense in $[0,\eta]$ we have
$$H=\{f\in\C([0,\eta])\setsep \forall x\in S(\eta): f(x)\in\{0,1\}\ \&\ 
\forall x,y\in S(\eta): x<y\Rightarrow f(x)\le f(y)\}.$$
This formula obviously implies that $H$ is $\tau_p(S(\eta))$-closed.

Further, let us describe the accumulation points of $H$. 
The function $g_0$ is an isolated point of $H$ as 
$$\{g_0\}=H\cap \{f\in\C([0,1])\setsep f(0)\ne0\}.$$
If $\alpha$ is an isolated ordinal or a limit ordinal with countable cofinality, then both $\alpha$ and $\alpha+1$ belong to $S(\eta)$ and
$$U=\{f\in\C([0,\eta])\setsep f(\alpha)\ne 1\ \&\ f(\alpha+1)\ne0\}$$
is a  $\tau_p(S(\eta))$-open set with $U\cap H=\{g_{\alpha+1}\}$, so $g_{\alpha+1}$ is an isolated point of $H$.

If $\alpha$ has uncountable cofinality, then $g_{\alpha+1}$ is the
$\tau_p(S)$-limit of the net $(g_\beta)_{\beta\in I(\eta)\cap[0,\alpha]}$.
 
Finally, it is clear that $0$ is an isolated point of $H$ if and only if $\eta\in S(\eta)$. 

(c) This assertion is obvious.

(d) Since any measure on $[0,\eta]$ is countably supported, $\Phi$ is clearly countably valued. Further, take any $M\subset D$ and assume that there is some
$$\mu\in\clin^{w^*} M\cap \Phi(M)^\perp\setminus\{0\}.$$
Since $\mu\ne 0$, by (a) there is some $\alpha\in I(\eta)$ such that $\mu(g_\alpha)\ne0$, i.e., $\mu([\alpha,\eta])\ne0$. Fix such $\alpha$ and let 
$$\gamma=\min\{\beta\in[\alpha,\eta]\setsep \mu([\alpha,\beta]\ne0\}.$$
Since $[\alpha,\gamma]$ is a clopen set and $\mu\in\clin^{w^*}M$, there is some $\nu\in M$ with $\nu([\alpha,\gamma])\ne0$. Further, let $\zeta\in[\alpha,\gamma]$ be the smallest element with $\nu(\{\zeta\})\ne0$. Since $\nu\in D$, necessarily $\zeta\in S(\eta)$. There are two possibilities.

Either $\zeta\in I(\eta)$. Then $g_\zeta\in \Phi(M)$ and so $\mu([\zeta,\eta])=0$. It follows that $\zeta>\alpha$ and, moreover, $\mu([\alpha,\zeta-1])=\mu([\alpha,\eta])-\mu([\zeta,\eta])\ne0$. Since $\zeta-1\in[\alpha,\gamma)$, it is a contradiction with the choice of $\gamma$.

Or $\zeta$ is limit. Then $\zeta>\alpha$ and so there is some $\beta\in(\alpha,\zeta)\cap C(\zeta)$.
It follows that $g_\beta\in \Phi(M)$, hence $\mu([\beta,\eta])=0$. Similarly as in the first case we deduce $\mu([\alpha,\beta-1])\ne0$, a contradiction with the choice of $\gamma$.

This completes the proof.
\end{proof}

Note that the Markushevich bases from the preceding proposition satisfies the properties from Theorem~\ref{T:char sigma subspace} if and only if $\eta\le\omega_1$. However, $D$ is a $\Sigma$-subspace if and only if $\eta<\omega_2$. Therefore for $\eta\in(\omega_1,\omega_2)$ there should be another Markushevich basis satisfying the respective properties. In fact, the Markushevich basis from the previous proposition coincides with the Markushevich basis canonically constructed using Theorem~\ref{T:M-basis} if and only if $\eta$ is a cardinal number. Next we are going to describe such a  Markushevich basis for general $\eta$.

Assume $\eta$ is infinite and let $\kappa=\card\eta$.  Note that $\card I(\eta)=\card I(\kappa)=\kappa$. Fix a bijection $\xi:I(\kappa)\to I(\eta)$ satisfying $\xi(0)=0$ and set
$$A_\alpha=\overline{\xi([0,\alpha]\cap I(\kappa))},\quad \alpha\le\kappa.$$
Then $A_0=\{0\}$, $A_\kappa=[0,\eta]$, the family $(A_\alpha)_{\alpha\le\kappa}$ is strictly increasing, and $A_\lambda=\overline{\bigcup_{\alpha<\lambda} A_\alpha}$ if $\lambda\le\kappa$ is limit. Moreover, $A_\alpha\in\A$ for each $\alpha\le\kappa$. Therefore this family generates a PRI on $\C([0,\eta])$ and we will describe the Markushevich basis provided by this PRI.

To define the basis We will use the following two auxiliary functions:
$$\left.\begin{array}{l}
z(x,\alpha)=\max\{y\in [x,\eta]\setsep [x,y]\cap A_\alpha=\emptyset\},\\
p(x,\alpha)=\max A_\alpha\cap [0,x),
\end{array}\right\} \mbox{ for }
 \alpha<\kappa,\ x\in[0,\eta]\setminus A_\alpha.$$
The following lemma summarizes basic properties of this function.

\begin{lemma}\label{L:z(x,alpha)} \ 
\begin{itemize}
\item[(i)] The functions $z(\cdot,\cdot)$ and $p(\cdot,\cdot)$ are well defined.
\item[(ii)] If $\alpha<\kappa$, $x\in[0,\eta]\setminus A_\alpha$ and $y\in(p(x,\alpha),z(x,\alpha)]$, then $z(y,\alpha)=z(x,\alpha)$ and $p(y,\alpha)=p(x,\alpha)$.
\item[(iii)] If $\alpha<\kappa$, $x\in[0,\eta]\setminus A_\alpha$ and $\beta\le\alpha$, then $z(x,\beta)\ge z(x,\alpha)$ and $p(x,\beta)\le p(x,\alpha)$.
\end{itemize}
\end{lemma}

\begin{proof}
(i) Fix $\alpha<\kappa$ and $x\in [0,\eta]\setminus A_\alpha$.
$A_\alpha\cap[0,x]$ is nonempty (it contains $0$) and it is a closed set. So, it has a maximum. Since $x\notin A_\alpha$, the maximum is strictly less than $x$. This shows that $p(x,\alpha)$ is well defined.

Let us continue by looking at $z(x,\alpha)$.
If $A_\alpha\cap[x,\eta]=\emptyset$, then $z(x,\alpha)=\eta$. If $A_\alpha\cap [x,\eta]\ne\emptyset$, let $y=\min A_\alpha\cap [x,\eta]$. Since $x\notin A_\alpha$, we deduce $y>x$. So, $[x,y)\cap A_\alpha=\emptyset$, hence $y$ is an isolated point of $A_\alpha$. Since $A_\alpha\in\A$, $y$ is an isolated ordinal and $z(x,\alpha)=y-1$.

The assertion (ii) is obvious.

(iii) Since $A_\beta\subset A_\alpha$, necessarily
$$A_\beta\cap [x,z(x,\alpha)]\subset A_\alpha\cap [x,z(x,\alpha)]=\emptyset,$$
hence by the definition of the function $z(\cdot,\cdot)$ we deduce  $z(x,\beta)\ge z(x,\alpha)$. Moreover, $p(x,\beta)\le p(x,\alpha)$ as 
$$A_\beta\cap[0,x)\subset A_\alpha\cap[0,x).$$
\end{proof}

Using the functions $z(\cdot,\cdot)$ and $p(\cdot,\cdot)$ we are going to define a Markushevich basis $(f_\alpha,\mu_\alpha)_{\alpha\in I(\kappa)}$.
$$\begin{aligned} f_0&=1, \quad \mu_0=\delta_0, \\
f_\alpha&=\chi_{[\xi(\alpha),z(\xi(\alpha),\alpha-1)]},
\quad \mu_\alpha=\delta_{\xi(\alpha)}
 - \delta_{p(\xi(\alpha),\alpha-1)}\mbox{\quad for }\alpha\ge1.\end{aligned}
$$
It follows from Lemma~\ref{L:z(x,alpha)}(i) that $(f_\alpha,\mu_\alpha)$ is a well-defined family in $\C([0,\eta])\times \M([0,\eta])$. Further properties are given in the following proposition.

\begin{prop}\label{P:ord-Mbaze z PRI} \ 
\begin{itemize}
\item[(a)] The above-defined family $(f_\alpha,\mu_\alpha)_{\alpha\in I(\kappa)}$ is a strong Markushevich basis of $\C([0,\eta])$. 
\item[(b)] If $\eta<\omega_2$, then the Markushevich basis from (a) satisfies the properties from the assertions (6)--(8) of Theorem~\ref{T:char sigma subspace}, as well as the properties described in Remark~\ref{Remark Sigma}(d,e).
\item[(c)] Set $H=\{f_\alpha\setsep \alpha\in I(\kappa)\}\cup\{0\}$. Then the following assertions are fulfilled:
\begin{itemize}
\item[(c-i)] If $\eta$ or $\kappa$ has uncountable cofinality, then $0$ is a $\sigma(\C([0,\eta]),D)$-accumulation point of $H$.
\item[(c-ii)] The nonzero $\sigma(\C([0,\eta]),D)$-accumulation points of $H$ are exactly the elements $\chi_{[a+1,d]}$, where $a<\eta$ has uncountable cofinality,
$$\begin{aligned}
\theta&=\liminf_{x\in[0,a)\cap I(\eta)} \xi^{-1}(x)<\xi^{-1}(a+1),\\
d&=\min\{z(a,\gamma)\setsep \gamma<\theta \}.\end{aligned}$$
\item[(c-iii)] The $\sigma(\C([0,\eta]),D)$-closure of $H$ equals to the union of $H$ and the set of accumulation points described in (c-ii).
\end{itemize}
\end{itemize}
\end{prop}

\begin{proof}
(a) Observe that $(P_{A_\alpha})_{\alpha\le\kappa}$ is a PRI on $\C([0,\eta])$ (more precisely, it satisfies all the properties of a PRI except that $P_{A_0}$ is not the zero projection, but a one-dimensional projection -- but this difference does not affect the applications). We will show that the family
 $(f_\alpha,\mu_\alpha)_{\alpha\in I(\kappa)}$ is the Markushevich basis resulting from this PRI in the sense of \cite[Proposition 6.2.4]{fabian-kniha}.

To this end observe that $P_{A_0}\C([0,\eta])=\lin\{f_0\}$ and $(P_{A_{\alpha+1}}-P_{A_\alpha})\C([0,\eta])=\lin\{f_{\alpha+1}\}$ for $\alpha<\kappa$.
The first equality is obvious, as $P_{A_0}f$ it the constant function equal to $f(0)$ for each $f\in \C([0,\eta])$. To show the second case fix $\alpha<\kappa$. Since $A_{\alpha+1}=A_\alpha\cup\{\xi(\alpha+1)\}$, for any $f\in\C([0,\eta])$ we get
$$
\begin{aligned}
x\in[0,\xi(\alpha+1)) & \Rightarrow P_{A_{\alpha+1}}f(x)=P_{A_\alpha}f(x),\\
x\in[\xi(\alpha+1),\zeta(\xi(\alpha+1),\alpha)]&\Rightarrow P_{A_{\alpha+1}}f(x)=f(\xi(\alpha+1)), P_{A_\alpha}f(x)=f(p(\xi(\alpha+1),\alpha)), \\
x\in(\zeta(\xi(\alpha+1),\alpha),\eta] & \Rightarrow P_{A_{\alpha+1}}f(x)=P_{A_\alpha}f(x);
\end{aligned}
$$
hence
$$(P_{A_{\alpha+1}}-P_{A_\alpha})f=(f(\xi(\alpha+1)-f(p(\xi(\alpha+1),\alpha)) f_{\alpha+1}.$$
This completes the proof of the second case. Moreover, the computation shows that $(f_\alpha,\mu_\alpha)_{\alpha\in I(\kappa)}$ is exactly the Markushevich basis provided by \cite[Proposition 6.2.4]{fabian-kniha}. Since the bases of the respective one-dimensional spaces are strong, the resulting basis is also strong
(as remarked in \cite[Theorem 5.1]{HMVZ-biortogonal}).

(b) By Remark~\ref{Remark Sigma}(d) it is enough to show that the Markushevich basis satisfies the property from the assertion (6) of Theorem~\ref{T:char sigma subspace}. If $\eta<\omega_1$, then the Markushevich basis is countable and the statement is trivial. If $\omega_1\le\eta<\omega_2$, then $\kappa=\omega_1$, hence 
the family $\{A_\alpha\setsep\alpha<\omega_1\}$ is a $\sigma$-closed cofinal subset of $\A_\omega$. Thus
$$D=\bigcup_{\alpha<\omega_1}P_{A_\alpha}^*\M([0,\eta]).$$
Since for any $\alpha<\beta<\omega_1$ and any $\mu\in\M([0,\eta])$ we have
$$P_{A_\alpha}^*\mu(f_\beta)=\mu(P_{A_\alpha}f_\beta)=0,$$
we deduce that the $\Sigma$-subspace induced by $M=\{f_\alpha\setsep \alpha\in I(\omega_1)\}$ contains $D$, thus it is equal to $D$ .

(c) By Lemma~\ref{L:ord dve topologie} we may work with the topology $\tau_p(S)$. 
Set 
$$F=\{0\}\cup\{\chi_{[\xi(\alpha),y]}\setsep \alpha\in I(\kappa), y\in[\xi(\alpha),\eta]\}.$$
Then clearly $H\subset F$. Moreover, $F$ is $\tau_p(S(\eta))$-closed. Indeed, as $S(\eta)$ is dense in $[0,\eta]$, we have
$$ \begin{aligned}
 F=\Big\{f\in\C([0,\eta]\setsep (\forall x\in S(\eta): f(x)\in\{0,1\})\ \&\ \Big(\forall & x,y,z\in S(\eta), x<y<z : \\
& (f(x)=1\ \&\ f(y)=0 \Rightarrow f(z)=0)\  \\ & \&\ (f(y)=0\ \&\ f(z)=1\Rightarrow f(x)=0)\Big)
\Big\}.\end{aligned}
$$

Now let us analyze the $\tau_p(S(\eta))$-accumulation points of $H$.

We start by proving (c-i).
 Let $U$ be a $\tau_p(S(\eta))$-neighborhood of $0$. It follows that there is a finite set $C\subset S(\eta)$ such that 
$$\{f\in\C([0,\eta]\setsep f|_C=0\}\subset U.$$
If $\eta$ has uncountable cofinality, then $\max C<\eta$. So, we can find $x\in I(\eta)$ such that $x>\max C$. Then clearly $f_{\xi^{-1}(x)}\in U\cap (H\setminus\{0\})$.
Next assume that $\kappa$ has uncountable cofinality. For each $x\in C$ limit choose a countable set $B(x)\subset I(\eta)$ with supremum $x$. Then
$$\sup(\xi^{-1}(C\cap I(\eta))\cup \bigcup_{x\in C\setminus I(\eta)} \xi^{-1}(B(x)))<\kappa.$$
So, fix some $\alpha\in I(\kappa)$ strictly greater than the left-hand side. Then $C\subset A_{\alpha-1}$, hence $f_\alpha\in U\cap(H\setminus \{0\})$. This completes the proof of (c-i).

\smallskip

Let us continue by proving (c-ii). By the above any accumulation point of $H$ is a characteristic function of a clopen interval in $[0,\eta]$. We distinguish three cases of such intervals.

{\tt Case 1:} $f=\chi_{[0,x]}$, where $x\in[0,\eta]$. Then $f\in \overline{H}^{\tau_p(S(\eta))}$ if and only if $x=\eta$ (i.e., $f=1$). In this case $f$ is an isolated point of $H$. Indeed, $f\in U$ and $U\cap H=\{f_0\}$ if
$$U=\{g\in\C([0,\eta])\setsep g(0)\ne 0\}\cap H.$$

{\tt Case 2:} $f=\chi_{[a+1,b]}$, where $a\in S(\eta)$ and $b\ge a+1$. Then  $f\in \overline{H}^{\tau_p(S)}$ if and only if $f\in H$. In this case $f$ is an isolated point of $H$. Indeed, if we set
$$U=\{g\in\C([0,\eta])\setsep g(a)\ne 1,g(a+1)\ne0\},$$
then $f\in U$ and $U\cap H=\{f_{\xi^{-1}(a+1)}\}$.

{\tt Case 3:} $f=\chi_{[a+1,b]}$, where $a\in[0,\eta]$ is an ordinal with uncountable cofinality and $b\ge a+1$. Let us define $\theta$ and $d$ as in the statement of (c-ii) and
 set
$$\alpha=\xi^{-1}(a+1),\quad c=z(a+1,\alpha-1).$$
Note that the definition of liminf together with the fact that ordinals are well ordered yield
$$\theta=\sup_{x<a}\ \min_{y\in I(\eta)\cap(x,a)}\xi^{-1}(y).$$
Since $\xi^{-1}$ is one-to-one,  $\theta$ is  necessarily a limit ordinal.
It follows that the following construction can be performed.

For any $\gamma<\theta$ let $z_\gamma<a$ be the smallest ordinal such that
$$\min\{\xi^{-1}(x)\setsep x\in(z_\gamma,a)\cap I(\eta)\}\in(\gamma,\theta)$$
and let $y_\gamma\in(z_\gamma,a)\cap I(\eta)$ be the (unique) element where the minimum is attained. Observe that both assignments $\gamma\mapsto z_\gamma$ and $\gamma\mapsto y_\gamma$ are nondecreasing. Indeed, let $\gamma_1\le\gamma_2<\theta$.
Then
$$\min\{\xi^{-1}(x)\setsep x\in(z_{\gamma_2},a)\cap I(\eta)\}\in(\gamma_2,\theta)\subset(\gamma_1,\theta),$$
so $z_{\gamma_1}\le z_{\gamma_2}$ by the minimality of $z_{\gamma_1}$. Further, if $y_{\gamma_1}> z_{\gamma_2}$, then obviously $y_{\gamma_2}=y_{\gamma_1}$. If the converse inequality holds, then
$y_{\gamma_2}>z_{\gamma_2}\ge y_{\gamma_1}$.
We further get $\xi^{-1}(y_{\gamma_1})\le \xi^{-1}(y_{\gamma_2})$ (as $y_{\gamma_2}\in(z_{\gamma_1},a)$). The next thing to observe is that $\sup_{\gamma<\theta}z_\gamma=a$. Indeed, denote the supremum by $u$ and assume that $u<a$. Then
$$
\theta>\min\{\xi^{-1}(x)\setsep x\in(u,a)\cap I(\eta)\}
\ge \sup_{\gamma<\theta}\min\{\xi^{-1}(x)\setsep x\in(z_\gamma,a)\cap I(\eta)\}
\ge \sup_{\gamma<\theta}\gamma=\theta,$$
a contradiction.

We continue by looking at the net $(f_{\xi^{-1}(y_\gamma)})_{\gamma<\theta}$.
Recall that $f_{\xi^{-1}(y_\gamma)}=\chi_{[y_\gamma,d_\gamma]}$, where 
$$d_\gamma=z(y_\gamma,\xi^{-1}(y_\gamma)-1).$$
Since  
$$\min\{\xi^{-1}(x)\setsep x\in [y_\gamma,a)\cap I(\eta)\}=\xi^{-1}(y_\gamma),$$
we deduce $[y_\gamma,a)\cap A_{\xi^{-1}(y_\gamma)-1}=\emptyset$. Further, $A_{\xi^{-1}(y_\gamma)-1}\in \A$, hence $a\notin A_{\xi^{-1}(y_\gamma)-1}$ as well.
It follows that $d_\gamma\ge a$. Moreover, if $\gamma_1\le\gamma_2$, then by the above $\xi^{-1}(y_{\gamma_1})\le \xi^{-1}(y_{\gamma_2})$ and hence $d_{\gamma_1}\ge d_{\gamma_2}$. It follows that $(d_\gamma)$ is eventually constant. In fact, it is eventually equal to $d$ (by the definition of $d$).

Now we are ready to distinguish the following two cases (recall that $\theta$ is limit and $\alpha$ isolated, thus $\theta\ne\alpha$).

\begin{description}
\item[$\theta>\alpha$\ ] Given $\gamma\in(\alpha,\theta)$, we have $\xi^{-1}(y_\gamma)>\gamma>\alpha$, hence
$a+1=\xi(\alpha)\in A_{\alpha}\subset A_{\xi^{-1}(y_\gamma)-1}$, thus $d_\gamma=a$.  
It follows that $d=a$ and that the net $(f_{\xi^{-1}(y_\gamma)})_{\gamma<\theta}$  $\tau_p(S(\eta))$-converges to $0$.  

Further, $f\in\overline{H}^{\tau_p(S)}$ if and only if $f\in H$ (i.e., $b=c$) and in this case $f$ is an isolated point of $H$. 
Indeed, fix some $\gamma\in(\alpha,\theta)$. Then
$$U=\{g\in\C([0,\eta]\setsep g(a+1)\ne0,g(y_\gamma)\ne 1\}$$
is a $\tau_p(S(\eta))$-neighborhood of $f$ not containing $0$. It is enough to show that $U\cap H=\{f_\alpha\}$. So, assume $f_\beta\in U$. Then $f_\beta(a+1)=1$ and $f_\beta(y_\gamma)=0$. It follows that 
$$y_\gamma<\xi(\beta)\le a+1\le z(\xi(\beta),\beta-1).$$
If $\xi(\beta)=a+1$, then $\beta=\alpha$. Otherwise $\xi(\beta)\in(y(\gamma),a)$. It follows that 
$$\beta=\xi^{-1}(\xi(\beta))>\xi^{-1}(y(\gamma))>\gamma>\alpha.$$
So, $a+1=\xi(\alpha)\in A_{\beta-1}$, a contradiction.

\item[$\theta<\alpha$\ ] In this case, given $\gamma<\theta$, we have $d_\gamma\ge a+1$ and hence $d_\gamma\ge c$ as well. It follows $d\ge c$. Further, the net $(f_{\xi^{-1}(y_\gamma)})_{\gamma<\theta}$ $\tau_p(S(\eta))$-converges to $\chi_{[a+1,d]}$  (note that $a\notin S(\eta)$). So, $\chi_{[a+1,d]}$ is an accumulation point of $H$.

It remains to show that $f$ is not an accumulation point of $H$ provided $d\ne b$. We distinguish two cases:
\begin{description}
\item[$b<d$\ ] Let 
$$U=\{g\in\C([0,\eta]\setsep g(a+1)\ne0, g(b+1)\ne 1\}.$$
Then $U$ is a $\tau_p(S(\eta))$-neighborhood of $f$ not containing $0$. Assume that
$f_\beta\in U$ for some $\beta\in I(\kappa)$. Then $f_\beta(a+1)=1$ and $f_\eta(b+1)=0$.
Hence 
$$\xi(\beta)\le a+1\le z(\xi(\beta),\beta-1)<b+1.$$
If $\xi(\beta)=a+1$, then $\beta=\alpha$ (which can happen only if $b=c$). 

So, assume $\xi(\beta)<a$. We have
$$\beta-1<\min\{\xi^{-1}(x)\setsep x\in[\xi(\beta),a)\cap I(\eta)\}<\theta.$$
Indeed, the first inequality follows from the fact that $[\xi(\beta),a+1]\cap A_{\beta-1}=\emptyset$, the second one from the definition of $\theta$. It follows that $\beta<\theta$ (recall that $\theta$ is limit). But then
$$d\le z(\xi(\beta),\beta-1)\le b,$$
a contradiction.
Therefore, $U\cap H$ contains $f_\alpha$ if $f=f_\alpha$ and is empty otherwise.

\item[$b>d$\ ] Recall that $\theta<\alpha$ and $d\ge c$. Let
$$U=\{g\in\C([0,\eta]\setsep g(a+1)\ne0, g(d+1)\ne 0\}.$$
Then $U$ is a $\tau_p(S(\eta))$-neighborhood of $f$ not containing $0$. Assume that
$f_\beta\in U$ for some $\beta\in I(\kappa)$. Then $f_\beta(a+1)=f_\beta(c+1)=f_\beta(d+1)=1$.
 Hence
$$\xi(\beta)\le a+1\le c\le d<b<z(\xi(\beta),\beta-1).$$
 If $\xi(\beta)=a+1$, then $\beta=\alpha$, which cannot happen as $b>d\ge c$. 
So, $\xi(\beta)<a$. As in the previous case we can show that $\beta<\theta$. So, we can find $\gamma\in(\beta,\theta)$ such that $y_\gamma\in(\xi(\beta),a)$ and $d_\gamma=d$. Let us consider a smaller neighborhood
$$V=\{g\in U\setsep g(y_\gamma)\ne1\}.$$
We claim that $V\cap H=\emptyset$. Indeed, let $f_\zeta\in V$ for some $\zeta\in I(\eta)$. Then $\xi(\zeta)>y_\gamma$ and $\xi(d+1)=1$ and so

$$d+1\le z(\xi(\zeta),\zeta-1)=z(a+1,\zeta-1)\le z(a+1,\gamma-1)=d,$$
a contradiction.
\end{description}
\end{description}
This completes the proof of the assertion (c-ii). The assertion (c-iii) is then obvious.
 \end{proof}

The assertion (c) of the previous proposition indicates that the properties of a Markushevich basis
constructed from a PRI may depend on the concrete choice of PRI (i.e., on the choice of the mapping $\xi$). The next example provides a strong evidence of this dependence.

\begin{example}\label{E:Mbaze divna} Assume that $\eta$ is a cardinal number (i.e., $\kappa=\eta$).
\begin{itemize}
\item[(a)] If $\xi: I(\eta)\to I(\eta)$ is the identity mapping, then the Markushevich basis $(f_\alpha,\mu_\alpha)_{\alpha\in I(\eta)}$ coincides with the  Markushevich basis $(g_\alpha,\nu_\alpha)_{\alpha\in I(\eta)}$ from Proposition~\ref{P:ord Mbaze ocasni}.
\item[(b)] Assume moreover $\eta\ge\omega_2$. Define a bijection $\xi:I(\eta)\to I(\eta)$ by the formula
$$\xi(\alpha)=\begin{cases}
\lambda+2 & \mbox{ if $\alpha=\lambda+1$ where $\lambda$ has uncountable cofinality},\\
\lambda+1 & \mbox{ if $\alpha=\lambda+2$ where $\lambda$ has uncountable cofinality},\\
\alpha &\mbox{ otherwise}.
\end{cases}$$
Let $(f_\alpha,\mu_\alpha)_{\alpha\in I(\eta)}$ be the Markushevich basis defined above and let $H=\{f_\alpha\setsep \alpha\in I(\eta)\}\cup\{0\}$. 
Then all the nonzero elements of $H$ are $\sigma(\C([0,\eta]),D(\eta))$-isolated points of $H$, $H$ is not $\sigma(\C([0,\eta]),D(\eta))$-Lindel\"of and there is $A_0\in \A_\omega$ such that $P_A(H)\not\subset H$ whenever $A\in\A_\omega$ and $A\supset A_0$. 
\end{itemize}
\end{example}

\begin{proof} The assertion (a) is obvious. Let us prove the assertion (b).
Observe that in this case we have
$$f_\alpha=\begin{cases}
\chi_{[\lambda+2,\eta]} &\mbox{ if $\alpha=\lambda+1$ where $\lambda$ has uncountable cofinality}, \\
\chi_{\{\lambda+1\}}&\mbox{ if $\alpha=\lambda+2$ where $\lambda$ has uncountable cofinality}, \\
\chi_{[\alpha,\eta]} &\mbox{ for other }\alpha\in I(\eta).
\end{cases}$$
The fact that all the nonzero elements of $H$ are isolated points of $H$ follows from Proposition~\ref{P:ord-Mbaze z PRI}(c), but it can be easily seen directly. Indeed, 
if $\lambda<\eta$ has uncountable cofinality, set
$$\begin{aligned}
U_{\lambda+1}&=\{g\in \C([0,\eta])\setsep g(\lambda+1)\ne 1,g(\lambda+2)\ne0 \},\\
U_{\lambda+2}&=\{g\in \C([0,\eta])\setsep g(\lambda+1)\ne 0,g(\lambda+2)\ne1 \}.
\end{aligned}$$
Further, let
$$\begin{aligned}
U_0&=\{g\in\C([0,\eta]\setsep g(0)\ne0\},\\
U_\alpha&=\{g\in C([0,\eta])\setsep g(\alpha)\ne0,g(\alpha-1)\ne1\}\mbox{ for other }\alpha\in I(\eta).
\end{aligned}$$
Then, for any $\alpha\in I(\eta)$, $U_\alpha$ is a $\sigma(\C([0,\eta]),D(\eta))$-open neighborhood of $f_\alpha$ such that $U_\alpha\cap H=\{f_\alpha\}$.

Moreover, set
$$V=\{g\in\C([0,\eta])\setsep g(\omega_1+1)\ne 1\}.$$
Then $V$ is  a $\sigma(\C([0,\eta]),D)$-open neighborhood of $0$ and
$$V\cap H=\{0,f_{\omega_1+1}\}\cup\{f_\alpha\setsep \alpha\in I(\eta)\cap [\omega_1+3,\eta]\}.$$
It follows that
$$\{V\}\cup\{U_\alpha\setsep \alpha\in I(\eta)\}$$
is a  $\sigma(\C([0,\eta]),D(\eta))$-open cover of $H$ without a countable subcover,
hence $H$ is not Lindel\"of in the topology $\sigma(\C([0,\eta]),D(\eta))$.

Moreover, let $A\in\A_\omega$ be such that $\omega_1+1\in A$. Then $P_A(H)\not\subset H$. Indeed, let $\alpha=\max A\cap[0,\omega_1]$ Then $\alpha<\omega_1$. Further,
$P_A(f_{\alpha+1})=\chi_{[\omega_1+1,\eta]}\notin H$.
\end{proof}

The assertion (ii) of the previous example shows that the topological properties of a Markushevich basis constructed from a PRI can be very bad. Related problems are discussed in the last section.

\subsection{Continuous functions on trees}  

Further examples of Banach spaces with a non-commutative retractional skeleton are spaces
of continuous functions on certain trees equipped with the coarse wedge topology studied for example in \cite{somaglia-trees,somaglia-treti}. Let us start by recalling the basic setting.

A \emph{tree} is a partially ordered set $(T,\le)$ such that for any $t\in T$ the set $\{s\in T
\setsep s<t\}$ is well ordered. A tree $(T,\le)$ is called \emph{rooted} if it has a unique minimal element (called the \emph{root} of $T$ and usually denoted by $0$). $T$ is called \emph{chain complete} if any chain in $T$ (i.e., any totally ordered subset of $T$) has a supremum (i.e., the smallest upper bound). By a tree we will always mean a rooted chain-complete tree. 

Let $(T,\le)$ be a tree. For any $t\in T$ we set
$$\begin{aligned}
\widehat{t}&=\{s\in T\setsep s\le t\},\\
V_t&=\{s\in T\setsep s\ge t\},\\
\htt(t,T)&=\mbox{ the order type of }\{s\in T\setsep s<t\},\\
\cf(t)&=\cf(\htt(t,T)) \mbox{\quad (the cofinality of $t$)},\\
\ims(T)&=\{s\in T\setsep s>t \mbox{ and (the order interval) }(t,s)=\emptyset\}
\\&=\{s\in T\setsep s\ge t\ \&\ \htt(s,T)=\htt(t,T)+1\}\\& (=\mbox{the set of all the immediate successors of }t).
\end{aligned}$$
Further, for any ordinal $\alpha$ we denote 
$$\Lev_\alpha(T)=\{t\in T\setsep \htt(t,T)=\alpha\}$$
and the \emph{height} of $T$ is defined by
$$\htt(T)=\min\{\alpha\setsep \Lev_\alpha(T)=\emptyset\}.$$
We will further need the following two important subsets of $T$.
$$\begin{aligned}
I(T)&=\bigcup\{\Lev_\alpha(T)\setsep \alpha<\htt(T) \mbox{ isolated}\}
\\&=\{x\in T\setsep \htt(x,T)\mbox{ is an isolated ordinal}\},\\
S(T)&=\bigcup\{\Lev_\alpha(T)\setsep \alpha<\htt(T), \cf\alpha \mbox{ is at most countable}\}
\\&=\{x\in T\setsep\cf(x)\mbox{ is at most countable}\}.
\end{aligned}$$
The \emph{coarse wedge topology} on a tree $T$ is the topology on $T$ whose subbase is
is the family 
$$ V_t, T\setminus V_t,\quad t\in T, \htt(t,T)\mbox{ is an isolated ordinal}.$$
It is easy to check that a neighborhood basis of $t\in T$ is the family
$$W_t^F=V_t\setminus \bigcup_{u\in F}V_u, \quad F\subset\ims(t)\mbox{ finite}$$
in case $\htt(t,T)$ is an isolated ordinal; and the family
$$W_s^F=V_s\setminus \bigcup_{u\in F}V_u, \quad s<t, \htt(s,T) \mbox{ is an isolated ordinal}, F\subset\ims(t)\mbox{ finite}$$
in case $\htt(t,T)$ is a limit ordinal.

Any tree equipped with the coarse wedge topology is a compact Hausdorff space \cite[Corollary 3.5]{nyikos}. This topology is one of many topologies studied on trees, see \cite{nyikos}. It also coincides with the path topology considered in \cite[pp. 288--289]{todorcevic-handbook} or
\cite{todorcevic-aronszajn}. Let us explain it a bit. If $T$ is a tree (not necessarily rooted or chain complete), we can consider its path space, i.e., the set of all the initial totally ordered segments embedded as characteristic functions to the product space $\{0,1\}^T$. In this way we obtain a compact Hausdorff space. Moreover, it is easy to check that the class of path spaces of arbitrary trees canonically coincides with the class of rooted chain complete trees equipped with the coarse wedge topology.

Any ordinal segment $[0,\eta]$ is a special case of a tree with the coarse wedge topology.
Therefore the results of the previous section can be viewed as a special case of the results in the current section. We will see that the situation of trees is more complicated. Let us start by the following analogue of Lemma~\ref{L:retrakce-ord}.

\begin{lemma}\label{L:trees retrakce} Let $T$ be a tree equipped with a coarse wedge topology. Let $A\subset T$ be a closed set containing $0$. The following are equivalent.
\begin{enumerate}
\item The mapping $r_A:T\to T$ defined by
$$r_A(t)=\max(A\cap \widehat{t}),\quad t\in T,$$
is a continuous retraction of $T$ onto $A$.
\item For each $x\in A$ on a limit level (i.e., such that $\htt(x,T)$ is a limit ordinal) we have
$$x=\sup\{y\in A\setsep y<x\}.$$
\item For each $x\in A$ we have
$$x=\sup\widehat{x}\cap I(T).$$
\end{enumerate}
\end{lemma}

\begin{proof} First observe that the mapping $r_A$ is a well-defined retraction of $T$ onto $A$ -- this follows easily from the assumptions that $A$ is closed and contains $0$.
Hence the point of the assertion (1) is the continuity of $r_A$.

(1)$\Rightarrow$(2) Assume $r_A$ is continuous and $x\in A$ is on a limit level. Note that $\widehat{x}$ is order isomorphic and homeomorphic to an ordinal segment and $r_A|_{\widehat{x}}$ coincides with the mapping $r_{A\cap \widehat{x}}$ from Lemma~\ref{L:retrakce-ord}. Hence we can conclude by Lemma~\ref{L:retrakce-ord}(3)$\Rightarrow$(1). 

(2)$\Rightarrow$(3) Fix any $x\in A$. If $x\in I(T)$, the the equality trivially holds -- $x$ is even maximum of the set on the right-hand side.

So, assume that $x$ is on a limit level. Fix an arbitrary $y<x$. By (2) we know that the order interval $(y,x)$ intersects $A$, hence we can define $z=\min((y,x)\cap A)$ (recall that initial segments of $T$ are well ordered). Since $z\in A$ and $(z,y)\cap A=\emptyset$, another use of (2) yields that $z\in I(T)$. This completes the proof.

(3)$\Rightarrow$(1) We will show that $r_A$ is continuous at each point. So, fix any $x\in T$. There are the following possibilities:

Case 1: $x\notin A$. Since $A$ is closed, there is a basic neighborhood $W_y^F$ of $x$ (recall that $y\le x$ is on an isolated level and $F\subset\ims(x)$ is finite) such that $W_y^F\cap A=\emptyset$. Then $r_A$ is constant on $W_y^F$, so it is continuous at $x$.

Case 2: $x\in A\cap I(T)$. Then $r_A(x)=x$. So, fix any open set $U$ containing $x$. By the definition of the topology there is a finite set $F\subset \ims(x)$ such that $W_x^F\subset U$. Since $r_A(W_x^F)\subset W_x^F\subset U$, the proof of continuity at $x$ is complete. 

Case 3: $x\in A$, $x$ on a limit level. Again, $r_A(x)=x$. Fix any open set $U$ containing $x$. By the definition of the topology there is some $y<x$ on an isolated level and a finite set $F\subset \ims(x)$ such that $W_y^F\subset U$. By (3) there is some $z\in (y,x)\cap A\cap I(T)$. Then $W_z^F$ is a neighborhood of $x$ and $r_A(W_z^F)\subset W_z^F\subset U$.
\end{proof}

Let $\A^0=\A^0(T)$ denote the family of all the closed subsets of $T$  containing $0$ and satisfying the equivalent assertions of Lemma~\ref{L:trees retrakce}. Then we have the following analogue of Lemma~\ref{L:kompatibilita-ord}.

\begin{lemma}\label{L:trees compatibility}
Let $A,B\in \A^0$ be such that $A\subset B$. Then $r_A\circ r_B=r_B\circ r_A=r_A$.
\end{lemma}

The proof is easy -- either one can copy the argument of Lemma~\ref{L:kompatibilita-ord} or one can apply this lemma to the initial segments of $T$.

Here we reached the limits of easy analogies. An analogue of Lemma~\ref{L:spojitost-ord}
fails for the family $\A^0$. It is witnessed by the following example.

\begin{example} Let 
$$T=[0,\omega]\cup\{\omega+1\}\times\en.$$
Assume that the order on $[0,\omega]$ coincides with the standard ordinal order and  $\{\omega+1\}\times\en$ is the set $\ims(\omega)$ of immediate successors of $\omega$. Then $T$ is a tree. Moreover, the sets
$$A_n=\{0,(\omega+1,1),\dots,(\omega+1,n)\}$$
belong to $\A^0$ and form an increasing sequence, while the set
$$\overline{\bigcup_{n\in\en}A_n}=\{0,\omega\}\cup\{\omega+1\}\times \en$$
does not belong to $\A^0$. Moreover, the sequence clearly has no supremum in $\A^0$. 
\end{example}

So, to get analogous results as for ordinal segment one should proceed more carefully.
Firstly, while any ordinal segment admits a retractional skeleton, for trees it is not the case. We recall the following result of J.~Somaglia (see \cite[Theorem 3.1]{somaglia-trees} and \cite[Theorem 4.2]{somaglia-treti}).

\begin{prop}\label{P:trees-skeleton-jacopo} Let $T$ be a tree. The following assertions are equivalent
\begin{enumerate}
\item $T$ admits a retractional skeleton.
\item $\C(T)$ admits a $1$-projectional skeleton.
\item $\ims(x)$ is finite for each $x\in T$ with $\cf(x)$ uncountable.
\end{enumerate}
\end{prop}

A retractional skeleton constructed in \cite{somaglia-trees} is formed by the retractions
$r_A$ where $A$ runs through a carefully chosen subfamily of $\A^0$. Using similar ideas we present a simplified more canonical approach. We start by restricting ourselves to a special case. The following lemma shows that this can be done without loss of generality.

\begin{lemma}\label{L:trees-r1} Let $T$ be a tree such that
\begin{equation*}
\ims(x)\mbox{ is finite for each $x\in T$ with $\cf(x)$ uncountable.}
\end{equation*}
Then there is a new partial order $\preceq$ on $T$ satisfying the following properties:
\begin{itemize}
\item $\preceq$ is finer than $\le$, i.e., $x\preceq y$ whenever $x,y\in T$ satisfy $x\le y$.
\item $I(T,\preceq)=I(T,\le)$, $S(T,\preceq)=S(T,\le)$.
\item The coarse wedge topologies on $T$ defined by $\le$ and by $\preceq$ coincide.
\item $\ims_{\preceq}(x)$ contains at most one point for each $x\in T$ with $\cf(x)$ uncountable.
\end{itemize}
\end{lemma}

\begin{proof} For any $z\in T$ with $\cf(z)$ uncountable and $\ims(z)\ne\emptyset$ set $N(z)=\card\ims(z)$ and fix an enumeration $\ims(z)=\{z[1],z[2],\dots,z[N(z)]\}$.
Then the new order which does the job may be defined by
$$\begin{aligned}
x\prec y \overset{def}{\equiv} & x<y, \mbox{ or}\\& \exists z\in T, \cf(z)\mbox{ uncountable }
\exists j,k\le N(z): j<k\ \&\ x=z[j]\ \&\ y=z[k], \mbox{ or }\\&\exists z\in T, \cf(z)\mbox{ uncountable }\exists j,k\le N(z): x=z[j] \& z[k]<y.
\end{aligned}$$
In other words, for each $z\in T$ with $\cf(z)$ uncountable and $\ims(z)\ne\emptyset$ we set
$$z\prec z[1]\prec z[2]\prec \dots\prec z[N(z)]$$
and
$$\ims_{\preceq} z[(N(z)]=\bigcup_{j=1}^{N(z)}\ims_\le z[j],$$
preserving the relations of the remaining points.
\end{proof}

In the sequel by an \emph{$r$-tree} we will mean a tree satisfying the condition (3) from Proposition~\ref{P:trees-skeleton-jacopo}, i.e., a tree having a retractional skeleton. Further, an \emph{$r_1$-tree} will be a tree satisfying the stronger condition from  Lemma~\ref{L:trees-r1}, i.e., such that
$$\ims(x)\mbox{ contains at most one point for each $x\in T$ with $\cf(x)$ uncountable}.$$
It is clear that $r_1$-trees form a subclass of $r$-trees. However, Lemma~\ref{L:trees-r1} says in particular, that any $r$-tree is homeomorphic to some $r_1$-tree. Thus dealing with $r_1$-trees instead of $r$-trees does not result in loosing generality. 

We  will need certain topological properties of trees. To investigate them we will use the following function.
Let $T$ be a tree. For $s,t\in T$ set
$$s\wedge t=\max \widehat{s}\cap\widehat{t}.$$
This is a well-defined element of $T$ -- note that $\widehat{s}\cap\widehat{t}$ is nonempty (it contains $0$), closed and well ordered. The following lemma summarizes several properties of this operation.

\begin{lemma}\label{L:trees wedge} Let $T$ be a tree.
\begin{itemize}
\item[(a)] The mapping $(s,t)\mapsto s\wedge t$ is continuous as a mapping $T\times T\to T$.
\item[(b)] Let $A\subset T$ and $x\in \overline{A}$. Then 
$$x=\sup\{s\wedge t\setsep s,t\in A\ \&\ s\wedge t\le x\}.$$
\item[(c)] Let $A\subset T$. Then the set
$$\{s\wedge t\setsep s,t\in A\}$$
is invariant for the operation $\wedge$.
\item[(d)] If $A\subset T$ is invariant for the operation $\wedge$, then so is $\overline{A}$.
\end{itemize}
\end{lemma}

\begin{proof}
(a) Fix any pair $(s,t)\in T\times T$. We distinguish three possibilities.

Case 1: $s$ and $t$ are incomparable. Then $s\wedge t<s$ and $s\wedge t<t$. So, we can fix
$s',t'\in\ims(s\wedge t)$ such that $s'\le s$ and $t'\le t$. Then $s'\ne t'$, $V_{s'}$ is a neighborhood of $s$, $V_{t'}$ is a neighborhood of $t$ and $u\wedge v=s\wedge t$ whenever $u\in V_{s'}$ and $v\in V_{t'}$.

Case 2: $s=t$. Let $W_x^F$ be a basic neighborhood of $s\wedge t=s=t$ (i.e., $x\le s\wedge t$ is on an isolated level and $F\subset\ims(s\wedge t)$ is finite). Then $W_x^F$ is also a neighborhood both of $s$ and of $t$ and $u\wedge v\in W_x^F$ whenever $u,v\in W_x^F$.

Case 3: $s$ and $t$ are comparable but different. Without loss of generality $s<t$. Then $s\wedge t=s$.  Let $W_x^F$ be a basic neighborhood of $s\wedge t=s$ (i.e., $x\le s$ is on an isolated level and $F\subset\ims(s)$ is finite). Let $y\in\ims(s)$ be such that $y\le t$. Then $V_y$ is a neighborhood of $t$, $W_x^{F\cup\{y\}}$ a neighborhood of $s$ and $u\wedge v\in W_x^F$ whenever $u\in W_x^{F\cup\{y\}}$ and $v\in V_y$.

(b) Assume $x\in\overline{A}$. If $x\in A$, the conclusion is obvious. Thus suppose $x\in \overline{A}\setminus A$. We distinguish two cases:

Case 1: $x$ is on an isolated level. Then $V_x$ is a neighborhood of $x$, so there is some $s\in V_x\cap A$. Since $x\notin A$, we get $s>x$. Fix 
$y\in \ims(x)$ with $y\le s$. Then $W_x^{\{y\}}$ is a neighborhood of $x$, hence there is some $t\in V_x^{\{y\}}\cap A$. Clearly $s\wedge t=x$.

Case 2: $x$ is on a limit level. Fix any $y<x$. Let $z\in\ims(y)$ be such that $z\le x$.
Then $z<x$ and $V_z$ is a neighborhood of $x$, thus we can find some $s\in V_z\cap A$.
If $s\le x$, the proof is complete (as $s\wedge s=s\in(z,x]$). So, assume $s\not\le x$, i.e., $s\wedge x<s$. Let $u\in \ims(s\wedge x)$ be such that $u\le s$. Then $u\not\le x$, thus $W_z^{\{u\}}$ is a neighborhood of $x$, so there is some $t\in W_z^{\{u\}}\cap A$. Then
$$y<z\le s\wedge t=u\wedge t<u,\mbox{ so }s\wedge t\le s\wedge x\le x.$$

(c) Denote the set from the statement by $\widetilde A$. Assume $a,b,c,d\in A$. We need to show that $(a\wedge b)\wedge(c\wedge d)\in\widetilde A$. Consider the three elements
$$a\wedge b, a\wedge c, a\wedge d.$$
They are contained in $\widehat{a}$ which is a linearly ordered set, so one of them should be smaller than the others. 

If the smallest one is $a\wedge b$, then $a\wedge b\le c$ and $a\wedge b\le d$, so $a\wedge b\le c\wedge d$, so $(a\wedge b)\wedge(c\wedge d)=a\wedge b\in\widetilde{A}$.

Assume that the smallest one is $a\wedge c$. Then 
$$\left.\begin{array}{c}
a\wedge c\le a\wedge b\\
\left.\begin{array}{c}
a\wedge c \le a\wedge d\le d\\
a\wedge c\le c\end{array}\right\}\Rightarrow a\wedge c\le c\wedge d
\end{array}\right\}\Rightarrow a\wedge c\le (a\wedge b)\wedge (c\wedge d).
$$
Moreover, let $x\in\ims(a\wedge c)$ be such that $x\le a$. Then $x\le a\wedge b$, thus
$x\in\widehat{a\wedge b}$. On the other hand, $x\notin \widehat{c}$, thus $x\notin\widehat{c\wedge d}$. It follows that  $(a\wedge b)\wedge (c\wedge d)=a\wedge c\in\widetilde{A}$.

The case when the smallest one is $a\wedge d$ is analogous to the previous one (just interchange the role of $c$ and $d$).

(d) Assume $A$ is invariant for $\wedge$. Let  $x,y\in\overline{A}$. If they are comparable, then  $x\wedge y\in\{x,y\}\subset \overline{A}$. So, assume $x$ and $y$ are incomparable, i.e. $x\wedge y<x$ and $x\wedge y<y$. By (b) there are $a,b,c,d\in A$ such that $x\wedge y< a\wedge b\le x$ and $x\wedge y<c\wedge d\le y$. It follows that
$$x\wedge y=(a\wedge b)\wedge(c\wedge d)\in A$$
as $A$ is invariant for $\wedge$.
\end{proof}

\begin{lemma}\label{L:trees vaha}
Let $T$ be an infinite tree. Then $w(T)=\dens(T)=\card I(T)$.
\end{lemma}

\begin{proof}
Recall that $w(T)$ denotes the weight and $\dens(T)$ the density of $T$. The inequality $\dens(T)\le w(T)$ holds in any topological space. Further, $w(T)\le \card I(T)$
by the definition of the coarse wedge topology.
Let us prove the remaining  inequality. Assume $A$ is a dense subset of $T$. Let 
$$\widetilde{A}=\{x\wedge y\setsep x,y\in A\}.$$
Then $\widetilde{A}$ is invariant for $\wedge$ (by Lemma~\ref{L:trees wedge}(c)), clearly it has the same cardinality as $A$ and it is a dense subset of $T$. It follows from Lemma~\ref{L:trees wedge}(b) that $\widetilde{A}\supset I(T)$. 
\end{proof}

\begin{prop}\label{P:tree monolithic}
Any tree is a monolithic space, i.e., the weight and density coincide for each its subset.
\end{prop}

\begin{proof}
Let $T$ be a tree and $A\subset T$ any its infinite subset (for finite sets the statement is trivial). Let $\kappa=\dens(A)$. Let
$$\widetilde{A}=\{x\wedge y\setsep x,y\in A\}.$$
By Lemma~\ref{L:trees wedge}(a) we see that $\widetilde{A}$ is a continuous image of $A\times A$, hence $\dens\widetilde{A}\le\kappa$. By Lemma~\ref{L:trees wedge}(c) $\widetilde{A}$ is invariant for $\wedge$, hence $F=\overline{\widetilde{A}}$ is also invariant for $\wedge$ by Lemma~\ref{L:trees wedge}(d). Clearly $\dens F\le\kappa$. Further, $F$ with the inherited order is a tree and the subspace topology coincides with the coarse wedge topology of $F$ by \cite[Lemma 2.1]{somaglia-treti}. By Lemma~\ref{L:trees vaha} we get that $w(F)=\dens(F)\le \kappa$. Hence $w(A)\le \kappa=\dens A$. Since the converse inequality holds always this completes the proof.
\end{proof}

To present a canonical retractional skeleton on an $r_1$-tree, we introduce for any such $T$ the following family.
$$\A=\A(T)=\{A\in\A^0(T)\setsep x\wedge y\in A\mbox{ whenever }x,y\in A\}$$

\begin{lemma}\label{L:trees continuity} 
Let $T$ be an $r_1$-tree. Then the following hold:
\begin{itemize}
\item[(i)] For any $A\in\A^0$ there is $B\in\A$ with $B\supset A$ and $w(B)\le\max\{w(A),\aleph_0\}$.
\item[(ii)] For any $A,B\in\A$  there is $C\in\A$ such that $C\supset A\cup B$ and $ w(C)\le\max\{w(A),w(B),\aleph_0\}$.
\item[(iii)] If $\A'\subset\A$ is up-directed by inclusion, then $B=\overline{\bigcup\A'}\in\A$ and, moreover,
$$r_B(x)=\lim_{A\in\A'}r_A(x), \quad x\in T.$$
\end{itemize}
\end{lemma}

\begin{proof}
(i) By Proposition~\ref{P:tree monolithic} the weight and density coincide for subsets of trees. So, we can work with densities and, moreover, the density of a subset is not larger than the density of the original set.

Similarly as in \cite{somaglia-trees} choose for any $x\in S(T)\setminus I(T)$  a countable set 
$\phi(x)\subset \widehat{x}\cap I(T)$ with supremum $x$. Now we are ready to provide a proof of (i).

Fix any $A\in\A^0$ and let $\kappa=\max\{w(A),\aleph_0\}$. Let us define by induction the following sequences of sets.
Set $A_0=A$. If $n\in\en$ is given and $A_{n-1}$ is defined we set
$$\begin{aligned}
B_n&=\{x\wedge y,\setsep x,y\in A_{n-1}\},\\
C_n&=B_n\cup\bigcup \{\phi(x)\setsep x\in B_n\cap (S(T)\setminus I(T)), x>\sup\{y<x\setsep y\in B_n\} \},\\
A_n&=\overline{C_n}.
\end{aligned}
$$
It is clear that $A_n$ is closed for each $n\in\en\cup\{0\}$. Further, $B_n$ is closed as well by Lemma~\ref{L:trees wedge}(a). We continue by showing that all the sets $A_n$, $B_n$ and $C_n$ have density at most $\kappa$.

$A_0=A$ has density at most $\kappa$ by the definition of $\kappa$. 
So, assume that $\dens A_{n-1}\le \kappa$. It follows from Lemma~\ref{L:trees wedge}(a) that $B_n$ is a closed set of density at most $\kappa$. Further, by Lemma~\ref{L:trees wedge}(c) it is invariant for $\wedge$, so by  \cite[Lemma 2.1]{somaglia-treti} the topology on $B_{n}$ coincides with the coarse wedge topology generated by the restricted order. Hence, by Lemma~\ref{L:trees vaha}
$\card I(B_n)\le\kappa$.
Further, clearly
$$\{x\in B_n\cap (S(T)\setminus I(T))\setsep  x>\sup\{y<x\setsep y\in B_n\}\}\subset I(B_n),$$  
hence $\card(C_n\setminus B_n)\le\kappa$. So, $\dens A_n\le \kappa$. 

We set $B=\overline{\bigcup_n A_n}$. Then $B$ is a closed set of density at most $\kappa$.
Let us show that $B\in\A$. Clearly $0\in B$. Further, $B$ is closed to the operation $\wedge$. Indeed, 
by construction $B=\overline{\bigcup_n B_n}$, each $B_n$ is closed to $\wedge$ and the sequence $(B_n)$ is increasing, so we can use Lemma~\ref{L:trees wedge}(d).
It remains to show that $B\in \A^0$. To this end fix any $x\in B\setminus I(T)$ and any $y<x$. 
We need to find $z\in(y,x)\cap B$. Let us distinguish three cases:

Case 1: $x\notin \bigcup_n A_n$. By Lemma~\ref{L:trees wedge}(b) there are $n\in\en$ and $a,b\in A_n$ such that $y<a\wedge b\le x$. Since $a\wedge b\in B_{n+1}\subset A_{n+1}$ we get $a\wedge b<x$. Thus $a\wedge b\in (y,x)\cap B$.

Case 2: $x\in A_n\cap S(T)$ for some $n\in\en$. If $(y,x)\cap A_n=\emptyset$, then $\phi(x)\subset C_{n+1}\subset B$. So, any $z\in\phi(x)\cap (y,x)$ does the job.

Case 3: $x\in A_n\setminus S(T)$ for some $n\in\en$. If $x\in A_0$, the conclusion follows from the assumption $A_0=A\in\A^0$. So, assume $x\notin A_0$. Then there is some $n\in\en$ with $x\in A_n\setminus A_{n-1}$. By Lemma~\ref{L:trees wedge}(b) there are $a,b\in C_n$ such that $y<a\wedge b\le x$. Then $a\wedge b\in B_{n+1}\subset B$. So, it is enough to show that $a\wedge b<x$. Assume that $a\wedge b=x$. Since $\cf(x)$ is uncountable, the assumption that $T$ is an $r_1$-tree implies that $a=x$ or $b=x$, so $x\in C_n$. But $C_n\setminus B_n\subset I(T)$, so $x\in B_n$. Hence $x=c\wedge d$ for some $c,d\in A_{n-1}$.
Using again that $T$ is an $r_1$-tree we deduce that $x=c$ or $x=d$, thus $x\in A_{n-1}$, 
a contradiction.

(ii) This assertion follows from (i) as $A\cup B\in\A^0$.

(iii) Let $\A'\subset\A$ be up-directed by inclusion and $B=\overline{\bigcup \A'}$.

Let us show that $B\in\A$. Clearly $B$ is closed and $0\in B$. Further, each $A\in\A'$ is invariant for $\wedge$ (as $\A'\subset\A$), hence $\bigcup\A'$ is invariant for $\wedge$ (as $\A'$ is up-directed).So, by Lemma~\ref{L:trees wedge}(d) $B$ is invariant for $\wedge$ as well. It remains to show that $B\in\A^0$. So, fix $x\in B$ on a limit level and any $y<x$. We shall prove that there is some $z\in (y,x)\cap B$.

If $x\in\bigcup\A'$, i.e. $x\in A$ for some $A\in\A'$, then there is $z\in(y,x)\cap A\subset (y,x)\cap B$ as $A\in\A\subset\A^0$. 

So, assume $x\in B\setminus \bigcup\A'$. By Lemma~\ref{L:trees wedge}(b) there are $a,b\in \bigcup\A'$ such that $y<a\wedge b\le x$. Since $\A'$ is up-directed, there is some $A\in\A'$ such that $a,b\in A$. Since $A\in\A$, we deduce that $a\wedge b\in A$, so $a\wedge b<x$.
Hence $a\wedge b\in (y,x)\cap B$.

This completes the proof that $B\in\A$. It remains to prove the limit formula for $r_B$.
So, take any $x\in T$.  Let us distinguish the following two cases:

Case 1: $r_B(x)\in A$ for some $A\in\A'$. Then for any $A'\in\A'$ with $A'\supset A$ we have $r_{A'}(x)=r_A(x)=r_B(x)$. This proves the convergence.

Case 2: $y=r_B(x)\notin\bigcup\A'$. Then $y$ is on a limit level of $T$. Indeed, assume that $y\in I(T)$. By Lemma~\ref{L:trees wedge}(b) there are $a,b\in \bigcup \A'$ with $a\wedge b=y$. Since $\A'$ is directed, there is $A\in\A'$ with $a,b\in A$. Then $y=a\wedge b\in A$, a contradiction. 

Let $U$ be an open set containing $y$.
Then there are $z<y$ on an isolated level of $T$ and a finite set $F\subset \ims(y)$ such that $W_z^F\subset U$. In case $y$ has uncountable cofinality, we may and shall assume that $F=\ims(y)$.
Since $y\in\overline{W_z^F\cap \bigcup \A'}$, there are $a,b\in W_z^F\bigcup \A'$ such that $z<a\wedge b\le y$ (by Lemma~\ref{L:trees wedge}(b)). Since $\A'$ is up-directed, there is $A\in\A'$ such that $a,b\in A$,
hence $a\wedge b\in A$. It follows that $a\wedge b<y$. Thus for any $A'\in\A'$ with $A'\supset A$ we have $a\wedge b\in\widehat{x}\cap A'\subset\widehat{y}$, hence
$r_{A'}(x)\in [a\wedge b,y]\subset W_z^F\subset U$.
\end{proof}

Let $\A_\omega=\A_\omega(T)$ denote the family of all the separable sets from $\A$. Then we get the following result.

\begin{prop}\label{P:trees retracni skeleton}
Let $T$ be an $r_1$-tree. Then $(r_A)_{A\in\A_\omega}$ is a retractional skeleton on $T$. The induced subset is $S(T)$.
\end{prop}

\begin{proof}
By Lemma~\ref{L:trees continuity}(ii) $\A_\omega$ is up-directed. If $A\in\A_\omega$, then $r_A(T)=A$, so it is metrizable by Proposition~\ref{P:tree monolithic}, hence the property (i) of retractional skeletons is satisfied. The property (ii) follows from Lemma~\ref{L:trees compatibility}, the property (iii) from Lemma~\ref{L:trees continuity}(iii).

Further, 
$$\bigcup_{A\in\A_\omega} r_A(T)=\bigcup\A_\omega =\{x\in T\setsep \cf(x)\mbox{ is at most countable}\}=S(T).$$
Indeed, the first equality follows from the fact that $r_A(T)=A$ for each $A\in\A_\omega$. Let us prove the second one.

$\subset$: Let $A\in\A_\omega$. Assume that there is some $x\in A$ with uncountable cofinality. Since $A\in\A$, the set  $\widehat{x}\cap A$ is uncountable. Since this set is well ordered and the inherited topology coincides with the order topology, it is not separable. Since separability is hereditary for subsets of $T$ by Proposition~\ref{P:tree monolithic}, $A$ is not separable, which is a contradiction.

$\supset$ If $x\in I(T)$, Then  $\{0,x\}\in\A_\omega$. 
If $x\in S(T)\setminus I(T)$, we can find an increasing sequence $(y_n)$ of elements from $I(T)$ with supremum $x$. Then $\{0,x\}\cup\{y_n\setsep y\in\en\}\in\A_\omega$.

In particular, $\bigcup\A'$ is dense in $T$, hence the property (iv) of retractional skeletons follows from Lemma~\ref{L:trees continuity}(iii). Therefore $(r_A)_{A\in\A_\omega}$ is a retractional skeleton on $T$. The formula for the induced subset follows from the above argument.
\end{proof}

For any $A\in\A$ let 
$$P_A(f)=f\circ r_A,\qquad f\in\C(T).$$
Then we get the following

\begin{prop}\label{P:trees projekcni skeleton}
Let $T$ be an $r_1$-tree. Then $(P_A)_{A\in\A_\omega}$ is a $1$-projectional skeleton on $\C(T)$. The induced subspace of the dual is
$$\begin{aligned}
D(T)&=\{\mu\in \M(T)\setsep \spt\mu\subset S\}\\& =\{\mu\in \M(T)\setsep \mu(\{x\})
=0\mbox{ whenever $\cf(x)$ is uncountable}\}.\end{aligned} $$
\end{prop}

\begin{proof}
$(P_A)_{A\in\A_\omega}$ is a $1$-projectional skeleton by Proposition~\ref{P:trees retracni skeleton} and Proposition~\ref{P:rs->ps}(a). The latter result also yields that the induced subspace is
$$D(T)=\{\mu\in\M(T)\setsep \spt\mu \mbox{ is a separable subset of }S(T)\}.$$
By \cite[Proposition 2.5]{somaglia-treti} any $\mu\in\M(T)$ has separable support, which proves the first equality. 
Let us show the second one. The inclusion $\subset$ is obvious, let us prove the converse one.

I.e., assume that $\mu\in\M(T)$ is such that $\mu(\{x\})=0$ whenever $\cf(x)$ is uncountable.
Let $\mu=\mu_d+\mu_c$, where $\mu_d$ is a discrete measure and $\mu_c$ is a continuous measure. Let $C=\{x\in T\setsep \mu(\{x\})\ne 0\}$. Then $C$ is a countable subset of $S(T)$,
thus $\overline{C}\subset S(T)$. Since $\spt \mu_d=\overline{C}$, we deduce $\spt\mu_d\subset S(T)$.

It remains to prove that $\spt\mu_c\subset S(T)$ as well. Assume that $x\in T$ with $\cf(x)$ uncountable. We know that $\mu_c(\{x\})=0$, hence also $\abs{\mu_c}(\{x\})=0$. Since $\abs{\mu_c}$ is regular, there is a sequence $(y_n)$ in $\widehat{x}\cap I(T)$ such that $\abs{\mu_c}(W_{y_n}^{\ims F})<\frac1n$. Let $y=\sup_n y_n$. This supremum exists as $(y_n)$ belongs to the well-ordered set $\widehat{x}$, Moreover, $y<x$ as $\cf(x)$ is uncountable. Let $z\in\ims(y)$ be such that $z\le x$. Then $W_z^{\ims(x)}$ is a neighborhood of $x$ such that $\abs{\mu_c}(W_z^{\ims(x)})=0$. Thus $x\notin \spt\mu_c$.
\end{proof}

Let us now provide a Markushevich basis of $\C(T)$ which is a generalization of the canonical Markushevich basis of $\C([0,\eta])$. Note that by the following proposition $\C(T)$ admits a strong Markushevich basis for an arbitrary $T$, projectional skeleton is not required.
We will further discuss its properties in case $T$ is an $r_1$-tree.
We start by defining the respective basis.

For $x\in I(T)$ let $g_x=\chi_{V_x}$. Then $g_x\in\C(T)$. Further,
set
$$\nu_0=\delta_0,\quad \nu_x=\delta_x-\delta_{x^-}, x\in I(T)\setminus\{0\},$$
where $x^-$ denotes the immediate predecessor of $x$.

\begin{prop}\label{P:trees Mbaze ocasni} Let $T$ be a tree.
\begin{itemize}
\item[(a)] The family $(g_x,\nu_x)_{x\in I(T)}$ is a strong Markushevich basis in $\C(T)$.
\item[(b)] Assume $T$ is an $r_1$-tree. Let $H=\{g_x\setsep x\in I(T)\}\cup\{0\}$. Then the following hold.
\begin{itemize}
\item[(b-i)] $H$ is $\sigma(\C(T),D(T))$-closed and $P_A(H)\subset H$ for each $A\in\A_\omega$.
\item[(b-ii)] Nonzero $\sigma(\C(T),D(T))$-accumulation points of $H$ are exactly functions
$$g_x,\quad x\in I(T)\setminus\{0\}, \cf(x^-) \mbox{uncountable}.$$
\item[(b-iii)] $0$ is a $\sigma(\C(T),D(T))$-accumulation point of $H$ if and only if the set of all the maximal elements of $T$ is either infinite or contains an element on a limit level of uncountable cofinality.
\item[(b-iv)] $P_A(H)\subset H$ for each $A\in\A_\omega(T)$.
\end{itemize}
\end{itemize}
\end{prop}

\begin{proof}
(a) It is clear that $(g_x,\nu_x)_{x\in I(T)}$ is a biorthogonal system, i.e., the first property of Markushevich bases is fulfilled. Let us continue by the third one, i.e., by showing that the measures $\nu_x$, $x\in I(T)$ separate points of $\C(T)$. To this end fix $f\in\C(T)\setminus\{0\}$. There is some $y\in T$ with $f(y)\ne 0$. Recall that $\widehat{y}$ is well ordered, so we can take the smallest $x\in\widehat{y}$ with $f(x)\ne0$. Since $f$ is continuous, necessarily $x\in I(T)$. Moreover, clearly $\nu_x(f)\ne0$.

To finish the proof we will need the following property of measures on $T$:
\begin{equation}\label{eq:tauadditive}
\begin{aligned}
\forall\nu\in\M(T)\,\forall C\subset I(T)&\mbox{ consisting of mutually incomparable elements}: \\&\nu\left(\bigcup_{x\in C}V_x\right)=\sum_{x\in C}\nu(V_x)\quad\mbox{(the series converges absolutely).}\end{aligned}
\end{equation}
Indeed, since $C$ consist of mutually incomparable elements on isolated levels, the family $V_x$, $x\in C$, is a disjoint family of open sets. Therefore the equality follows from $\tau$-additivity of Radon measures.

The second property of Markushevich bases follows from the stronger property defining strong Markushevich bases. Fix $f\in\C(T)$. Set 
$$A=\{x\in I(T)\setsep \nu_x(f)\ne 0\},\quad M=\{g_x\setsep x\in A\}.$$
The proof will be complete if we show $f\in\clin M$. To this end we will use the Hahn-Banach theorem.
So, fix any $\mu\in\M(T)$ such that $\mu|_M=0$. We are going to show that $\mu(f)=0$.
If $f=0$, the assertion is trivial, so suppose $f\ne 0$. If $f$ is constant, then $f=f(0)\ne0$, thus $0\in A$ and $\chi_{V_0}=1\in M$. Therefore
$$\mu(f)=f(0)\mu(V_0)=0.$$
So, assume $f$ is not constant. We will construct by transfinite induction subsets $T_\alpha\subset T$ and $R_\alpha\subset T$ as follows.

Set $T_0=\emptyset$.

Assume that $\alpha>0$ and that we have constructed $T_\beta$ for $\beta<\alpha$. Assume moreover that $(T_\beta)_{\beta<\alpha}$ is an increasing transfinite sequence of closed sets which are also downward closed (i.e., $\widehat{x}\subset T_\beta$ whenever $x\in T_\beta$ for $\beta<\alpha$). We define $R_\alpha$ to be the set of all the minimal elements from $T\setminus \bigcup_{\beta<\alpha} T_\beta$. Note that $R_\alpha$ consists of mutually incomparable elements of $T$ and $R_1=\{0\}$.
Set
$$T_{\alpha}=\bigcup_{\beta<\alpha}T_\beta\cup\bigcup_{x\in R_\alpha}\{ y\in V_x\setsep f\mbox{ is constant on }[x,y]\}.$$
It is clear that $T_{\alpha}\supset \bigcup_{\beta<\alpha}T_\beta\cup R_\alpha$ and it is downward closed. Further, it is also a closed subset of $T$. Indeed, fix any $y\in T\setminus T_{\alpha}$.
Then, in particular, $y\in T\setminus \bigcup_{\beta<\alpha} T_\beta$, thus there is $x\in R_\alpha$ with $x\le y$.
Since $x\in T_{\alpha}$, necessarily $x<y$. Further, $f$ is not constant on $[x,y]$. Let $z\in[x,y]$ the the smallest element with $f(z)\ne f(x)$. Then $z>x$ and, by continuity of $f$, it is on an isolated level. So, $V_z$ is an open subset of $T$. Further, clearly $y\in V_z\subset T\setminus T_{\alpha}$. Hence $T\setminus T_{\alpha}$ is open, so $T_{\alpha}$ is closed.

Observe that 
$$T=\bigcup_{\alpha<\omega_1}T_\alpha.$$
Indeed, assume that $T\setminus \bigcup_{\alpha<\omega_1}T_\alpha\ne\emptyset$.
So, fix a minimal $x\in T\setminus \bigcup_{\alpha<\omega_1}T_\alpha$. For each $\alpha\in[1,\omega_1)$ let $x_\alpha$ be the unique element of $\widehat{x}\cap R_\alpha$.
By construction the net $(x_\alpha)$ is strictly increasing and has supremum $x$. It follows that $\cf(x)=\aleph_1$. But $f$, being continuous on $\widehat{x}$, is constant on $[y,x]$ for some $y<x$. Let $\alpha<\omega_1$ be such that $x_\alpha>y$. Then $f(x_{\alpha+1})=f(x_\alpha)$, a contradiction.

Further, for each $\alpha\in[1,\omega_1)$ isolated we have $R_\alpha\subset A$, so $\chi_{V_x}=g_x\in M$ whenever $x\in R_\alpha$. Therefore
$$\forall \alpha\in[1,\omega_1)\mbox{ isolated }\forall x\in R_\alpha: \mu(V_x)=0,$$
in particular 
$$\mu(T\setminus T_\alpha)=\mu(\bigcup_{x\in R_\alpha} V_x)=\sum_{x\in R_\alpha} \mu(V_x)=0$$
for any isolated $\alpha\in[1,\omega_1)$ isolated (by \eqref{eq:tauadditive}).

We further claim that $\int_{T_\alpha}f\di\mu=0$ for each $\alpha<\omega_1$. 

Let us start by proving it for $\alpha=1$. Note that $f=f(0)$ on $T_1$. So, if $f(0)=0$, the integral is zero by trivial reasons. Assume $f(0)\ne0$. Then $0\in A$, hence $1=g_0\in M$, so 
$\mu(T)=0$. It follows that
$$\int_{T_1} f\di\mu=f(0)\mu(T_1)=f(0)(\mu(T)-\mu(T\setminus T_1))=0.$$

Next assume that $\alpha\in[1,\omega_1)$ and $\int_{T_\alpha}f\di\mu=0$. Then
$$\begin{aligned}
\int_{T_{\alpha+1}\setminus T_\alpha} f\di\mu&=\sum_{x\in R_{\alpha+1}} \int_{V_x\cap T_{\alpha+1}} f\di\mu
=\sum_{x\in R_{\alpha+1}} f(x)\cdot\mu(V_x\cap T_{\alpha+1})
\\&=\sum_{x\in R_{\alpha+1}} f(x)\cdot\mu\left(V_x\setminus\bigcup_{y\in V_x\cap R_{\alpha+2}}V_y\right)
\\&=\sum_{x\in R_{\alpha+1}} f(x)\left(\mu(V_x) - \sum_{y\in V_x\cap R_{\alpha+2}}\mu(V_y)\right)=0\end{aligned}
$$
Thus $\int_{T_{\alpha+1}}f\di\mu=0$.

Finally, assume $\alpha<\omega_1$ is limit and $\int_{T_\beta}f\di\mu=0$ for each $\beta\in[1,\alpha)$. 
Let $T_\alpha^0=\bigcup_{\beta<\alpha}T_\beta$. Then $T_\alpha^0$ is an $F_\sigma$ set and $\int_{T_\alpha^0}f\di\mu=0$ by sigma-additivity of $\mu$. 

If $R_\alpha=\emptyset$, i.e., $T_\alpha^0=T$, the proof is completed. So, assume $R_\alpha\ne\emptyset$. Note that $R_\alpha$ is a Borel set, as $H=T_\alpha^0\cup R_\alpha$ is closed.
Moreover, since $H$ is also downward closed, we have $H\in\A$, thus the retraction $r_H$ is continuous. We claim that the set $R_\alpha$ is $r_H(\mu)$-null, that is
$$\forall B\subset R_\alpha\mbox{ Borel}:r_H(\mu)(B)=0.$$
Indeed, let $\B$ denote the set of all the Borel subsets of $R_\alpha$ which have $r_H(\mu)$-measure zero. Observe that $V_y\cap R_\alpha\in\B$ for any $y\in \bigcup_{\beta<\alpha}R_{\beta+1}$. Indeed, let $y\in R_{\beta+1}$ for some $\beta<\alpha$. Then
$$r_H(\mu)(V_y\cap R_{\alpha})=\mu(r_H^{-1}(V_y\cap R_\alpha))=
\mu\left(\bigcup_{x\in V_y\cap R_\alpha} V_x\right) =\mu\left(\bigcap_{\gamma\in (\beta,\alpha)}
\bigcup_{z\in V_y\cap R_{\gamma+1}} V_z\right)=0.
$$
The first two equalities follow from definitions. The third one follows from the equality of the respective sets, which we are going to prove.

$\subset$: Assume $x\in V_y\in R_\alpha$. Let $\gamma\in(\beta,\alpha)$ be arbitrary. Since $y\in R_{\beta+1}\subset T_\gamma$ and $x\notin T_\gamma$, there is (a unique) $z\in R_{\gamma+1}$ with $z\in(y,x)$. Then $z\in V_y\cap R_{\gamma+1}$ and $V_x\subset V_z$.

$\supset$: Assume that for each $\gamma\in(\beta,\alpha)$ there is some $z_\gamma\in V_y\cap R_{\gamma+1}$ with $u\in V_{z_\gamma}$. Then $(z_\gamma)_{\gamma\in(\beta,\alpha)}$ is an increasing net in $(y,u)$. Denote its supremum by $x$. Then $u\ge x$, so $u\in V_x$. Moreover, $x\in R_\alpha$ by construction. The last equality follows from \eqref{eq:tauadditive}.

Further, the sets of the form $R_\alpha\cap V_y$, $y\in \bigcup_{\beta<\alpha} R_{\beta+1}$, form a basis of the topology of $R_\alpha$. This basis is $\sigma$-disjoint and closed to finite intersections. It follows that each open set belong to $\B$, thus $\B$ contains all Borel sets.
Thus
$$\begin{aligned}
0&=\int_{R_\alpha} f\di r_H(\mu) 
= \int_{{T\setminus T_\alpha^0}} f\circ r_H\di\mu
= \int_{T_{\alpha}\setminus T_\alpha^0} f\circ r_H\di\mu + \int_{{T\setminus T_\alpha}} f\circ r_H\di\mu  
\\&=\int_{{T_{\alpha}\setminus T_\alpha^0}}f\di\mu
+ \sum_{x\in R_{\alpha+1}}f(r_H(x))\mu(V_x)
= \int_{{T_{\alpha}\setminus T_\alpha^0}}f\di\mu,\end{aligned}$$
hence $\int_{T_\alpha}f\di\mu=0$, completing the induction argument.

Finally, since $(T_\alpha)_{\alpha<\omega_1}$ is an increasing transfinite sequence of closed sets covering $T$ and $\spt\mu$ is separable (see \cite[Proposition 2.5]{somaglia-treti}), there is some $\alpha<\omega_1$ such that $\spt\mu\subset T_\alpha$. It follows that $\int_T f\di\mu=0$ which completes the proof.

(b) Assume that $T$ is an $r_1$-tree. To prove (b-i) we observe that
$$ H=\{f\in\C(T)\setsep \forall x\in S(T): f(x)\in\{0,1\} 
\& \forall x,y\in S(T): x\wedge y\in S(T)\Rightarrow f(x\wedge y)=f(x)\cdot f(y)\}.
$$
Indeed, the inclusion $\subset$ is obvious. To prove the converse one fix any $f$ in the set on the right-hand side. Since $S(T)$ is dense, $f$ attains only the values $0$ and $1$, so $f=\chi_A$ for a clopen set $A\subset T$. Given $x\in A$, we have $f(x)=1$. By continuity of $f$ we can find some $y\in\widehat{x}\cap I(T)$ with $f(y)=1$. So, we get $V_y\cap S(T)\subset A$. Since $S(T)$ is dense, we deduce $V_y\subset A$. It follows that $A$ is covered by sets $V_y$, $y\in A\cap I(T)$. By compactness we can find a finite subcover. Moreover, this subcover can be disjoint (as any two sets of the form $V_y$ are either disjoint or one of them contains the other). We claim that this cover contains only one set. Indeed, given any two points $y,z\in A\cap I(T)$ such that the sets $V_y$ and $V_z$ belong to the subcover and are disjoint, we deduce that $y$ and $z$ are incomparable, thus $y\wedge z\in S(T)$ (as $T$ is an $r_1$-tree), so $y\wedge z\in A$. It follows that there is some $x\in I(T)\cap A$ such that $V_z$ belongs to the subcover and $y\wedge z\in V_x$. But then $V_y\cup V_z\subset V_x$, a contradiction with the assumption that the subcover is disjoint.

So, the equality is proved. Finally, it is clear that the set of the right-hand side is $\tau_p(S(T))$-closed and hence, a fortiori, $\sigma(\C(T),D(T))$-closed.

Let us continue by proving the assertion (b-ii). Let $x\in I(T)$. If $x=0$, then $g_x$ is an isolated point of $H$ as
$$\{f\in\C(T)\setsep f(0)\ne0\}\cap H=\{g_0\}.$$
If $x\ne 0$ and $x^-\in S(T)$,
then $g_x$ is again an isolated point of $H$ as
$$\{f\in\C(T)\setsep f(x)\ne0\ \&\ f(x^-)\ne1\}\cap H=\{g_x\}.$$
Finally, assume that $x\in I(T)\setminus\{0\}$ and $x^-$ has uncountable cofinality. We are going to prove 
$$g_y \xrightarrow[y\in \widehat{x^-}\cap I(T)]{\sigma(\C(T),D(T))} g_x.$$
To this end take any $\mu\in D(T)$. Then $\spt\mu$ is a compact subset of $S(T)$.
In particular, $x^-\notin\spt\mu$. Thus there is some $y_0\in\widehat{x}\cap I(T)$ such that
$W_{y_0}^{\{x\}}\cap \spt\mu=\emptyset$ (recall that $\ims(x^-)=\{x\}$). Then for each $y\in (y_0,x)\cap I(T)$ we have
$$\int g_y\di\mu=\mu(V_y)=\mu(V_x)=\int g_x\di\mu. 
$$
Hence, the convergence is proved, so $g_x$ is an accumulation point of $H$ and the proof of (b-ii) is completed.

Let us look at (b-iii). Denote by $M$ the set of all the maximal elements of $T$. If there is some $x\in M\setminus S(T)$, then in the same way as above we prove that
$$g_y \xrightarrow[y\in \widehat{x}\cap I(T)]{\sigma(\C(T),D(T))} 0,$$
so $0$ is an accumulation point of $H$.

Next assume that $M$ is infinite. We will construct by induction elements $x_n\in M$ and $y_n\in I(T)$ such that the following conditions are fulfilled for each $n\in\en$.
\begin{itemize}
\item $y_n\le x_n$;
\item $y_n>\max\{y_j\wedge x_n\setsep 1\le j<n\}$;
\item $M\setminus \bigcup_{j=1}^n V_{y_j}$ is infinite.
\end{itemize}

We start by fixing two distinct points $a,b\in M$. Since they are incomparable, $a\wedge b< a$ and $a\wedge b<b$. So, we can fix $c,d\in\ims(a\wedge b)$ with $c\le a$ and $d\le b$. Then $V_c\cap V_d=\emptyset$, hence at least one of the sets $M\setminus V_c$, $M\setminus V_d$ is infinite.
Without loss of generality assume the first case occurs. Then set $x_1=a$ and $y_1=c$ and all the conditions are fulfilled for $n=1$.

Further, assume that $n\in\en$ and $x_j$ and $y_j$ are given for $j\le n$ such that the conditions are fulfilled for $j\le n$. Fix two distinct points $a,b\in M\setminus\bigcup_{j=1}^n V_{y_j}$ (this is possible as the respective set is infinite). Fix $c,d\in I(T)$ such that
$$\max\{a\wedge b,a\wedge y_1,\dots,a\wedge y_n\}<c\le a,\
\max\{a\wedge b,b\wedge y_1,\dots,b\wedge y_n\}<d\le b.$$
Then $V_c$ and $V_d$ are disjoint, hence at least one of the sets
$M\setminus (\bigcup_{j=1}^n V_{y_j}\cup V_c)$, $M\setminus (\bigcup_{j=1}^n V_{y_j}\cup V_d)$
is infinite. Assume without loss of generality that the first case occurs. Then we can set $x_{n+1}=a$ and $y_{n+1}=c$.

Therefore, the construction can be performed. Note that the sets $V_{y_n}$, $n\in\en$, are pairwise disjoint, hence
$$g_{y_n}=\chi_{V_{y_n}}\to 0\mbox{ pointwise on }T,$$
hence also $g_{y_n}\to 0$ weakly in $\C(T)$ (by Lebesgue dominated convergence theorem), hence, a fortiori, $g_{y_n}\to 0$ in $\sigma(\C(T),D(T))$. It follows that $0$ is a $\sigma(\C(T),D(T))$-accumulation point of $H$. This completes the proof of the `if part' of (b-iii).

To prove the `only if part' assume that $M$ is finite and $M\subset S(T)$. Then 
$$U=\{f\in\C(T)\setsep \abs{f(x)}<1\mbox{ for }x\in M\}$$
is a $\sigma(\C(T),D(T))$-neighborhood of $0$ and $U\cap H=\{0\}$. Thus $0$ is an isolated point of $H$.

It remains to prove the assertion (b-iv). So, fix $A\in\A_\omega(T)$. Then, of course, $P_A0=0$.
Further, clearly $P_Ag_x=0$ if $A\cap V_x=\emptyset$. So, assume that $A\cap V_x\ne\emptyset$. Since $A$ is closed and stable to the operation $\wedge$, it follows that the set $A\cap V_x$ admits a minimum, say $y$. Then $P_Ag_x=g_y$.
\end{proof}

The next proposition provides a construction of a projectional generator in the spaces $\C(T)$.

\begin{prop}\label{P:trees PG}
Let $T$ be an $r_1$-tree.
\begin{itemize}
\item For any $\mu\in D(T)$ there is a countable set $C(\mu)\subset I(T)$ such that 
$$\forall x,y\in\spt\mu: x\wedge y=\sup(C(\mu)\cap\widehat{x\wedge y}).$$ 
\item For each $\mu\in D(T)$ set $$\Phi(\mu)=\{g_x\setsep x\in C(\mu)\}.$$  Then the pair $(D(T),\Phi)$ is a projectional generator on $\C(T)$.
\end{itemize} 
\end{prop}

\begin{proof} Let $\mu\in D(T)$. Then $\spt\mu$ is a separable subset of $S(\mu)$.  
Let $C_0(\mu)$ be a countable dense subset of $\spt\mu$. Set $C_1(\mu)=\{x\wedge y\setsep x,y\in C_0(\mu)\}$. Then $C_1(\mu)$ is countable and it is contained in $S(T)$ (as $T$ is an $r_1$-tree). So, we can find a countable subset $C(\mu)\subset I(T)$ such that 
 $$x=\sup(C(\mu)\cap\widehat{x})\mbox{ for }x\in C_1(\mu).$$
Let us show that $C(\mu)$ has the property. Let $x,y\in\spt\mu$. We distinguish the following possibilities:

Case 1: $x$ and $y$ are comparable. Without loss of generality $x\le y$, i.e., $x\wedge y=x$.
This case splits into two subcases:

(a) $x\in I(T)$: By Lemma~\ref{L:trees wedge}(b) there are $a,b\in C_0(\mu)$ such that $a\wedge b=x$. Hence $x\in C_1(\mu)$, so necessarily $x\in C(\mu)$.

(b) $x\notin I(T)$: Fix any $z<x$. By Lemma~\ref{L:trees wedge}(b) there are $a,b\in C_0(\mu)$ such that $z<a\wedge b\le x$. Then $a\wedge b\in C_1(\mu)$, so there is some $y\in C(\mu)\cap(z,a\wedge b)\subset C(\mu)\cap (z,x)$.

Case 2: $x$ and $y$ are incomparable. Then $x\wedge y<x$ and $x\wedge y<y$.  By Lemma~\ref{L:trees wedge}(b) there are $a,b,c,d\in C_0(\mu)$ such that $x\wedge y<a\wedge b\le x$ and $x\wedge y<c\wedge d\le y$. It follows that $(a\wedge b)\wedge (c\wedge d)=x\wedge y$.
By Lemma~\ref{L:trees wedge}(c) we get $x\wedge y\in C_1(\mu)$, thus the conclusion follows.

\smallskip

Hence we can define the mapping $\Phi$. Let us observe that
\begin{equation}
\label{eq:nulovost}
\forall\mu\in D(T)\,\forall\nu\in\Phi(\mu)^\perp\,\forall x,y\in \spt\mu:\nu(V_{x\wedge y})=0.\end{equation}
Indeed, fix any $\mu\in D(T)$ and $\nu\in\Phi(\mu)^\perp$. Then $\nu(V_x)=0$ for each $x\in C(\mu)$. If $x,y\in\spt\mu$, by the choice of $C(\mu)$ there is an increasing sequence $(z_n)$ in $C(\mu)$ with supremum $x\wedge y$. Then
$$\nu(V_{x\wedge y})=\nu(\bigcap_{n}V_{z_n})=\lim_n\nu(V_{z_n})=0.$$

Let us continue by showing that $(D(T),\Phi)$ is a projectional generator. Fix any $M\subset D(T)$ and any measure $\nu\in \clin^{w^*}{M}\cap \Phi(D)^\perp$. Assume that $\nu\ne 0$. It follows that there is some $x\in I(T)$ with $\nu(g_x)\ne 0$.

We perform the following inductive construction.

We have $\nu(g_x)\ne0$. Thus there is some $\mu_1\in M$ with $\mu_1(V_x)=\mu_1(g_x)\ne 0$. Since $V_x$ is clopen, necessarily $\spt\nu_1\cap V_x\ne\emptyset$. Fix $z_1\in \spt\nu\cap V_x$. 

Now assume we have some $z_k\in V_x$ such that $z_k\in\{u\wedge v\setsep u,v\in\spt\mu\}$.
By \eqref{eq:nulovost} we have $\nu(V_{z_k})=0$, so $\nu(V_x\setminus V_{z_k})\ne 0$. In particular, $z_k>x$.  Since $V_x$ is clopen and $V_{z_k}$ is closed, it follows that $\spt\nu\cap (V_x\setminus V_{z_k})\ne\emptyset$. So, fix an element $y_k\in \spt\nu\cap (V_x\setminus V_{z_k})$ and set $z_{k+1}=y_k\wedge z_k$. Then  $z_{k+1}<z_k$, $z_{k+1}\in V_x$ and, moreover,
by Lemma~\ref{L:trees wedge}(c) we deduce $z_{k+1}\in\{u\wedge v\setsep u,v\in\spt\mu\}$.

This completes the induction. So, we have constructed in $V_x$ an infinite decreasing sequence,
which is impossible. This contradiction completes the proof.
\end{proof}

\begin{remark} 
(a) In the several preceding statements we deal with $r_1$-trees, but they admit variants for $r$-trees. One possibility is to use Lemma~\ref{L:trees-r1} to transfer the results. Another possibility is to define a more technical variant of the families $\A(T)$ and $\A_\omega(T)$.

(b) If $T$ is an $r$-tree which is not an $r_1$-tree, then the set $H$ from Proposition~\ref{P:trees Mbaze ocasni}(b) is is not $\sigma(\C(T),D(T))$-closed. It can be shown that its nonzero accumulation points are exactly the characteristic functions of the sets
$$\bigcup_{y\in \ims(x)}V_y,\quad x\in T\setminus S(T).$$

(c) The above-defined Markushevich basis satisfies the properties from Theorem~\ref{T:char sigma subspace}(6,7) if and only if $\htt(T)\le\omega_1+1$ (i.e., $\Lev_{\omega_1+1}(T)=\emptyset$, in other words $\ims(x)=\emptyset$ whenever $\cf(x)$ is uncountable). However, $D(T)$ is a $\Sigma$-subspace in more cases, see \cite[Theorem 3.2]{somaglia-treti}. It follows that, at least in some cases, the canonical Markushevich basis cannot be constructed using a PRI.
\end{remark}

Let us now look at the question when $\C(T)$ is $1$-Plichko. First observe that the following equivalences follow from the results of \cite{somaglia-trees,somaglia-treti}.

\begin{prop}\label{P:trees v=1P} Let $T$ be a tree. The following assertions are equivalent.
\begin{enumerate}
\item $T$ is a Valdivia compact space.
\item $\C(T)$ is $1$-Plichko.
\item $S(T)$ is a $\Sigma$-subset of $T$.
\item $D(T)$ is a $\Sigma$-subspace of $\C(T)^*$.
\end{enumerate}
\end{prop}

\begin{proof}
The implications (4)$\Rightarrow$(3)$\Rightarrow$(1) are obvious. The equivalence (1)$\Leftrightarrow$(2) is proved in \cite[Theorem 4.1]{somaglia-treti}. The implication (1)$\Rightarrow$(3) is easy and follows from the proof of \cite[Proposition 3.2]{somaglia-trees}. Finally, the implication (3)$\Rightarrow$(4) follows from \cite[Proposition 5.1]{ja-survey}.
\end{proof}

A partial characterization of Valdivia trees is given in \cite[Theorem 3.2]{somaglia-treti}, a complete characterization is still missing. We will provide an alternative proof of the assertion (1) of the quoted theorem. The original proof is done by a clever transfinite induction. We are going to present a short proof using Theorem~\ref{T:komutativni} (the transfinite induction is hidden therein). The statement we are going to prove is the content of the following proposition.

\begin{prop}\label{P:Valdiviatrees}
Let $T$ be an $r$-tree such that $\htt(T)<\omega_2$ and the set
$$R(T)=\{x\in T\setsep \cf(x)\mbox{ is uncountable and }\ims(x)\ne0\}$$
can be expressed as the union of $\omega_1$-many relatively discrete sets.
Then $T$ is Valdivia.
\end{prop}

We will use the following lemma characterizing  $\sigma(\C(T),D(T))$-to-$\sigma(\C(T),D(T))$ continuity of projections $P_A$. Assume that $T$ is an $r_1$-tree.
Then for any $\A'\subset \A_\omega(T)$ up-directed we have $A=\overline{\bigcup\A'}\in\A$ and the projection $P_{\A'}$ from \eqref{eq:P_A def} coincides with the projection $P_A$ (due to Lemma~\ref{L:trees compatibility} and Lemma~\ref{L:compact-projekce}).

\begin{lemma}\label{L:trees P_A spoj} Let $T$ be an $r_1$-tree and $A\in\A(T)$. The following are equivalent.
\begin{enumerate}
\item $P_A$ is $\sigma(\C(T),D(T))$-to-$\sigma(\C(T),D(T))$ continuous.
\item $r_A(S(T))\subset S(T)$.
\item $\ims(x)\subset A$ whenever $x\in A$ and $\cf(x)$ is uncountable.
\end{enumerate}
\end{lemma}

\begin{proof}
(1)$\Rightarrow$(2) Assume $P_A$ is $\sigma(\C(T),D)$-to-$\sigma(\C(T),D)$ continuous.
By Proposition~\ref{P:sigma(X,D)-spojitost} we get $P_A^*(D)\subset D$. Note that
$P_A^*(\mu)=r(\mu)$ by Lemma~\ref{L:compact retrakce}(c), in particular $P_A^*(\delta_x)=\delta_{r_A(x)}$ for any $x\in T$. It follows that $r_A(S)\subset S$.

(2)$\Rightarrow$(1) Assume $r_A(S)\subset S$. We claim that $P_A^*(D)\subset D$. To show that
fix $\mu\in D$. Let $F=\spt\mu$. Then $F$ is a compact separable subset of $D$. Thus $r_A(F)$ is also a compact separable subset of $D$. Moreover, $\spt P_A^*\mu\subset r_A(F)$. Indeed,
if $B\subset T\setminus r_A(F)$ is any Borel set, then
$$P_A^*\mu(B)=r_A(\mu)(B)=\mu(r_A^{-1}(B))=\mu(\emptyset)=0.$$
So, we have proved that $P_A^*(D)\subset D$. The  $\sigma(\C(T),D)$-to-$\sigma(\C(T),D)$ continuity of $P_A$ now follows from Proposition~\ref{P:sigma(X,D)-spojitost}.

The equivalence (2)$\Leftrightarrow$(3) is obvious.
\end{proof}

\begin{proof}[Proof of Proposition~\ref{P:Valdiviatrees}.] If $R(T)=\emptyset$, then $P_A$ is $\sigma(\C(T),D(T))$-to-$\sigma(\C(T),D(T))$ continuous for each $A\in\A(T)$ (by Lemma~\ref{L:trees P_A spoj}). It follows  by Theorem~\ref{T:komutativni} that $D(T)$ is a $\Sigma$-subspace, hence $T$ is Valdivia by Proposition~\ref{P:trees v=1P}.

Assume that $R(T)\ne\emptyset$. Let $\eta=\htt(T)$. Then $\eta>\omega_1+1$. Moreover, by the assumption $\eta<\omega_2$, thus $\card\eta=\aleph_1$. So, we can fix a bijection $\xi:I(\omega_1)\to I(\eta)$.
By Lemma~\ref{L:trees-r1} we can assume that $T$ is an $r_1$-tree. Fix a disjoint decomposition
$$R(T)=\bigcup_{\alpha<\omega_1}R_\alpha,$$
where each $R_\alpha$ is relatively discrete. 

Fix $\alpha<\omega_1$. For any $x\in R_\alpha$ there is a neighborhood $U$ of $x$ with $U\cap R_\alpha=\{x\}$. We can choose such a basic neighborhood, so there is $z(x)\in I(T)\cap \widehat{x}$ such that $W_{z(x)}^{\ims(x)}\cap R_\alpha=\{x\}$. Observe that the family 
 $W_{z(x)}^{\ims(x)}$, $x\in R_\alpha$ is disjoint. Indeed, let $x,y\in R_\alpha$ be two distinct points. If the points $z(x)$ and $z(y)$ are incomparable, then even $V_{z(x)}$ and $V_{z(y)}$ are disjoint. Assume that $z(x)$ and $z(y)$  are comparable, without loss of generality $z(x)\le z(y)$. Since $y\notin W_{z(x)}^{\ims(x)}$, necessarily $y>x$. Further, $x\notin W_{z(y)}^{\ims(y)}$, thus $z(y)>x$. Hence  $W_{z(x)}^{\ims(x)} \cap W_{z(y)}^{\ims(y)}=\emptyset$.

Let us define a subfamily of $\A_\omega(T)$ by the formula
\begin{multline*}
\A'= \{A\in\A(\omega) \setsep \forall \alpha\in I(\omega_1): 
A\cap\Lev_{\xi(\alpha)}(T)\ne\emptyset\\ \Rightarrow \forall\beta\le\alpha\,\forall x\in R_\beta:
(A\cap W_{z(x)}^{\ims(x)}\ne\emptyset \Rightarrow \ims(x)\subset A)\}
\end{multline*}

Let us show that $\A'$ is a cofinal and $\sigma$-closed subfamily of $\A_\omega(T)$.

Let $(A_n)$ be an increasing sequence in $\A'$. We will show that $A=\overline{\bigcup_n A_n}\in\A'$. 
Clearly we have $A\in\A_\omega(T)$. Further,  fix any $\alpha\in I(\omega_1)$ such that $A\cap \Lev_{\xi(\alpha)}(T)\ne\emptyset$, $\beta\le\alpha$ and $x\in R_\beta$ such that
$A\cap W_{z(x)}^{\ims(x)}\ne\emptyset$. 
Since $W_{z(x)}^{\ims(x)}$ is an open set, there is some $m\in \en$ with $A_m\cap W_{z(x)}^{\ims(x)}\ne\emptyset$.
Further, choose some $y\in A\cap\Lev_{\xi(\alpha)}(T)$. Since $\xi(\alpha)$ is an isolated ordinal, Lemma~\ref{L:trees wedge}(b) yields $a,b\in\bigcup_n A_n$ with $y=a\wedge b$. Since the sequence $(A_n)$ is increasing, there is some $n\in\en$ with $a,b\in A_n$. Then $y=a\wedge b\in A_n$ as well. So, $\Lev_{\xi(\alpha)}(T)\cap A_k\ne\emptyset$ for $k\ge n$. 
It follow that $\ims(x)\subset A_k$ for $k\ge\max\{m,n\}$. This  shows that $A\in\A'$ which completes the proof that $\A'$ is $\sigma$-closed.

Let us continue by showing that $\A'$ is cofinal. To this end fix any $A_0\in\A_\omega(T)$.
Given $A_{n-1}\in\A_\omega(T)$ for some $n\in\en$ we perform the following construction.
\begin{itemize}
\item Set $J_n=\{\alpha\in I(\omega_1)\setsep A_{n-1}\cap \Lev_{\xi(\alpha)}(T)\ne\emptyset\}$. Then $J_n$ is countable (by Lemma~\ref{L:trees vaha}).
\item For any $\beta\le\sup J_n$ let
$$M_\beta^n=\{x\in R_\beta\setsep W_{z(x)}^{\ims(x)}\cap A_{n-1}\ne\emptyset\}.$$
Then $M_\beta^n$ is countable
\item Choose $A_n\in\A_\omega(T)$ such that $A_n\supset A_{n-1}\cup\bigcup \{\ims(x)\setsep x\in\bigcup_{\beta\le\sup J_n}M_\beta^n\}$.
\end{itemize}
Finally, let $A=\overline{\bigcup_n A_n}$. Then $A\in\A'$. Indeed, clearly $A\in\A_\omega(T)$. Moreover, let  $\alpha\in I(\omega_1)$ with $A\cap \Lev_{\xi(\alpha)}(T)\ne\emptyset$, $\beta\le\alpha$ and $x\in R_\beta$ such that $A\cap W_{z(x)}^{\ims(x)}\ne\emptyset$. 
As above there are some $m,n\in\en$ such that $A_m\cap W_{z(x)}^{\ims(x)}\ne\emptyset$ and
$A_n\cap \Lev_{\xi(\alpha)}(T)\ne\emptyset$. Let $k\ge\max\{m,n\}$. Then $\beta\le\sup J_k$
and $x\in M_\beta^k$. It follows that $\ims x\subset A_{k+1}\subset A$.

Finally, let $\A''\subset\A'$ be any directed subfamily. Set $A=\overline{\bigcup\A''}$. Let us show that $P_A$ is $\sigma(\C(T),D(T))$-to-$\sigma(\C(T),D(T))$-continuous using Lemma~\ref{L:trees P_A spoj}. So, fix $x\in R(T)\cap A$. Fix $\beta<\omega_1$ with $x\in R_\beta$. Set
$$\gamma=\sup\xi([0,\beta+1]\cap I(\omega_1))\cap [0,\htt(x))$$ and find some $z\in[z(x),x)\cap I(T)$ with $\htt(z)>\gamma$. Then $V_z^{\ims(x)}$ is a neighborhood of $x$, thus $x\in\overline{V_z^{\ims(x)}\cap\bigcup\A''}$. By Lemma~\ref{L:trees wedge}(b) there are $a,b\in\bigcup\A''$ with $z<a\wedge b\le x$. Since $\A''$ is directed, there is some $B\in\A''$ with $a,b\in B$. Then $a\wedge b\in B$, so $a\wedge b<x$. Further, since $B\in\A(T)$, there is some
$c\in B\cap(z,a\wedge b]\cap I(T)$. Let $\alpha=\xi^{-1}(\htt(c))$. Then $\alpha\ge\gamma>\beta$
and $W_{z(x)}^{\ims(x)}\cap B\ne\emptyset$. It follows that $\ims(x)\subset B\subset A$.

This completes the proof (using Theorem~\ref{T:komutativni}).
\end{proof}

\subsection{Duals of Asplund spaces}

The third class of spaces having a possibly non-commutative projectional skeleton is the class of duals of Asplund spaces. Asplund spaces can be even characterized in this way. These characterizations are summarized in the following theorem.

\begin{thm}\label{T:asplund}
Let $X$ be a Banach space. The following assertions are equivalent.
\begin{enumerate}
\item $X$ is Asplund, i.e., $Y^*$ is separable for each $Y\subset X$ separable.
\item There is a projectional generator on $X^*$ of the form $(X,\Phi)$ (i.e., its domain is $X$).
\item There is an $\omega$-monotone mapping $\psi:[X]^{\le\omega}\to[X\cup X^*]^{\le\omega}$ such that 
 $$\bigcup \{\psi(C)\cap X^*\setsep C\in[X]^{\le\omega}\}$$ is dense in $X^*$ and, moreover, 
for each $C\in[X]^{\le\omega}$ we have
\begin{itemize}
\item $\psi(C)\supset C$;
\item $\overline{\psi(C)\cap X}$ and $\overline{\psi(C)\cap X^*}$ are linear subspaces;
\item the mapping  $x^*\mapsto x^*|_{\overline{\psi(C)\cap X}}$ maps $\overline{\psi(C)\cap X^*}$ isometrically onto $(\overline{\psi(C)\cap X})^*$;
\end{itemize}
\item There is an $\omega$-monotone mapping $G:[X]^{\le\omega}\to[X^*]^{\le\omega}$ such that
\begin{itemize}
\item $\bigcup \{G(C)\setsep C\in[X]^{\le\omega}\}$ is dense in $X^*$;
\item For each $C\in [X]^{\le\omega}$ the mapping $x^*\mapsto x^*|_{\clin C}$ maps $\overline{G(C)}$ onto $(\clin C)^*$.
\end{itemize}
\item There is a projectional skeleton on $X^*$ such that the induced subspace contains $X$.
\end{enumerate}
\end{thm}

The equivalences from this theorem are known. The equivalence (1)$\Leftrightarrow$(5) follows from
 \cite[Proposition 26 and Theorem 15]{kubis-skeleton} (see also \cite[Remark on p. 1628]{CuFa-AsplundWCG}). The equivalence (1)$\Leftrightarrow$(2) is proved in \cite[Proposition 8.2.1]{fabian-kniha} and the equivalence (1)$\Leftrightarrow$(4) follows from the proof of  \cite[Theorem 2.3]{CuFa-Asplund}. We provide a proof and point out what is deep and what is easy. 

\begin{proof}
The implication (1)$\Rightarrow$(2) is the deep one and is proved in \cite[Proposition 8.2.1]{fabian-kniha}. Let us recall just a sketch of the proof. Let $X$ be an Asplund space. It follows from the Jayne-Rogers selection theorem \cite[Theorem 8.1.2]{fabian-kniha} that there is a function $g:X\to X^*$ with the properties
\begin{itemize}
\item $\norm{g(x)}=1$ and $g(x)(x)=\norm{x}$ for each $x\in X$;
\item $g$ is of the first Baire class.
\end{itemize}
It follows there is a sequence $(g_n)$ of continuous functions $g_n:X\to B_{X^*}$ which pointwise converges to $g$. If we take
$$\Phi(x)=\{g_n(x)\setsep n\in\en\},\quad x\in X,$$
then the pair $(X,\Phi)$ is a projectional generator. 

Indeed, assume $M\subset X$ is such that $\overline{M}$ is a linear subspace and that there is some $$x^{**}\in \wscl{M\cap B_X}\cap \Phi(M)^\perp\setminus\{0\}.$$
Since the functions $g_n$ are continuous, we deduce that $x^{**}\in\Phi(\overline{M})^\perp$, so, without loss of generality $M$ is a closed linear subspace of $X$.
Fix some $x^*\in X^*$ such that $x^{**}(x^*)\ne0$. We will construct by induction points $x_n\in B_X\cap M$ and $y^*_{n,k}\in X^*$ for $k,n\in\en$
such that the following conditions are fulfilled for each $n\in\en$.
\begin{itemize}
\item $\abs{x^*(x_n)-x^{**}(x^*)}<\frac1n$,
\item $\abs{y_{m,k}^*(x_n)}<\frac1n$ for $m,k<n$,
\item $\{y_{n,k}^*\setsep n\in\en\}=\Phi(\lin_\qe\{x_1,\dots,x_k\})$.
\end{itemize}
It is clear that the construction can be performed. Set $$V_0=\lin_\qe\{x_n\setsep n\in\en\}\mbox{ and }V=\clin\{x_n\setsep n\in\en\}=\overline{V_0}.$$
Let $z^{**}$ be any weak$^*$-cluster point (in $X^{**}$) of the sequence $(x_n)$. Then $z^{**}(x^*)=x^{**}(x^*)\ne0$.
Further $z^{**}(y_{n,k})=0$ for $k,n\in\en$, hence $z^{**}\in\Phi(V_0)^\perp$. By the continuity of the functions $g_n$ we deduce $z^{**}\in\Phi(V)^\perp$, hence also $z^{**}(g(x))=0$ for each $x\in V$. 

Let $J:V\to X$ be the canonical isometric inclusion. Then $J^*:X^*\to V^*$ is the restriction map and $J^{**}:V^{**}\to X^{**}$ is an isometric inclusion with range $\wscl V$.
Since $z^{**}\in\wscl V$, we can define $v^{**}=(J^{**})^{-1}z^{**}$. Then $v^{**}\in V^{**}\setminus \{0\}$ and for each $x\in V$ we have
$$v^{**}(g(x)|_V)=v^{**}(J^*g(x))=J^{**}(v^{**})(g(x))=z^{**}(g(x))=0.$$
Further, for each $x\in V$ we have $\norm{g(x)|_V}\le 1$ and $g(x)(x)=\norm{x}$. It follows that $\{g(x)|_V\setsep x\in V\}$ is a James boundary for $V$. Since $V^*$ is separable, by \cite[Theorem 3.122]{FHHMZ}  (or by Rod\'e's theorem -- see \cite{rode} or \cite[Theorem 5.7]{FLP}) we deduce $\clin\{g(x)|_V\setsep x\in V\}=V^*$, so $v^{**}=0$, a contradiction.

Finally, note that the proof was done for real spaces, but the complex case easily follows. Indeed, if $X$ is a complex Asplund space, its real version is a real Asplund space and the projectional generator for the real version works for the complex case as well.

(2)$\Rightarrow$(3) This implication is rather easy, it follows essentially from the proof of \cite[Lemma 6.1.3]{fabian-kniha}.
We will provide a proof in the real case. The proof in the complex case is the same, one just needs to replace $\qe$ by $\qe+i\qe$ at the appropriate places.

SO, let $(X,\Phi)$ be a projectional generator. Further,
for each $x^*\in X^*$ let $\eta(x)\subset B_{X}$ be a countable set such that $\norm{x^*}=\sup\{\abs{x^*(x)}\setsep x\in\eta(x^*)\}$.

Fix any $C\in[X]^{\le\omega}$. Let $\psi_0(C)=\lin_\qe C$ and define for $n\in\en\cup\{0\}$ by induction 
\begin{itemize}
\item $\psi_{2n+1}(C)=\psi_{2n}(C)\cup \lin_\qe (\psi_{2n}(C)\cap X^*\cup \Phi(\psi_{2n}(C)\cap X))$,
\item $\psi_{2n+2}(C)=\psi_{2n+1}(C)\cup\lin_\qe (\psi_{2n+1}(C)\cap X\cup\eta(\psi_{2n+1}(C)\cap X^*))$.
\end{itemize}
Clearly the mappings $\psi_n$ are $\omega$-monotone, thus the mapping $\psi$ defined by
$$\psi(C)=\bigcup_n \psi_n(C)$$
is $\omega$-monotone as well. We will show that $\psi$ is the sought mapping.

Fix any $C\in[X]^{\le\omega}$. Then clearly $C\subset\psi(C)$ and both $\psi(C)\cap X$ and $\psi(C)\cap X^*$ are countable $\qe$-linear spaces, hence their closures are linear spaces. Further, for any $x^*\in\psi(C)\cap X^*$ we have $\eta(x^*)\subset\psi(C)\cap X$, thus
$\norm{x^*}=\norm{x^*|_{\overline{\psi(C)\cap X}}}$. So, it follows that the restriction mapping
 $x^*\mapsto x^*|_{\overline{\psi(C)\cap X}}$ is an isometry of $\overline{\psi(C)\cap X^*}$ into $(\overline{\psi(C)\cap X})^*$. To complete the proof of the third property it remains to show that it is even onto. 
 
 To this end denote $Y=\overline{\psi(C)\cap X}$, $Z=\overline{\psi(C)\cap X^*}$ and let $j$ be the canonical isometric embedding of $Y$ into $X$. Then $j^*:X^*\to Y^*$ is the restriction mapping. Above we have proved that $j^*|_Z$ is an isometry, so it has a closed range. If it is not onto, Hahn-Banach theorem yields $y^{**}\in Y^{**}\setminus\{0\}$ such that $y^{**}|_{j^*(Z)}=0$.
Set $x^{**}=j^{**}y^{**}$. Since $j^{**}$ is again an isometric embedding, $x^{**}\ne0$. Moreover,
$x^{**}\in\wscl{Y}$ (as $\wscl Y$ is the range of $j^{**}$).  Further, for any $x^*\in Z$ we have
$$x^{**}(x^*)=j^{**}x^{**}(x^*)=x^{**}(j^*x^*)=0,$$
thus $x^{**}\in Z^\perp$. So,
$$0\ne x^{**}\in\wscl{Y}\cap Z^\perp=\wscl{\psi(C)\cap X}\cap (\psi(C)\cap X^*)^\perp
\supset \wscl{\psi(C)\cap X}\cap (\Phi(\psi(C)\cap X))^\perp,$$
a contradiction with the properties of projectional generator.

Finally observe that $\Phi(X)\subset \bigcup \{\psi(C)\cap X^*\setsep C\in[X]^{\le\omega}\}$, thus
$$
\left(\bigcup \{\psi(C)\cap X^*\setsep C\in[X]^{\le\omega}\}\right)^\perp\subset \Phi(X)^\perp
=\Phi(X)^\perp\cap \wscl X=\{0\}.$$
Since $\bigcup \{\psi(C)\cap X^*\setsep C\in[X]^{\le\omega}\}$ is $\qe$-linear, Hahn-Banach theorem completes the proof.

(3)$\Rightarrow$(4) It is enough to set $G(C)=\psi(C)\cap X^*$. 

(4)$\Rightarrow$(5) Let $G$ be the mapping provided by (4). The index set for the skeleton will be
$$\Gamma=\{C\in [X]^{\le\omega}\setsep x^*\mapsto x^*|_{\clin C}\mbox{ is an isometry on } \overline{G(C)}\}.$$
Let us show that $\Gamma$ is a cofinal $\sigma$-closed subset of $[X]^{\le\omega}$. 

Fix an increasing sequence $(C_n)$ in $\Gamma$ and set $C=\bigcup_n C_n$. Then $C\in\Gamma$ as well. Indeed, let $x^*\in G(C)$. Then there is some $n\in\en$ with $x^*\in G(C_n)$. It follows
that
$$\norm{x^*|_{\clin C}}\le\norm{x^*}=\norm{x^*|_{\clin C_n}}\le\norm{x^*|_{\clin C}},$$
so $\norm{x^*}=\norm{x^*|_{\clin C}}$. Passing to the closure shows that $C\in\Gamma$.

Further, let $C\in[X]^{\le\omega}$ be arbitrary. Let $\eta$ denote the mapping used in the proof of (2)$\Rightarrow$(3). Then we set $C_0=C$ and define by induction $C_n=C_{n-1}\cup\eta(G(C_{n-1}))$
for $n\in\en$. Finally, set $B=\bigcup_n C_n$. Then $B\in[X]^{\le\omega}$, $B\supset C$ and $\eta(G(B))\subset B$, so clearly $B\in\Gamma$.

Having the index set, we will construct the projections. Fix any $C\in\Gamma$. Let $Y=\clin C$ and $Z=\overline{G(C)}$. Then the mapping $x^*\mapsto X^*|_Y$ is an isometry of $Z$ onto $Y^*$. 
Let $j:Y\to X$ and $\iota:Z\to X^*$ be the canonical isometric inclusions. Since $j^*:X^*\to Y^*$ is the restriction mapping, we get that $j^*\circ\iota$ is an isometry of $Z$ onto $Y^*$. It follows that $\iota^*\circ j^{**}=(j^*\circ\iota)^*$ is an isometry of $Y^{**}$ onto $Z^*$. Further, $j^{**}$ is an isometric inclusion of $Y^{**}$ into $X^{**}$ with range $\wscl Y$.
So, $\iota^*$ maps $\wscl Y$ isometrically onto $Z^*$. Since $\iota^*(x^{**})=x^{**}|_Z$, Lemma~\ref{L:jednaprojekce} shows that there is a bounded linear projection $P_C$ on $X^*$ such that $P_CX^*=Z$ and $P_C^*X^{**}=\wscl Y$ (in fact, $P_C$ has norm one by the respective proof).

Let us show that $(P_C)_{C\in\Gamma}$ is a projectional skeleton on $X^*$. We already know that each $P_C$ is a bounded linear projection. Since $P_C X^*=\overline{G(C)}$ for each $C\in\Gamma$,
it is separable, so the property (i) is fulfilled. The property (iii) follows from the assumption that $G$ is $\omega$-monotone. Let us show the property (ii). Assume $C_1,C_2\in \Gamma$ are such that $C_1\subset C_2$. Then
$$P_{C_1}X^*=\overline{G(C_1)}\subset\overline{G(C_2)}=P_{C_2}X^*$$
and
$$P_{C_1}^*X^{**}=\clin^{w^*} C_1\subset\clin^{w^*} C_2=P_{C_2}^*X^{**},$$
so $P_{C_1}P_{C_2}=P_{C_2}P_{C_1}=P_{C_1}$. Finally, the property (iv) follows from the first property of $G$. Indeed, let $x^*\in X^*$. Then there are sequences $(C_n)$ in $[X]^{\le\omega}$ and $(x_n^*)$ in $X^*$ such that $x_n^*\in G(C_n)$ and $x_n^*\to x^*$. Let $C\in\Gamma$ be a set containing each $C_n$. then $x_n^*\in G(C)$ for each $n\in\en$, thus $x^*\in \overline{G(C)}=P_CX^*$.

Finally, since $P_C^* X^{**}=\clin^{w^*}C\supset C$, the induced subspace contains $X$.

(5)$\Rightarrow$(1) Let $(P_s)_{s\in\Gamma}$ be a projectional skeleton on $X^*$ such that the induced subspace contains $X$. Since $X$ is $1$-norming in $X^{**}$ we may assume without loss of generality that it is a $1$-projectional skeleton (by Lemma~\ref{L:norma projekci}). Let $Y$ be a separable subspace of $X$. Let $C\subset Y$ be a countable dense set. Then there is some $s\in\Gamma$ such that $P_sx=x$ for $x\in C$. It follows that $P_sx=x$ for $x\in Y$. 

Let $y^*\in Y^*$. Hahn-Banach theorem yields $x^*\in X^*$ with $x^*|_Y=y^*$. Moreover, for any $y\in Y$ we have
$$P_sx^*(y)=x^*(P_s^* y)=x^*(y)=y^*(y).$$
It follows that the mapping $x^*\mapsto x^*|_Y$ maps $P_sX^*$ onto $Y^*$. Since $P_sX^*$ is separable, we infer that $Y^*$ is separable as well.
\end{proof}

\begin{remark}\label{remark-Asplund}
(1) Let us stress that the characterizing property of Asplund spaces is not the existence of a $1$-projectional skeleton on the dual space, but the existence of such a skeleton whose induced subspace contains the original space (canonically embedded in the bidual). Indeed, for example $\C(K)^*$ is $1$-Plichko for any compact space $K$ (see, e.g., \cite[Example 4.10(a)]{ja-bidual} or \cite[Theorem 5.5]{ja-val-exa}), but not every $\C(K)$ space is Asplund. More generally, dual to any $C^*$-algebra is $1$-Plichko by \cite[Corollary 1.3]{BHK-vN} (for further generalizations see \cite{BHK-JBW,triples}).

(2) Let $X$ be an Asplund space. By the preceding theorem we know that there is a projectional skeleton on $X^*$ such that the induced subspace contains $X$. We point out that the induced subspace is not equal to $X$ (unless $X$ is reflexive), it is larger and equal to
$$\begin{aligned}
D(X)&=\{x^{**}\in X^{**}\setsep \exists C\subset X\mbox{ countable }:x^{**}\in\wscl C\}
\\&=\{x^{**}\in  X^{**}\setsep \exists (x_n)\mbox{ a sequence in }X: x_n\xrightarrow{w^*} x^{**}\}.\end{aligned}$$
This follows easily from the topological properties of induced subspaces.

(3) The projectional skeleton from the preceding theorem need not be commutative. The class of those spaces $X$ such that $X$ is contained in a $\Sigma$-subspace of $X^{**}$ (thus such that $D(X)$ is a $\Sigma$-subspace) is called \emph{class (T)} in \cite{ja-bidual}. By the above theorem the class (T) is a subclass of Asplund spaces (see also \cite[Theorem 4.1]{ja-bidual}). It contains many Asplund spaces (cf. \cite[Theorem 4.4, its corollaries and Example 4.8]{ja-bidual}), but not all Asplund spaces (by \cite[Example 4.10(b)]{ja-bidual}). 

There are Asplund spaces which do not belong to the class (T) but simultaneously their duals are $1$-Plichko. Indeed, if $K$ is any scattered compact space, then $\C(K)$ is Asplund and, moreover, $\C(K)^*$ is canonically isometric to $\ell^1(K)$ which is $1$-Plichko. If $K$ is uncountable, then there are many $1$-norming $\Sigma$-subspaces of $\ell^1(K)^*=\ell^\infty(K)$ (see \cite[Example 6.9]{ja-survey}), but it may happen that none of them contains $\C(K)$ (this takes place for example if $K=[0,\omega_2]$, see \cite[Example 4.10(b)]{ja-bidual} and its proof).
\end{remark}

We do not know of any nontrivial characterization of the class (T). However, there is a smaller subclass having nice characterizations. They are collected in the following theorem.

\begin{thm}\label{T:AsplundWCG} Let $X$ be a Banach space. The following assertions are equivalent.
\begin{enumerate}
\item $X$ is simultaneously Asplund and weakly Lindel\"of determined.
\item There is a norm-dense subset $M\subset X^*$ and an $\omega$-monotone mapping $\varphi:[X\cup M]^{\le\omega}\to[X\cup M]^{\le\omega}$ such that for each $A\in[X\cup M]^{\omega}$ we have
\begin{itemize}
\item $A\subset\varphi(A)$,
\item both $\overline{\varphi(A)\cap M}$ and $\overline{\varphi(A)\cap X}$ are linear subspaces,
\item $\wscl{\varphi(A)\cap M}=\overline{\varphi(A)\cap M}$,
\item the mapping $x^*\mapsto x^*|_{\overline{\varphi(A)\cap X}}$ is a bijection of $\overline{\varphi(A)\cap M}$ onto $(\overline{\varphi(A)\cap X})^*$
\end{itemize}
\item There is a shrinking projectional skeleton $(P_s)_{s\in\Gamma}$ on $X$ (i.e., such that $(P_s^*)_{s\in\Gamma}$ is a projectional skeleton on $X^*$). Moreover, this skeleton may be taken to be commutative.
\item There is a projectional skeleton $(Q_s)_{s\in\Gamma}$ on $X^*$ such that $Q_s^*(X)\subset X$ for each $s\in \Gamma$. Moreover, this skeleton may be taken to be commutative.
\item There is a shrinking Markushevich basis $(x_\alpha,x_\alpha^*)_{\alpha\in\Lambda}$ on $X$
(i.e., such that $\clin\{x_\alpha^*\setsep \alpha\in\Lambda\}=X^*$).
\item There is a Markushevich basis $(x_\alpha^*,x_\alpha^{**})_{\alpha\in\Lambda}$ on $X^*$ such that $x_\alpha^{**}\in X$ for each $\alpha\in\Lambda$.
\item $X$ is simultaneously Asplund and weakly compactly generated.
\end{enumerate}
\end{thm}

Several equivalences from this theorem are already known. The equivalence (1), (5) and (7) is contained in \cite[Theorem 8.3.3 and the following remark]{fabian-kniha}. The equivalence (1)$\Leftrightarrow$(3) is proved in \cite[Theorem 15]{CuFa-AsplundWCG}. The added value of the present theorem is mainly the assertion (2) and the (almost cyclic) way of proving. 

\begin{proof}[Proof of Theorem~\ref{T:AsplundWCG}.] (1)$\Rightarrow$(2)
Assume that $X$ is simultaneously Asplund and WLD. Since $X$ is WLD, $X^*$ is a $\Sigma$-subspace of itself, so there is a (commutative) projectional skeleton on $X$ with induced subspace $X^*$ (by Theorem~\ref{T:char sigma subspace}). Let $\psi:[X\cup X^*]^{\le\omega}\to[X\cup X^*]^{\le\omega}$ be an $\omega$-monotone mapping with properties from Theorem~\ref{T:char induced}(3).
Further, let $\theta:[X]^{\le\omega}\to[X\cup X^*]^{\le\omega}$ be an $\omega$-monotone mapping with properties from Theorem~\ref{T:asplund}(3). 

We will modify these two mappings. First, let 
$$M=\bigcup\{\theta(C)\cap X^*\setsep C\in[X]^{\le\omega}\}.$$ It follows from the properties of $\theta$ that $M$ is a norm-dense subset of $X^*$. For each $x^*\in M$ fix some $C(x^*)\in[X]^{\le\omega}$ with $x^*\in\theta(C(x^*))$.
Define a mapping $\tilde{\theta}:[X\cup M]^{\le\omega}\to[X\cup M]^{\le\omega}$ by
$$\tilde{\theta}(A)=\theta\left(A\cap X\cup \bigcup\{C(x^*)\setsep x^*\in A\cap M\}\right),\qquad A\in[X\cup M]^{\le\omega}.$$
It is clear that $\tilde{\theta}$ is $\omega$-monotone and $\tilde{\theta}(A)\supset A$ for each $A\in[X\cup M]^{\le\omega}$. Further, since $\tilde{\theta}(A)=\theta(C)$ for some $C$, $\tilde{\theta}$ has the obvious analogues of the properties of $\theta$.

Let us continue by modifying $\psi$. Since $M$ is norm-dense in $X^*$, there is a mapping
$\zeta:X^*\to[M]^{\le\omega}$ such that $x^*\in\overline{\zeta(x^*)}$ for each $x^*\in X^*$. Moreover, we can assume that $\zeta(x^*)=\{x^*\}$ for $x^*\in M$. For $A\in[X\cup M]^{\omega}$
set $\psi_0(A)=A$ and define by induction
$$\psi_n(A)=  \psi(\psi_{n-1}(A))\cap X \cup\zeta(\psi(\psi_{n-1}(A))\cap X^*)\mbox{\quad for } n\in\en.$$
It is clear that $\psi_n$ is an $\omega$-monotone mapping of $[X\cup M]^{\le\omega}$ to $[X\cup M]^{\le\omega}$ and that $\psi_n(A)\supset\psi_{n-1}(A)$ for each $A\in[X\cup M]^{\le\omega}$. Hence also the mapping
$$\tilde{\psi}(A)=\bigcup_n\psi_n(A),\qquad A\in[X\cup M]^{\le\omega}$$
is $\omega$-monotone.

Finally, we will define the sought mapping $\varphi$. For $A\in[X\cup M]^{\le\omega}$ set
$\varphi_0(A)=A$ and define by induction
$$\varphi_n(A)=\tilde{\psi}(\tilde{\theta}(A)),\qquad n\in\en,$$
and $\varphi(A)=\bigcup_n\varphi_n(A)$.

It is clear each $\varphi_n(A)$ is $\omega$-monotone and $\varphi_n(A)\supset\varphi_{n-1}(A)$ for each $n$, thus also $\varphi$ is $\omega$-monotone. Moreover, for any $A\in[M\cup X]^{\le\omega}$ the following properties hold.
\begin{itemize}
\item $A\subset\varphi(A)$. 
\item Both $\overline{\varphi(A)\cap X}$ and $\overline{\varphi(A)\cap M}$ are linear subspaces.
\item $x^*\mapsto x^*|_{\overline{\varphi(A)\cap X}}$ is an isometry of $\overline{\varphi(A)\cap M}$ onto $(\overline{\varphi(A)\cap X})^*$. This follows from the properties of $\theta$ as $\varphi(A)=\tilde{\theta}(\varphi(A))=\theta(C)$ for some $C$.
\item $x^*\mapsto x^*|_{\overline{\varphi(A)\cap X}}$ is a bijection of $\wscl{\varphi(A)\cap M}$ onto $(\overline{\varphi(A)\cap X})^*$. Indeed, $\varphi(A)=\tilde{\psi}(\varphi(A))$ and for any $C$ we have
$$\tilde{\psi}(C)\cap X=\psi(\tilde{\psi}(C))\cap X,\quad
\overline{\tilde{\psi}(C)\cap M}=\overline{\psi(\tilde{\psi}(C))\cap M}.$$
\end{itemize}
So, in particular $\overline{\varphi(A)\cap M}=\wscl{\varphi(A)\cap M}$ and the proof is complete.

(2)$\Rightarrow$(3) Let $\varphi$ be the mapping from (2). For the index set take $[X\cup M]^{\le\omega}$. By Lemma~\ref{L:jednaprojekce} for each $A\in\Gamma$ there is a bounded linear projection $P_A$ on $X$ such that $P_AX=\overline{\varphi(A)\cap X}$ and $P_A^*X^*=\wscl{\varphi(A)\cap M}=\overline{{\varphi(A)\cap M}}$. Now, as in the proof of Theorem~\ref{T:char induced}(3)$\Rightarrow$(1) we see that $(P_A)_{A\in\Gamma}$ is a projectional skeleton on $X$ with induced subspace $X^*$. Moreover, by the $\omega$-monotonicity of $\varphi$ together with the coincidence of the weak$^*$ and norm closures of $\varphi(A)\cap M$ we infer that $(P_A^*)_{A\in\Gamma}$ is a projectional skeleton on $X^*$.

(3)$\Rightarrow$(4) Assume that $(P_\alpha)_{\alpha\in\Lambda}$ is a shrinking projectional skeleton on $X$. Then $(P_\alpha^*)_{\alpha\in\Lambda}$ is a projectional skeleton on $X^*$. Moreover, $P_\alpha^{**}(X)\subset X$ as the restriction of $P_\alpha^{**}$ to $X$ equals $P_\alpha$.

(4)$\Rightarrow$(3) Assume $(Q_\alpha)_{\alpha\in\Lambda}$ is projectional skeleton on $X^*$ such that $Q_\alpha^*X\subset X$ for each $\alpha$. By \cite[Proposition 9]{kubis-skeleton} we can without loss of generality assume that the projections $Q_\alpha$ are uniformly bounded, say by a constant $C\ge1$.

Set $P_\alpha=Q_\alpha^*|_X$. We claim that $(P_\alpha)_{\alpha\in\Lambda}$ is a projectional skeleton on $X$ and $P_\alpha^*=Q_\alpha$ for each $\alpha\in\Lambda$.
It is clear that each $P_\alpha$ is a projection such that $\norm{P_\alpha}\le C$. Moreover,
obviously $P_\alpha^*=Q_\alpha$. It remains to check the properties (i)--(iv) of  projectional skeletons. Firstly, $Q_\alpha X^*=P_\alpha^*X^*$ is isomorphic to the dual of $P_\alpha X$.
So, $P_\alpha X$ is separable (as its dual is), which proves the property (i). The property (ii) is obvious. To prove the property (iii) let $(\alpha_n)$ be an increasing sequence in $\Lambda$. Since $(Q_\alpha)_{\alpha\in\Lambda}$ is a projectional skeleton, there is $\alpha=\sup_n \alpha_n\in\Lambda$ and, moreover, $Q_{\alpha_n}x^*\to Q_\alpha x^*$ for each $x^*\in X^*$ (as the skeleton $(Q_\alpha)_{\alpha\in\Lambda}$ satisfies the property (iii')). Now it easily follows that $Q_\alpha^* x^{**}\overset{w^*}{\longrightarrow} Q_\alpha^* x^{**}$ for each $x^{**}\in X$.
Since the restriction of the weak$^*$ topology on $X^{**}$ to $X$ coincides with the weak topology of $X$, we deduce $P_{\alpha_n} x\overset{w}{\longrightarrow}P_\alpha x$ for each $x\in X$, so
$$P_\alpha X=\overline{\bigcup_n P_{\alpha_n}X}^{w}=\overline{\bigcup_n P_{\alpha_n}X}$$
as the union is a linear subspace. This completes the proof of the property (iii). 

To prove the property (iv) set
$$P x=\lim_{\alpha\in\Lambda} P_\alpha x,\qquad x\in X.$$
Then $P$ is a well-defined projection on $X$ with $\norm{P}\le C$ (cf. \eqref{eq:P_A def}). We claim that $P$ is the identity of $X$. If not, then $\ker P\ne\{0\}$, so there is some $x\in X\setminus\{0\}$ with $Px=0$. It means that $P_\alpha x=0$ for each $\alpha\in X$ (see Lemma~\ref{L:P_sP_A}(a)). Thus
$$x\in \bigcap_{\alpha\in\Lambda}\ker P_\alpha =  \bigcap_{\alpha\in\Lambda} (P_\alpha^* X^*)_\perp =\left(\bigcup_{\alpha\in\alpha} Q_\alpha X^*\right)_{\perp}=(X^*)_\perp=\{0\},$$
a contradiction. So, $P$ is the identity mapping and now the property (iv) easily follows from the property (iii).

(3)$\Rightarrow$(5)\&(7) Assuming (3) we will show that there is a Markushevich basis $(x_\alpha,x^*_\alpha)_{\alpha\in\Lambda}$ in $X$ which is shrinking, and, moreover,
the set $\{x_\alpha\setsep \alpha\in\Lambda\}\cup\{0\}$ is $\sigma$-compact in the weak topology.

The proof will be done by transfinite induction on the density character of $X$. First assume that $X$ is separable. Let $(P_s)_{s\in \Gamma}$ be a shrinking projectional skeleton on $X$. It follows from the properties of the skeleton that there is some $s\in\Gamma$ such that $P_s$ is the identity on $X$. Thus $P_s^*$ is the identity on $X^*$. Since the adjoint projections form a projectional skeleton on $X^*$, they have separable ranges. So, $X^*$ is separable. Now, it is a classical result that any space with a separable dual admits a shrinking Markushevich basis (see, e.g., \cite[Theorem 1.22]{HMVZ-biortogonal}). Moreover, the basis is countable, so the weak $\sigma$-compactness trivially follows.

Further, assume that $\kappa$ is an uncountable cardinal such that the implication 
holds for any space of density strictly less than $\kappa$. Let $X$ be a Banach space of density character $\kappa$ having a shrinking projectional skeleton $(P_s)_{s\in\Gamma}$. Since the induced subspace equals whole $X^*$, we can assume that it is a commutative $1$-projectional skeleton (by Theorem~\ref{T:komutativni} and Lemma~\ref{L:norma projekci}). Let $(A_\alpha)_{\alpha\le\kappa}$ be a transfinite sequence of subsets of $\Gamma$ provided by Lemma~\ref{L:A_alpha}. By Proposition~\ref{P:PRI} the transfinite sequence $(P_{A_\alpha})_{\alpha\le\kappa}$ is a PRI on $X$. 

Further, denote $Q_s=P_s^*$ for $s\in\Gamma$. By the assumptions $(Q_s)_{s\in\Gamma}$ is a projectional skeleton on $X^*$ (in fact a commutative $1$-projectional skeleton by the above).
So, for any directed set $A\subset\Gamma$ we can define the projection $Q_A$ on $X^*$ by the formula \eqref{eq:P_A def}. We claim that $Q_A=P_A^*$. Indeed, given $x^*\in X^*$ and $x\in X$ we have
$$P_A^* x^* (x)=x^*(P_Ax)=\lim_{s\in A} x^*(P_sx)=\lim_{s\in A}P_s^*x^*(x)=\lim_{s\in A}Q_sx^*(x)=Q_Ax^*(x).$$

It follows that $(P_{A_\alpha}^*)_{\alpha\le\kappa}=(Q_{A_\alpha})_{\alpha\le\kappa}$ is a PRI on $X^*$.

Now fix $\alpha<\kappa$ and let $P=P_{\alpha+1}-P_\alpha$. Since the skeleton $(P_s)_{s\in\Gamma}$ is commutative, Lemma~\ref{L:P_sP_A} yields $P_sP=PP_s$ for each $s\in \Gamma$. In particular, $(P_s|_{PX})_{s\in \Gamma}$ is a projectional skeleton on $PX$. We will show it is shrinking.
To this end first observe that
$$(P_s|_{PX})^*y^*=P_s^*(y^*\circ P)|_{PX},\qquad y^*\in(PX)^*, s\in\Gamma.$$
Indeed, fix $s\in\Gamma$, $y^*\in(PX)^*$ and $x\in PX$. Then
$$(P_s|_{PX})^*y^*(x)=y^*(P_s|_{PX}x)=y^*(P_sx)=y^*(P_sPx)=y^*(PP_sx)=y^*\circ P(P_sx)=
P_s^*(y^*\circ P) x.$$
So, if we define $T:(PX)^*\to X^*$ by $T(y^*)=y^*\circ P$ and $S:X^*\to (PX)^*$ by $Sx^*=x^*|_{PX}$,
then $(P_s|_{PX})^*=SP_s^*T$. We know that $(P_s|_{PX})^*$ is a projection on $(PX)^*$ and that $(P_s^*)_{s\in\Gamma}$ is a projectional skeleton on $X^*$. Let us prove the properties (i)-(iv) of projectional skeletons. The property (i) follows from the fact that $P_s^*$ has separable range,
the property (ii) is obvious. The property (iii) is easy. Let us show the property (iv). Fix any $y^*\in (PX)^*$. Then $Ty^*\in X^*$, so there is some $s\in\Gamma$ with $P_s^*Ty^*=Ty^*$. Then 
$SP_s^*Ty^*=STy^*=y^*\circ P|_{PX}=y^*$.

So, for each $\alpha<\kappa$ the space $(P_{\alpha+1}-P_\alpha)X$ admits a shrinking projectional skeleton, so by the induction hypothesis it admits a shrinking Markushevich basis $(x_{\alpha,j},x_{\alpha,j}^*)_{j\in J_\alpha}$ such that the set $\{x_{\alpha,j}\setsep j\in J_\alpha\}\cup\{0\}$ is $\sigma$-compact in the weak topology. Moreover, we can assume that $\norm{x_{\alpha,j}}\le1$ for each $j$ and $\alpha$.

We claim that 
$$(x_{\alpha,j},x_{\alpha,j}^*\circ(P_{\alpha+1}-P_\alpha))_{\alpha<\kappa,j\in J_\alpha}$$
is a Markushevich basis with the required properties. This follows from the proof of \cite[Proposition 6.2.5]{fabian-kniha}. Indeed, in the assertion (ii) of the quoted proposition is proved that it is a shrinking Markushevich basis. Moreover, in the assertion (i) it is proved that
$\{0\}\cup\bigcup_{\alpha<\kappa}K_\alpha$ is weakly compact whenever $K_\alpha\subset B_X\cap (P_{\alpha+1}-P_\alpha)X$ is weakly compact for each $\alpha<\kappa$.

(5)$\Rightarrow$(6) Let $(x_\alpha,x_\alpha^*)_{\alpha\in\Lambda}$ be a shrinking Markushevich basis on $X$. Then clearly $(x_\alpha^*,x_\alpha)_{\alpha\in\Lambda}$ is a Markushevich basis on $X^*$.

(6)$\Rightarrow$(5) Let $(x_\alpha^*,x_\alpha)_{\alpha\in\Lambda}$ be a Markushevich basis on $X^*$, where $x_\alpha\in X$ for each $\alpha\in\Lambda$. Then $(x_\alpha,x_\alpha^*)_{\alpha\in\Lambda}$ is a shrinking Markushevich basis on $X$. The only property which requires a proof is that $\clin\{x_\alpha\setsep \alpha\in\Lambda\}=X$. But it is an immediate consequence of the bipolar theorem, as 
$$\clin\{x_\alpha\setsep \alpha\in\Lambda\}=(\{x_\alpha\setsep \alpha\in\Lambda\}^\perp)_\perp=\{0\}_\perp=X.$$

(5)$\Rightarrow$(1) Let $(x_\alpha,x_\alpha^*)_{\alpha\in\Lambda}$ be a shrinking Markushevich basis on $X$. Since
$$\{x^*\in X^*\setsep \{\alpha\in\Lambda\setsep x^*(x_\alpha)\ne0 \}\mbox{ is countable}\}$$
is a closed linear subspace of $X^*$ containing each $x_\alpha^*$, it is equal to $X^*$. So, $X^*$ is a $\Sigma$-subspace of itself, so $X$ is WLD.

Similarly, $X$ is contained in a $\Sigma$-subspace of $X^{**}$, so $X$ belongs to the class (T).
Hence $X$ is Asplund by \cite[Theorem 4.1]{ja-bidual}
(or by Theorems~\ref{T:asplund} and~\ref{T:char sigma subspace}).

(7)$\Rightarrow$(1) It follows from the deep results of \cite{Amir-Lindenstrauss} that any WCG space is WLD.

\smallskip

It remains to observe that the projectional skeleton in (3) or (4) can be chosen to be commutative.

Case (3): Assume that $(P_s)_{s\in\Gamma}$ is a shrinking projectional skeleton on $X$. Then the induced subspace  is whole $X^*$, so the skeleton can be taken to be commutative (up to restricting to a $\sigma$-closed cofinal subset of the index set by Theorem~\ref{T:komutativni}.

Case (4): Assume that $(Q_s)_{s\in\Gamma}$ is a projectional skeleton on $X^*$ such that $Q_s^*(X)\subset X$ for each $s\in\Gamma$. By the proof of (4)$\Rightarrow$(3) there is a $\sigma$-closed cofinal subset $\Gamma'\subset\Gamma$ and a projectional skeleton $(P_s)_{s\in\Gamma'}$ on $X$ such that $P_s^*=Q_s$ for $s\in\Gamma'$. So, by case (3) we can assume that the skeleton $(P_s)$ is commutative, then $(Q_s)$ is commutative as well.
\end{proof}

If $X$ is an Asplund space, we know that $X^*$ admits a projectional skeleton (with induced subspace $D(X)$), hence there exists a Markushevich basis of $X^*$ (by Theorem~\ref{T:M-basis}). However, it is not clear whether it may have nice properties related to $D(X)$. In the following example we show that at least sometimes it is true, but the reason has nothing to do with the specific projectional skeleton.

\begin{example2}\label{ex:scattered} Let $K$ be a scattered locally compact space. Then $X=\C_0(K)$ is an Asplund space
and $X^*$ can be canonically identified with $\ell^1(K)$.
 Consider the canonical Markushevich basis of $\ell^1(K)$, i.e., $(e_x,e_x^*)_{x\in K}$, where $e_x$ and $e_x^*$ are the canonical basic vectors in $\ell^1(K)$ and $\ell^\infty(K)$, respectively. Regardless of the concrete topological structure of $K$ the set
$$H=\{e_x\setsep x\in K\}\cup\{0\}$$
is $\sigma(X^*,D(X))$-closed in $X^*$. Indeed, it is even weak$^*$-closed -- if $K$ is even compact, then $(H\setminus\{0\},w^*)$ is homeomorphic to $K$, if $K$ is not compact, then $(H,w^*)$ is homeomorphic to the one-point compactification of $K$.

Moreover, for the Markushevich basis $(e_x,e_x^*)_{x\in K}$ defined above the following holds.
\begin{itemize}
\item $e_x^*\in X$ for each $x\in K$ if and only if $K$ is discrete, i.e., if $X=c_0(K)$.
\item $e_x^*\in D(X)$ for each $x\in K$ if and only if each point of $K$ is $G_\delta$, i.e., if $K$ is locally countable.
\end{itemize}
\end{example2}

In some cases there is a better Markushevich basis than the one described in the previous example.
Some concrete cases are described in the following examples.

\begin{example2}
Let $K$ be a countable compact space. Then $\C(K)$ admits a (countable) shrinking Markushevich basis (as $\C(K)^*$ is separable), but it must be different from the basis from the previous example unless $K$ is finite. An explicit formula can be given as follows. Firstly, $K$ is homeomorphic to the ordinal segment $[0,\eta]$ for some $\eta<\omega_1$ (we assume $K$ is infinite, so $\eta\ge\omega$). Fix a bijection $\xi:[0,\omega)\to[0,\eta]$ such that $\xi(0)=0$. We define a Markushevich bases $(f_n,\mu_n)_{n<\omega}$ in  $\C(K)$ as follows:
$$\begin{aligned}
f_0&=\chi_{\{\xi(0)\}}=\chi_{\{0\}}, \mu_0=\delta_{\xi(0)}=\delta_0,\\
f_n&=\chi_{(\max[0,\xi(n))\cap\xi(\{0,\dots,n-1\}),\xi(n)]} \mbox{ if }n\ge1,\\
\mu_n&=\begin{cases}
\delta_{\xi(n)} &\mbox{ if }\xi(\{0,\dots,n-1\})\cap(\xi(n),\eta]=\emptyset,\\
\delta_{\xi(n)}-\delta_{\min\xi(\{0,\dots,n-1\})\cap(\xi(n),\eta]} & \mbox{ otherwise}.
\end{cases}
\end{aligned}$$
It is easy to check that $(f_n,\mu_n)_{n<\omega}$ is a shrinking Markushevich basis of $\C(K)$.
\end{example2}

\begin{example2}
If $K$ is a compact scattered space of cardinality $\aleph_1$, then there is a Markushevich basis on $X^*=\C(K)^*$ such that the biorthogonal functionals belong to $D(X)$. Unless $K$ is countable, the basis from Example~\ref{ex:scattered} fails this property. In case $K=[0,\eta]$ where $\eta\in[\omega_1,\omega_2)$, an explicit formula can be given as follows. Fix a bijection $\xi:[0,\omega_1)\to[0,\eta]$ such that $\xi(0)=0$.  We define a Markushevich bases $(\mu_\alpha,f_\alpha)_{\alpha<\omega_1}$ in  $\C(K)^*$ as follows:
$$\begin{aligned}
f_0&=\chi_{\{\xi(0)\}}=\chi_{\{0\}}, \mu_0=\delta_{\xi(0)}=\delta_0,\\
f_\alpha&=\chi_{[\sup\{\xi(\beta)+1\setsep\beta<\alpha\ \&\ \xi(\beta)<\xi(\alpha)\},\xi(\alpha)]} \mbox{ if }\alpha\ge1,\\
\mu_n&=\begin{cases}
\delta_{\xi(\alpha)} &\mbox{ if }\xi([0,\alpha))\cap(\xi(\alpha),\eta]=\emptyset,\\
\delta_{\xi(\alpha)}-\delta_{\min(\xi([0,\alpha))\cap(\xi(n),\eta])} & \mbox{ otherwise}.
\end{cases}
\end{aligned}$$
It is easy to check that $(\mu_\alpha,f_\alpha)_{\alpha<\omega_1}$ is a Markushevich basis of $\C(K)^{*}$. Moreover, $f_\alpha\in D(\C(K))$ for each $\alpha<\omega_1$ as each $f_\alpha$ is a function of the first Baire class, being the characteristic function of a closed $G_\delta$ set.
\end{example2}

\section{Open problems}

In this section we collect several questions on projectional skeletons, projectional generators, Markushevich bases and related topics which remain open.

We start by the following question on a possible generalization of Corollary~\ref{C:wld}.

\begin{ques} Assume that $X$ is a Banach space and $D\subset X^*$ is a subspace induced by a projectional skeleton on $X$ which is of finite codimension in $X^*$. Is $D$ necessarily a $\Sigma$-subspace?
\end{ques}

Banach spaces whose duals admit a $\Sigma$-subspace of finite codimension have been studied in \cite{ja-almostwld}. It is not clear whether there are non-commutative variants. It is easy to observe that they cannot be found among continuous functions on ordinals or on trees. A natural candidate could be a dual to a quasireflexive space. Indeed, let $X$ be quasireflexive. Then $X$ is Asplund, so there is a projectional skeleton on $X^*$ such that the induced subspace contains $X$. Since $X$ is of finite codimension in $X^{**}$, a fortiori $D(X)$ is of finite codimension in $X^{**}$. However, any quasireflexive space is weakly compactly generated by \cite{valdivia-qr}, so it belongs to the class (T). In fact, since the dual of a quasireflexive space is again quasireflexive, hence weakly compactly generated, necessarily $D(X)=X^{**}$. But it seems that the following problem is open.

\begin{ques} Let $X$ be an Asplund space such that $D(X)$ has finite codimension in $X^{**}$. Does $X$ belong to the class (T)?
\end{ques}


A further question is connected to the existence of a nice Markushevich basis. We know that any Banach space admitting a projectional skeleton has a Markushevich basis (by Theorem~\ref{T:M-basis}). However, the following natural question is open.

\begin{ques}\label{Q:M-basis uzavrena?} Assume that a Banach space $X$ admits a projectional skeleton $(P_s)_{s\in\Gamma}$ and $D\subset X^*$ is the induced subspace. Does there exist a Markushevich basis $(x_\alpha,x_\alpha^*)_{\alpha\in\Lambda}$ such that the set $H=\{x_\alpha\setsep \alpha\in\Lambda\}\cup\{0\}$ satisfies 
\begin{itemize}
\item $H$ is $\sigma(X,D)$-closed in $X$; or
\item $P_s(H)\subset H$ for $s\in\Gamma'$ for some cofinal $\sigma$-closed $\Gamma'\subset\Gamma$; or at least
\item $(H,\sigma(X,D))$ is monotonically Sokolov?
\end{itemize}
\end{ques}

Observe that the positive answer to the first question implies the positive answer to the second one (by Lemma~\ref{L:invariance}(b)). Further, the positive answer to the second question implies the positive answer to the third one. This follows from the following lemma. 

\begin{lemma} Let $X$ be a Banach space with a projectional skeleton $(P_s)_{s\in\Gamma}$. Let $D$ denote the respective induced subspace. Let $H\subset X$. Assume that $P_s(H)\subset H$ for $s$ from some cofinal $\sigma$-closed $\Gamma'\subset\Gamma$. Then $(H,\sigma(X,D))$ is monotonically Sokolov.
\end{lemma}

\begin{proof}
Since $(P_s)_{s\in\Gamma'}$ is a projectional skeleton on $X$ with induced subspace $D$, without loss of generality we assume $\Gamma'=\Gamma$. Let $\A\mapsto(r_{\A},\N(\A))$ be the assignment constructed in the proof of the implication (1)$\Rightarrow$(5) of Theorem~\ref{T:char induced}.
Now, for $\B\in[\F(H)]^{\le\omega}$ set
$$\A(\B)=\{\overline{F}^{\sigma(X,D)}\setsep F\in\B\}\in[\F(X)]^{\le\omega}.$$
Clearly $\A$ is $\omega$-monotone. Next, for $\B\in\F(H)$ set
$$q_{\B}=r_{\A(\B)}|_H \mbox{ and }\M(\B)=\{N\cap H\setsep N\in\N(\A(\B))\}.$$
Clearly, the assignment $\M$ is $\omega$-monotone. Moreover, by the construction the mapping $r_{\A(\B)}$ is one of the projections $r_s$. So, $H$ is invariant for $r_{\A(\B)}$ and hence $q_\B$ is a well-defined continuous retraction on $(H,\sigma(X,D))$. Further, for each $B\in \B$ we have
$$q_\B(B)=r_{\A(\B)}(\overline B\cap H)\subset r_{\A(\B)}(\overline B)\cap r_{\A(\B)}(H)
\subset \overline{B}\cap H=B.$$
Finally, $\M(\B)$ is clearly an outer network for $q_\B(H)$. 
\end{proof}

The answer to the first question is positive in case the skeleton is commutative and it is witnessed by the Markushevich basis constructed using a PRI (see Theorem~\ref{T:char sigma subspace} and Remark~\ref{Remark Sigma}(d,e)).
The non-commutative case seems to be more complicated. The answer is positive for spaces of continuous functions on ordinals 
(by Proposition~\ref{P:ord Mbaze ocasni}), for continuous functions on trees (by Proposition~\ref{P:trees Mbaze ocasni})
and for duals to Asplund $C(K)$ spaces (by Example~\ref{ex:scattered}). Let us point out that the Markushevich basis witnessing
the positive answer is in all the cases in a sense `canonical', but it need not come from a PRI (see the comments after the quoted results). Moreover, a Markushevich basis constructed using a PRI may even fail all the properties (see Example~\ref{E:Mbaze divna}(b)).
The quoted example illustrates that the choice of a particular PRI does matter. So, it is natural to ask whether there is always a `nice'
PRI. In particular:

\begin{ques}
Let $\eta\ge\omega_2$ and $\kappa=\card\eta$. Is there a bijection $\xi:I(\kappa)\to I(\eta)$ such that the set $\{f_\alpha\setsep \alpha\in I(\kappa)\}\cup\{0\}$ (defined before Proposition~\ref{P:ord-Mbaze z PRI}) is $\sigma(\C([0,\eta]),D(\eta))$-closed?
\end{ques}

Note that if $\eta$ is a cardinal, $\xi$ can be the identity (by Example~\ref{E:Mbaze divna}(a)). In some further special cases it is not hard to construct respective $\xi$. But we do not know whether it is possible in general.

The following special case of Question~\ref{Q:M-basis uzavrena?} seems to be open as well.

\begin{ques}
Let $X$ be an Asplund space. Is there a Markushevich basis $(x_\alpha^*,x_\alpha^{**})_{\alpha\in\Lambda}$ of $X^{**}$ such that the set $\{0\}\cup\{x_\alpha^*\setsep \alpha\in\Lambda\}$ is $\sigma(X^*,D(X))$-closed (or even weak$^*$-closed)?
(Recall that $D(X)$ is the subspace of $X^{**}$ defined in Remark~\ref{remark-Asplund}(2).)
\end{ques}

Another question is whether the existence of a projectional skeleton is equivalent to the existence of a projectional generator.

\begin{ques}\label{Q:PG} Assume that a Banach space $X$ admits a projectional skeleton $(P_s)_{s\in\Gamma}$ and $D\subset X^*$ is the induced subspace.
\begin{itemize}
\item Is there a projectional generator on $X$ with domain $D$?
\item Given $x^*\in D$, fix $s(x^*)\in\Gamma$ such that $x^*=P_{s(x^*)}^*x^*$ and let $\Phi(x)\subset P_{s(x^*)}X$ be a countable dense subset. Is $(D,\Phi)$ a projectional generator?
\end{itemize}
\end{ques}

We point out that the answer to the first question is positive in the commutative case (by Theorem~\ref{T:char sigma subspace}),
for continuous functions on ordinals (by Propositions~\ref{P:ord Mbaze ocasni}(d)) and for continuous functions on trees
(by Proposition~\ref{P:trees PG}). Note that the quoted propositions do not answer the second question, as the projectional generators are constructed in a similar but a bit different manner than suggested.

If $X$ is an Asplund space, its dual admits a projectional generator with domain $X$ (see Theorem~\ref{T:asplund}).
However, the following question seems to be open.

\begin{ques}
Let $X$ be an Asplund space. Is there a projectional generator on $X^*$ with domain $D(X)$?
\end{ques}

Conversely, assume that $X$ is a Banach space which admits a projectional generator with domain $Y\subset X^*$.
It is not hard to construct then a projectional skeleton -- one possibility is to use
\cite[Proposition 7 and Theorem 15]{kubis-skeleton}. Another possibility is to use the method of the proof of \cite[Lemma 6.1.3 and Proposition 6.1.7]{fabian-kniha} to prove the validity of the condition (2) in Theorem~\ref{T:char induced} starting from a projectional generator. The space induced by the respective skeleton may be larger than $Y$ -- it is the smallest weak$^*$-countably closed subspace of $X^*$ containing $Y$, call it $D(Y)$. The following abstract question seems to be open as well.

\begin{ques}
Assume that $X$ is a Banach space which admits a projectional generator with domain $Y\subset X^*$. Does it admit a projectional generator with domain $D(Y)$? Can such a projectional generator be found as an extension of the original one?
\end{ques}

Another natural question is the following one.

\begin{ques}
Assume that $X$ is a Banach space which admits a projectional generator with domain $Y\subset X^*$. Is there a projectional generator $\Phi$ with domain $Y$ such that
$$\forall M\subset Y: \clin^{w^*} M\cap \Phi(M)^\perp=\{0\}\quad ?$$
\end{ques}

This stronger condition can be achieved for Plichko spaces (by the proof of Theorem~\ref{T:char sigma subspace}), for continuous functions on ordinals (by the proof of Proposition~\ref{P:ord Mbaze ocasni}(d)) and for continuous functions on trees (by the proof of Proposition~\ref{P:trees PG}). On the other hand, for duals of Asplund spaces the stronger condition was not proved. In the proof of Theorem~\ref{T:asplund} it was strongly used the assumption that $\overline{M}$ is linear.

Another interesting problems concern the relationship of primarily Lindel\"of spaces and monotonically Sokolov ones. Monotonically Sokolov spaces (more precisely their continuous images) can be viewed as a noncommutative version of primarily Lindel\"of ones. The first question is the following one.

\begin{ques}\label{qq1} Are monotonically Sokolov spaces stable under continuous images? \end{ques}

Note that primarily Lindel\"of spaces are stable to continuous images by the very definition, monotonically Sokolov spaces are stable to $\er$-quotient images by \cite[Theorem 3.4(g)]{RoHe-sokolov}. The stability to continuous images is not discussed in \cite{RoHe-sokolov}. We conjecture that the stability fails but we do not know any counterexample.

Assuming the answer is negative, the following question is natural.

\begin{ques} Assume that $T$ is simultaneously primarily Lindel\"of and monotonically Sokolov. Is $T$ an $\er$-quotient image of a closed subset of $(L_\Gamma)^\en$?
\end{ques}

Another question is inspired by the fact that primarily Lindel\"of spaces are defined by an explicit representation, while monotonically Sokolov
are defined by existence of a certain family of retractions. So, we can ask the following general question.

\begin{ques} Is it possible to characterize monotonically Sokolov space by an explicit representation (similar to that of primarily Lindel\"of spaces)?
\end{ques}

Note that this is related to a similar problem of the existence of an explicit representation of compact spaces with a retractional skeleton (similar
to that of Valdivia compacta) or of Banach spaces with a projectional skeleton (similar to that of Plichko spaces). It seems to be related also to the problem of a relationship of a Markushevich basis to the subspace induced by a projectional skeleton discussed above.

\subsection*{Added during the revision process} After submition of  this paper  Question~\ref{qq1} has been answered in the negative in \cite[Examples~3.1 and~3.2]{CSRH-new}. 

\section*{Acknowledgment}

The author is grateful to Mari\'an Fabian for many useful comments.

\def\cprime{$'$}

\end{document}